\documentclass[letterpaper,11pt,oneside,reqno]{amsart}

\usepackage[pdftex,backref=page,colorlinks=true,linkcolor=blue,citecolor=red]{hyperref}
\usepackage[alphabetic,nobysame]{amsrefs}

\usepackage{amsmath,amssymb,amsthm,amsfonts,mathtools}
\usepackage{graphicx,color}
\usepackage{upgreek}
\usepackage{esint}
\usepackage[mathscr]{euscript}

\allowdisplaybreaks
\numberwithin{equation}{section}

\usepackage{tikz}
\usetikzlibrary{shapes,arrows,positioning,decorations.markings}

\usepackage{ytableau}
\usepackage{array}
\usepackage{adjustbox}
\usepackage{cleveref}
\usepackage{enumerate}
\usepackage[margin=10pt]{caption}

\usepackage[DIV=12]{typearea}

\newcommand{\ssp}{\hspace{1pt}}
\renewcommand{\le}{\leqslant}
\renewcommand{\ge}{\geqslant}
\renewcommand{\Re}{\operatorname{Re}}

\newcommand{\inv}{\operatorname{inv}}
\newcommand{\Prob}{\operatorname{\mathbb{P}}}
\newcommand{\ls}{\mathsf{s}}
\newcommand{\bs}{\mathsf{S}}
\DeclareMathOperator{\PD}{PD}
\DeclareMathOperator{\RPD}{RPD}
\DeclareMathOperator{\BPD}{BPD}
\DeclareMathOperator{\RBPD}{RBPD}
\DeclareMathOperator{\cross}{cross}

\DeclareMathOperator{\NWbump}{NWbump}
\DeclareMathOperator{\emptytile}{empty}
\DeclareMathOperator{\wt}{wt}

\newtheorem{proposition}{Proposition}[section]
\newtheorem{lemma}[proposition]{Lemma}
\newtheorem{corollary}[proposition]{Corollary}
\newtheorem{theorem}[proposition]{Theorem}
\newtheorem{conjecture}[proposition]{Conjecture}

\theoremstyle{definition}
\newtheorem{definition}[proposition]{Definition}
\newtheorem{remark}[proposition]{Remark}

\begin{document}

\title[Grothendieck Shenanigans]{Grothendieck Shenanigans:\\
Permutons from pipe dreams via integrable probability}

\author{
	A. H. Morales,
	G. Panova,
	L. Petrov,
	D. Yeliussizov}

\date{}

\begin{abstract}
	We study random permutations corresponding to pipe dreams. Our
	main model is motivated by the Grothendieck
	polynomials with parameter $\beta=1$ arising in the
	$K$-theory of the flag variety. The probability weight of a
	permutation is proportional to the principal specialization
	(setting all variables to~$1$)
	of the
	corresponding
	Grothendieck polynomial. By mapping
	this random permutation to a version of TASEP (Totally
	Asymmetric Simple Exclusion Process), we describe the
	limiting permuton and fluctuations around it as the order
	$n$ of the permutation grows to infinity. The fluctuations
	are of order $n^{\frac13}$ and have the Tracy--Widom GUE distribution,
	which places this algebraic ($K$-theoretic) model
	into the Kardar--Parisi--Zhang universality class.
	As an application, we find the expected
	number of inversions in this random permutation, and
	contrast it with the case of
	non-reduced pipe dreams.

	Inspired by
	Stanley's question for the maximal value of principal
	specializations of Schubert polynomials, we resolve the
	analogous question for $\beta=1$ Grothendieck polynomials,
	and provide bounds for general~$\beta$. This analysis uses a correspondence with the free fermion six-vertex model, and the frozen boundary of the Aztec diamond.
\end{abstract}

% PLAIN TEXT ABSTRACT
% We study random permutations corresponding to pipe dreams. Our main model is motivated by the Grothendieck polynomials with parameter β=1 arising in the K-theory of the flag variety. The probability weight of a permutation is proportional to the principal specialization (setting all variables to 1) of the corresponding Grothendieck polynomial. By mapping this random permutation to a version of the Totally Asymmetric Simple Exclusion Process (TASEP), we describe the limiting permuton and fluctuations around it as the order n of the permutation grows to infinity. The fluctuations are of order n^(1/3) and have the Tracy–Widom GUE distribution, which places this algebraic (K-theoretic) model into the Kardar–Parisi–Zhang universality class. As an application, we find the expected number of inversions in this random permutation, and contrast it with the case of non-reduced pipe dreams.
%
% Inspired by Stanley’s question for the maximal value of principal specializations of Schubert polynomials, we resolve the analogous question for β=1 Grothendieck polynomials, and provide bounds for general β. This analysis uses a correspondence with the free fermion six-vertex model, and the frozen boundary of the Aztec diamond.

\maketitle

\section{Introduction}
\label{sec:intro}

\subsection{A story from Algebra to Probability}
\label{sub:overview}

Algebraic Combinatorics established itself as a field that uses combinatorial methods to understand algebraic behavior in problems ranging from Group Theory to Algebraic Geometry. It started with stark exact formulas, like the celebrated hook-length formula for the dimension of irreducible modules of the symmetric group $S_n$; beautiful interpretations, such as the Littlewood--Richardson rule for the structure constants of representations of the general linear group $GL_n$; powerful and intricate bijections, such as the Robinson--Schensted--Knuth correspondence. However, exact answers only go so far, leaving room for questions like  ``approximately how many'', ``what are the typical objects'',  and ``what is the typical behavior''. These questions lead us into the realm of Asymptotic Algebraic Combinatorics, which aims to answer them with the help of tools originating outside of Combinatorics. In the present work, we employ Integrable Probability, a rapidly evolving field focused on developing and analyzing interacting particle systems and random growth models possessing a certain degree of structure or symmetry. The arising probabilistic models exhibit rich structure leading to new permutons representing the typical permutations. The connection between the algebraic model and statistical mechanics is multi-fold via a correspondence between the so-called bumpless pipe dream models for Schubert/Grothendieck and the six-vertex model. The well studied free fermion six-vertex model and the frozen boundary of the Aztec diamond are key to understanding the maximal permutations.

On the algebro-combinatorial side, many significant questions concerning exact formulas and combinatorial interpretations in the cohomology and $K$-theory of the flag variety remain open. Understanding their asymptotic behavior is thus even more natural. Stanley \cite{stanley2017some} asked the most basic question on the  principal specializations of Schubert polynomials $\mathfrak{S}_w$
(representing cohomology classes of the flag variety): does the following limit exist
\begin{equation*}
	\lim_{n\to \infty} \frac{1}{n^2} \log_2 \max_{w \in S_n} \mathfrak{S}_w\bigl(
	\underbrace{1,\ldots,1 }_n\bigr),
\end{equation*}
and if so, what is it and for which permutations $w$ is this achieved.
This question (including the existence of the limit) is still open.
In~\cite{MoralesPakPanova2019}, a lower bound of about $0.29$ was established for {\em layered permutations}. An upper bound of about $0.37$ comes from a remarkable connection with Alternating Sign Matrices and the six-vertex model (see \Cref{rem:size ASMs vs PD}). As we shall see later, this question has an even more interesting interpretation from the Probability/Statistical Mechanics side as it is asking for the asymptotic growth of the energy in a model with \emph{long-range interactions} (pipes are not allowed to cross more than once), and for such models no general tools are known.

A natural one-parameter generalization of Schubert
polynomials are the Grothendieck polynomials
$\mathfrak{G}_w^{\beta}$, which represent $K$-theoretic
classes of the flag variety. Extending Stanley's questions, we
would like to understand the asymptotic behavior of maximal
 principal specializations $\max_{w\in S_n}\mathfrak{G}_w^{\beta}(1^n)$
of Grothendieck polynomials. This
question was first touched on in
\cite{MPP4GrothExcited}, \cite{dennin2022}. Establishing a family of
such maximal permutations would shed light on the
Schubert questions as $\beta \to 0$, as well as on the corresponding long-range interaction model.

Thanks to a combinatorial model for Grothendieck and
Schubert polynomials, these questions have very natural
statistical mechanics interpretations.
Namely, both polynomials are partition functions of tilings
into crosses and elbows of a size $n$ triangle (staircase) shape,
which result in a configuration of $n$ ``pipes''. Such pipe configurations
are often called \emph{pipe dreams}.
In the Schubert case, the only valid
tilings are the ones where no two pipes cross more than
once. This is a global (long-range interaction) condition,
which makes the model much less tractable.
In the Grothendieck case
for $\beta=1$, all tilings are allowed, but the pipes must be
resolved (reduced) to obtain a permutation. One of the key ideas leading to our analysis is that this model can be mapped to a
colored stochastic six-vertex model
(and furter to TASEP, the Totally Asymmetric Simple Exclusion Process).
The resulting interacting particle systems have only local (short-range) interactions, and are
amenable to techniques from Integrable Probability.

\medskip

In this work,
we investigate the asymptotics of the \emph{typical
$\beta=1$ Grothendieck random permutations}
and characterize their limit shape which is described by a permuton.
% In principle, the typical
% permutation (sampled from the limiting permuton) should give the
% asymptotically maximal principal specialization
% $\mathfrak{G}_w^{\beta=1}(1^n)$, but we do not address this question here.
We also study fluctuations of Grothendieck random permutations around the limiting permuton.
They
are of order $n^{\frac13}$ and asymptotically have the Tracy--Widom GUE distribution.
This distribution was
first observed in the fluctuations of
the largest eigenvalue of Gaussian
random matrices
with unitary symmetry. By now, having Tracy--Widom fluctuations is an indication that a
model is within the
Kardar--Parisi--Zhang (KPZ) universality class \cite{CorwinKPZ}, which includes a wide range of
random growth models and interacting particle systems. We also derive the expected number of inversions, which are of order $n^2$.

\medskip

We also study the natural, yet not as algebraically
motivated, model, where the staircase is tiled with crosses
with probability $p$ and elbows with probability $1-p$, and
the pipes follow the tiles without any reductions or
reassignments. Using simple random walks, we
show that the resulting random permutation is close to the identity, and, moreover,
that the total
displacement
$\sum_{i=1}^{n}|i-w_i|$
and length (number of inversions)
of this permutation
are of order $n^{3/2}$. This was motivated by, and partially resolves, an open problem of Colin Defant.

\medskip

Returning to the original question, we  investigate the asymptotics
of $\max_{w \in S_n} \mathfrak{G}_w^\beta(1^n)$.
When
$\beta=1$, this maximum is known to be $O(2^{\binom{n}{2}})$. Using the
correspondence with
2-enumerated Alternating Sign Matrices
(equivalently, the six-vertex model with domain wall boundary conditions
and free-fermion weights; or the model of uniform domino tilings of the Aztec diamond),
we show that a large family of layered
permutations also achieve this asymptotic maximum. For
general $\beta$, we establish certain bounds for the maximal
 principal specialization.

\subsection{Pipe dreams and permutations}
\label{sub:pipe_dreams}
We denote by $S_n$ the set permutations of
$\{1,2,\ldots,n\}$ that we write in the one-line notation
$w=w_1w_2\cdots w_n$ unless indicated otherwise.
We also denote the image of $i$ under $w$ by $w(i)$, and use the notation
$w_i=w(i)$ interchangeably when this does not lead to confusion.
The longest
permutation $n\,n-1\ldots 2\,1$ is denoted by $w_0=w_0(n)$.
For a permutation $w$ of length $\ell$, we denote by $R(w)$
the set of \emph{reduced words} of $w$, that is, tuples
$(r_1,\ldots,r_{\ell})$ such that $s_{r_1}\cdots s_{r_\ell}$
is a reduced decomposition of $w$, where $s_i=(i, i+1)$ are the
simple transpositions.

Grothendieck polynomials can  be defined combinatorially
via pipe dreams (equivalently, rc-graphs), as partition functions of the following model.

A \emph{pipe dream}
of order $n$
is a tiling of the
staircase shape
(having $n-1$ boxes in the first row,
$n-2$ boxes in the second row, and so on,
with boxes left-justified)
by tiles of two types:
\emph{bumps}
\raisebox{-2pt}{\includegraphics[scale=0.4]{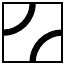}}
and
\emph{crossings}
\raisebox{-2pt}{\includegraphics[scale=0.4]{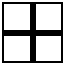}}.
The $n$-th diagonal below the staircase is
equipped with half bumps
\raisebox{-2pt}{\includegraphics[scale=0.4]{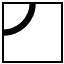}}
(whose boundary we do not draw).
Each of the boxes can have a tile of any type, so there
are $2^{\binom{n}{2}}$ pipe dreams of order $n$.
See \Cref{fig:pipe_dream}, left, for an example of a pipe dream of order $6$.

\begin{figure}[htpb]
	\centering
	\includegraphics[width=.75\textwidth]{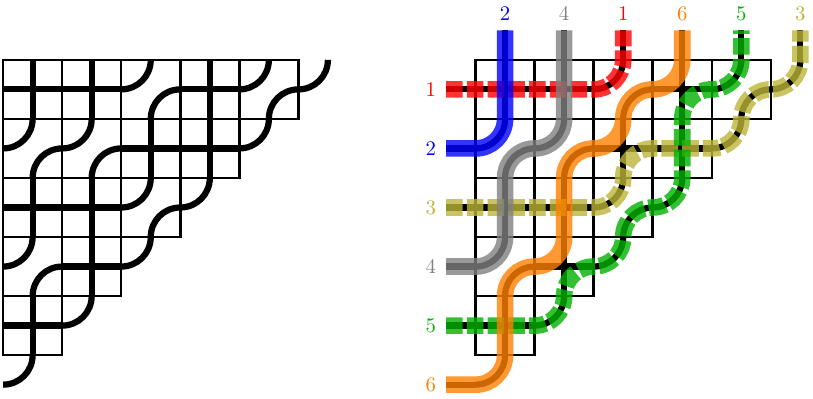}
	\caption{\textbf{Left}: A pipe dream $D$ of order $6$.
		\textbf{Right}: A reduction of the pipe dream leading to the permutation
		$w(D)=241653$.
		The right image appears in color online. Dashing is added for the
		printed version and accessibility.}
	\label{fig:pipe_dream}
\end{figure}

A pipe dream (a tiling of the staircase shape) forms
a collection of strands (or \emph{pipes}) labeled~$1$ to $n$ from
the row where they start.
A pipe dream is called \emph{reduced} if
any two pipes cross through each other at most once.

\begin{definition}[Reduction of a pipe dream]
	\label{def:reduction_pipe_dream}
	Given a pipe dream~$D$
	that is not necessarily reduced, the \emph{reduction} of~$D$
	is a unique reduced pipe dream~$D'$ obtained as follows:
	starting at the bottom left tile traverse the pipe dream upwards along columns and to the right. For each encountered crossing,
	replace it with a bump if the pipes have already crossed in the already traversed squares.
	The labeling of pipes together with a reduction
	is indicated by colored paths in
	\Cref{fig:pipe_dream}, right.
\end{definition}

\begin{definition}[Permutation from a pipe dream]
	\label{def:permutation_from_pipe_dream}
	One can associate a permutation $w(D)\in S_n$ to a pipe dream
	of order $n$ as
	follows. If $D$ is reduced then $w(D)^{-1}_j$ is the column where
	the pipe $j$ ends up in.\footnote{Throughout
	the paper, the column coordinate~$j$ increases from left to right,
	and the row coordinate~$i$ increases from top to bottom.}
	Equivalently, the column $j$ contains the exiting pipe of color
	$w(D)_j$. Note that to capture the column number for the pipe $j$,
	we need the inverse permutation $w(D)^{-1}$.

	If $D$ is not reduced, then $w(D)$
	is the permutation associated to the reduction $D'$ of $D$.
\end{definition}

\begin{remark}
	\label{rmk:Demazure_product}
	Alternatively, the permutation $w(D)$
	from a non-reduced pipe dream
	can be defined using the \emph{Demazure product}
	\cite{demazure1974desingularisation}.
	Namely, for an elementary transposition
	$s_i=(i,i+1)$ (with $1\le i\le n-1$)
	and a permutation $w\in S_n$, the Demazure
	(or 0-Hecke) product
	is defined as
	\begin{equation}
		\label{eq:Demazure_product}
		s_i\star w\coloneqq
		\begin{cases}
			ws_i, & \textnormal{if}\ \ell(ws_i)>\ell(w),\\
			w, & \textnormal{otherwise}.
		\end{cases}
	\end{equation}
	Here $\ell(\cdot)$ is the length of a permutation, that is,
	its number of inversions.

	To each cross in $D$, associate a transposition $s_i$.
	Reading from the bottom left to top right, we obtain a
	(not necessarily reduced) word.
	For example, the pipe dream in \Cref{fig:pipe_dream}, left,
	corresponds to the word $s_5s_5s_3s_4s_1s_2s_4s_5s_4$.
	Replacing this product
	of elementary transpositions
	by the Demazure product, we obtain the permutation
	$w(D)^{-1}$. For our example, the Demazure product is
	\begin{equation*}
		s_5\star
		s_5\star
		s_3\star
		s_4\star
		s_1\star
		s_2\star
		s_4\star
		s_5\star
		s_4=
		s_5s_2s_1s_4s_3s_5=316254,
	\end{equation*}
	which is the inverse of $w(D)$ given in \Cref{fig:pipe_dream}, right.
\end{remark}

Let $\PD(n)$ and $\RPD(n)$ be the sets of pipe
dreams and reduced pipe dreams of size $n$.
For each $w\in S_n$,
let $\PD(w)$ and
$\RPD(w)$ be, respectively, the sets of pipe dreams and reduced pipe dreams~$D$ such that $w(D)=w$. Note that
$\#\PD(n)=2^{\binom{n}{2}}$, whereas there is no simple
known formula for $\#\RPD(n)$
\cite[\href{http://oeis.org/A331920}{A331920}]{oeis}. Given
a pipe dream $D$, let $\cross(D)$ be the set of
coordinates~$(i,j)$
of the cross tiles.
The \emph{weight} of $D$ is the monomial
$\wt(D)\coloneqq \prod_{(i,j)\in \cross(D)} x_i$.
For example, for the pipe dream in \Cref{fig:pipe_dream}, left, we have
$\wt(D)=x_1^3 x_2^2 x_3^2 x_4 x_5$.

The nonsymmetric Grothendieck polynomials generalize the Schubert and the Schur polynomials,
and capture the $K$-theory of the flag variety.
They can be defined or interpreted in a number of ways,
including via
divided difference operators \cite{Lascoux1990}, see \Cref{sub:Grothendieck_polynomials_intro},
pipe dreams \cite{BilleyBergeron}, bumpless pipe dreams
\cite{Weigandt2020_bumpless},
\cite{LamLeeShimozono},
and solvable lattice models
\cite{BrubakerALCO}, \cite{BuciumasScrimshaw}.
We can also define them as the partition functions of this model, via the
the following
result  due to
\cite{fomin1996yang}, \cite{fomin1994grothendieck}, and \cite{BilleyBergeron}.
\begin{theorem}[Grothendieck polynomials as sums over pipe dreams]
	\label{thm:Grothendieck_pipe_dreams_intro}
	For any $w\in S_n$, we have
\begin{equation} \label{eq: Groth in terms of pd}
	\mathfrak{G}_w^\beta(x_1,\ldots,x_n ) \,=\, \sum_{D \in \PD(w)} \beta^{\#\cross(D)-\ell(w)} \wt(D).
\end{equation}
In particular, setting $\beta=0$ forbids non-reduced
pipe dreams, so the Schubert polynomial is
\begin{equation}
	\label{eq: schubs in terms of rpd}
	\mathfrak{S}_w({\bf x}) = \sum_{D\in \RPD(w)} \wt(D).
\end{equation}
\end{theorem}

\subsection{The origins of Grothendieck polynomials}
\label{sub:Grothendieck_polynomials_intro}

The \emph{(single) Grothendieck polynomials} were introduced
by Lascoux and Sch\"utzenberger
\cite{lascoux1982structure},
\cite{Lascoux1990} to study the $K$-theory of
flag varieties. Their original recursive definition is as
follows. Let $\pi_i:\mathbb{Z}[\beta][{\bf x}]\to
\mathbb{Z}[\beta][{\bf x}]$ denote the \emph{isobaric
divided difference operator}:
\begin{equation*}
	\pi_i f\coloneqq\frac{(1-\beta x_{i+1})\ssp f-(1-\beta x_i)\ssp s_i\cdot  f}{x_i-x_{i+1}},
\end{equation*}
where $s_i$ acts on a polynomial $f$ by permuting $x_i$ and $x_{i+1}$.
The Grothendieck polynomials $\mathfrak{G}_w^\beta$ are recursively
determined by the following conditions:
\begin{enumerate}[$\bullet$]
	\item
		For the longest permutation, we have
		$\mathfrak{G}_{w_0(n)}^\beta=
		x_1^{n-1}x_2^{n-2}\cdots x_{n-1}$.
	\item For all $w\in S_n$ and $i=1,\ldots, n-1$
		such that
		$\ell(w s_i)=\ell(w)+1$,
		we have
		$\mathfrak{G}_w^\beta =
		\pi_i \mathfrak{G}_{ws_i}^\beta$.
\end{enumerate}
Setting $\beta=0$ in $\mathfrak{G}_w^\beta$, we obtain a Schubert
polynomial $\mathfrak{S}_w$ which represents cohomology
classes of Schubert cycles in the flag variety
\cite{bernstein1973schubert},
\cite{demazure1974desingularisation},
\cite{lascoux1982structure}.
Note that some authors use the parameter $(-\beta)$ instead of $\beta$.
Our choice of the sign of $\beta$ is dictated by having positive coefficients in combinatorial
formulas for the Grothendieck polynomials.

\subsection{Grothendieck random permutations from reduced pipe dreams}
\label{sub:random_permutations}

Fix $p\in[0,1]$.
Equip the set of all pipe dreams
with a probability measure by independently
placing the tiles in each box:
\begin{equation}
	\label{eq:pipe_dream_measure}
	\raisebox{-6pt}{\includegraphics[scale=0.6]{cross}}
	\ \; \textnormal{with probability}\ p,
	\qquad
	\raisebox{-6pt}{\includegraphics[scale=0.6]{elbows}}
	\ \; \textnormal{with probability}\ 1-p.
\end{equation}
In particular, for $p=\frac12$, we have the uniform measure on the set of pipe dreams.

By reducing this random pipe dream as in \Cref{def:reduction_pipe_dream},
we obtain a
random permutation $\mathbf{w}\in S_n$
which we call the \emph{Grothendieck random permutation} (of order~$n$ and with parameter~$p$;
we suppress the dependence on $n$ and $p$ in the notation).
The name is justified by a connection with the polynomials $\mathfrak{G}_w^{\beta=1}$.
Indeed,
we have for any $w\in S_n$:
\begin{equation}
	\label{eq:pipe_dream_specialization_p_Grothendieck_intro}
	\Prob(\mathbf{w}=w)
	=
	\sum_{D \in \PD(w)} p^{\cross(D)}(1-p)^{\mathrm{elbow}(D)}
	=
	(1-p)^{\binom{n}{2}}\ssp \mathfrak{G}_w^{\beta=1}
	\left( {\frac{p}{1-p},\ldots, \frac{p}{1-p}} \right),
\end{equation}
where the first equality is simply the definition of the measure \eqref{eq:pipe_dream_measure},
and the second immediately follows from \Cref{thm:Grothendieck_pipe_dreams_intro} with $\beta=1$.
In \Cref{sub:square_ice_aztec_asymptotics} below, we explain how the
same distribution on permutations arises from
the six-vertex model with domain wall boundary conditions and free-fermion weights (corresponding
to 2-enumeration of Alternating Sign Matrices).
\begin{figure}[htpb]
	\centering
	\includegraphics[width=.32\textwidth]{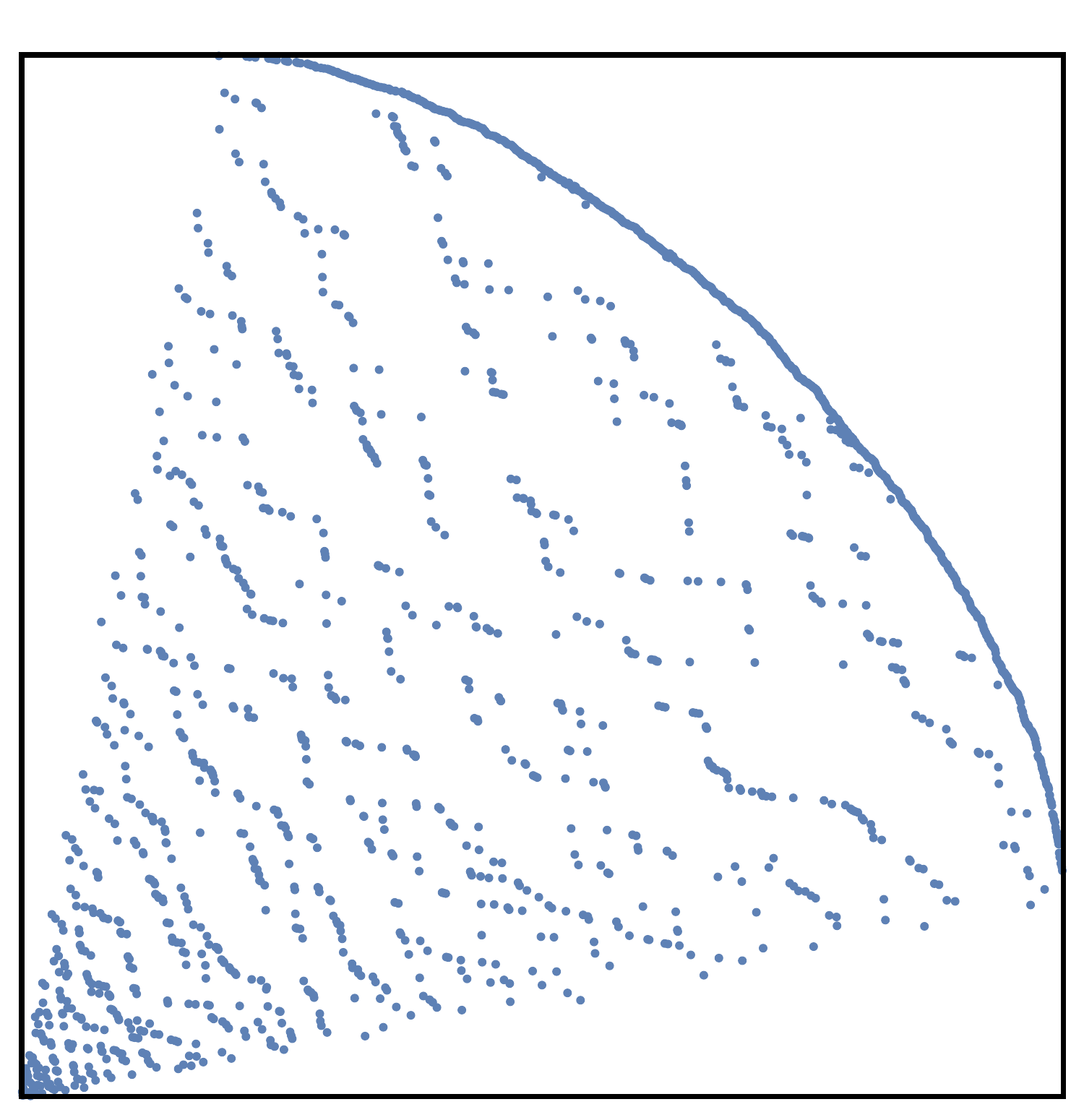} \
	\includegraphics[width=.32\textwidth]{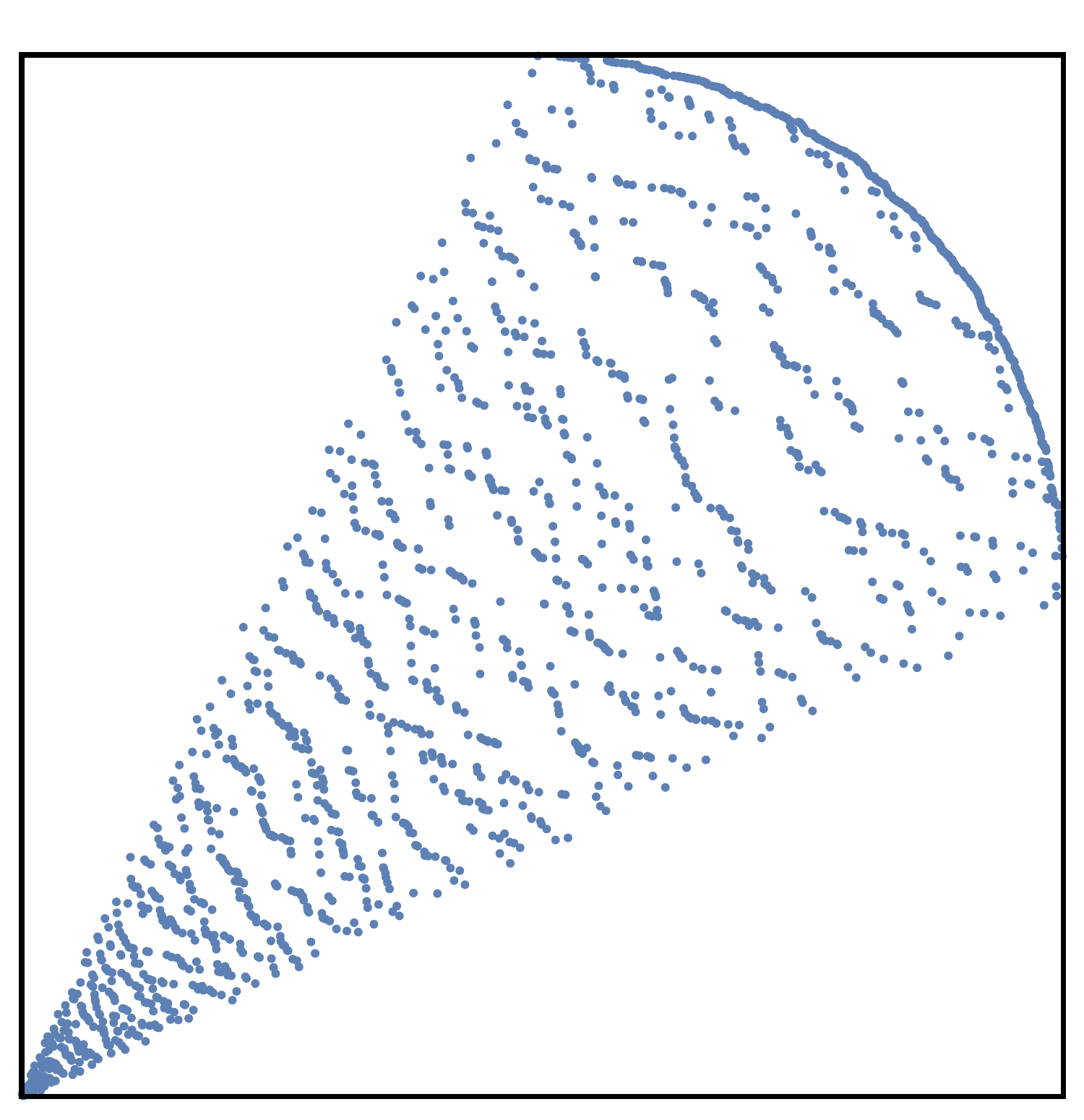} \
	\includegraphics[width=.32\textwidth, angle=90]{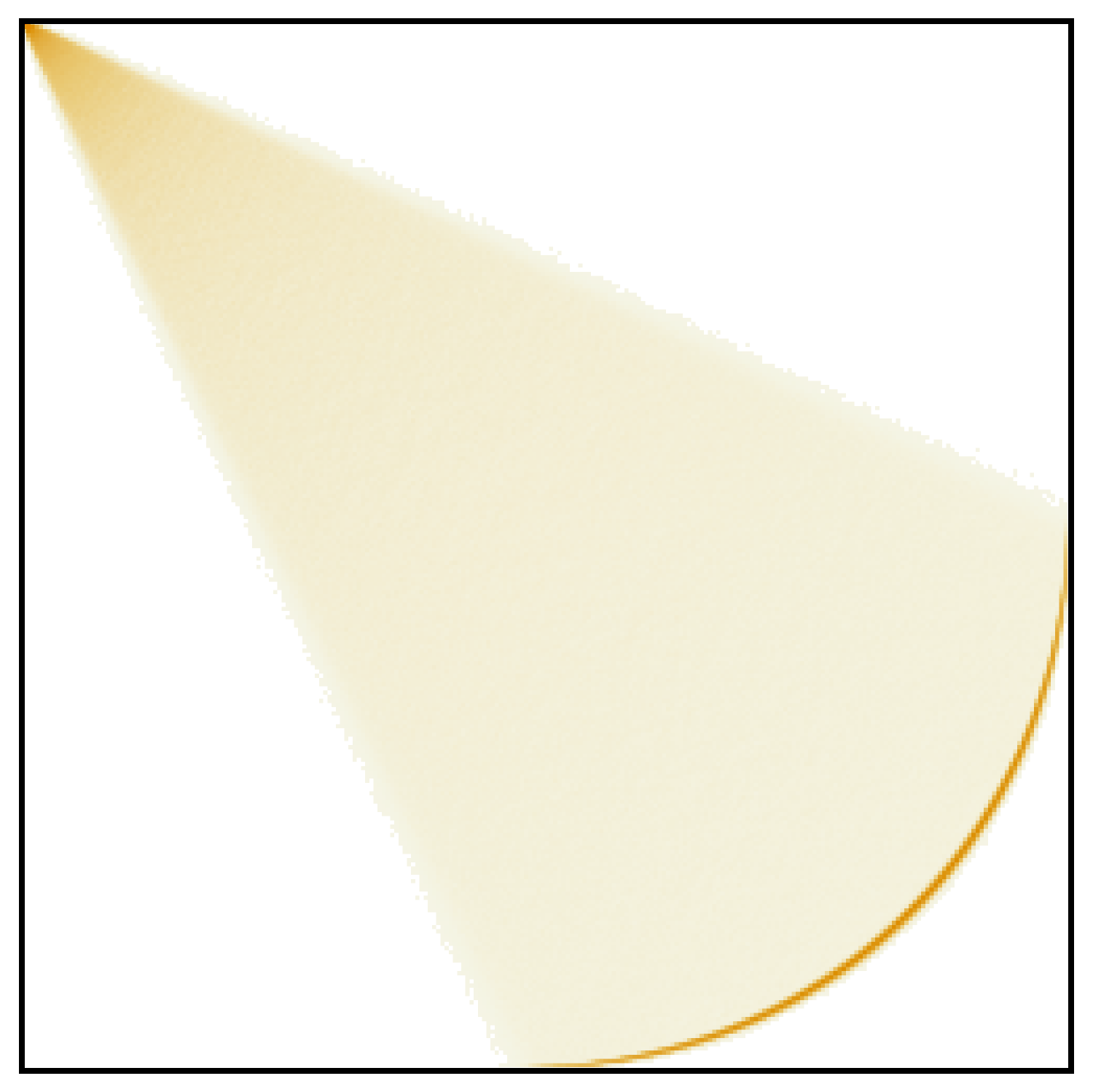}
	\caption{\textbf{Left and center}: A sample of a Grothendieck random permutation of order $n=2000$ with
		$p=\frac{4}{5}$ and $p=\frac{1}{2}$, respectively.
		\textbf{Right}: An average of Grothendieck
		random permutations with $n=2000$ and $p=\frac{1}{2}$
		over $2000$ samples.
		We take a sum of permutation matrices, and coarse-grain the result
		into $8\times 8$ blocks. The plot is the heat map of the resulting matrix,
		which approximates the Grothendieck permuton.
		The \texttt{C} code for these simulations is available
		on the \texttt{arXiv} as an ancillary file.
		An interactive simulation is available at
		\cite{Grothendieck_Simulations_1}.}
	\label{fig:intro_simulations}
\end{figure}

When $p=0$ or $p=1$, we almost surely have
$\mathbf{w}=\mathrm{id}$ or $\mathbf{w}=(n,n-1,\ldots,1 )$, respectively.
These cases are trivial, and in the rest of the paper we assume that $p\in(0,1)$.
Individual samples of
Grothendieck random permutations with $p=\frac{4}{5}$ and $p=\frac{1}{2}$
and a plot constructed from averaging over many samples
are given in
\Cref{fig:intro_simulations}.

\subsection{Asymptotics of Grothendieck random permutations}
\label{sub:asymptotics_permutations}

For $w\in S_n$,
define its \emph{height function} by
\begin{equation*}
	H(x,y)\coloneqq
	\operatorname{\#}\left(
		\left\{ w^{-1}(x),w^{-1}(x+1),\ldots,w^{-1}(n)  \right\}
		\cap
	\left\{ y,y+1,\ldots,n  \right\} \right),
\end{equation*}
where $1\le x,y\le n$. In terms of the pipe dream as in \Cref{fig:pipe_dream}, right,
$H(x,y)$ is the number of pipes of color $\ge x$ which exit through the positions
$\ge y$ at the top. In the example in \Cref{fig:pipe_dream}, right,
we have $H(4,3)=2$. See also \Cref{fig:height_function_permutation_matrix}
for an interpretation in terms of the permutation matrix of $w$,
where $H(x,y)$ gives the number of 1s in the rectangle with lower left corner at $(x,y)$.

Let now $\mathbf{w}\in S_n$ be the Grothendieck random permutation with parameter $p\in(0,1)$
which we assume fixed.
This makes the height function $H(x,y)$ random, too.
We show that $H(x,y)$ satisfies the law of large numbers,
and characterize its asymptotic fluctuations:

\begin{theorem}
	\label{thm:Grothendieck_permutations_asymptotics_intro}
	\begin{enumerate}[{\bf{}1.\/}]
		\item
			There exists a limiting height function $\mathsf{h}^\circ$ such that
			\begin{equation*}
				\lim_{n\to\infty} n^{-1}\ssp H(\lfloor n\ssp\mathsf{x} \rfloor,
				\lfloor n\ssp\mathsf{y} \rfloor)=\mathsf{h}^\circ(\mathsf{x},\mathsf{y}),
				\qquad
				(\mathsf{x},\mathsf{y})\in[0,1]^2,
			\end{equation*}
			where the convergence is in probability.
			The function $\mathsf{h}^\circ(\mathsf{x},\mathsf{y})$
			is
			explicit, see
			\eqref{eq:h_circ_by_zones}.
			It is continuous and depends only on $p$.
			The graph of $\mathsf{h}^\circ$
			is given in
			\Cref{fig:Grothendieck_limit_shape_zones}, right.
		\item The fluctuations of $H(x,y)$ around $\mathsf{h}^\circ$
			belong to the KPZ universality class:
			\begin{equation*}
				\lim_{n\to\infty}\Prob
				\left( \frac{H(\lfloor n\ssp\mathsf{x} \rfloor,
				\lfloor n\ssp\mathsf{y} \rfloor)-n\ssp\mathsf{h}^\circ(\mathsf{x},\mathsf{y})}
				{\mathsf{v}(\mathsf{x},\mathsf{y})\ssp n^{1/3}}
				\leq r \right)=F_2(r),\qquad r\in \mathbb{R},
			\end{equation*}
			where $F_2$
			is the cumulative distribution function of the Tracy--Widom GUE distribution,
			and $\mathsf{v}(\mathsf{x},\mathsf{y})$ is given in \eqref{eq:V_constant_for_H}.
	\end{enumerate}
\end{theorem}
\Cref{thm:Grothendieck_permutations_asymptotics_intro}
justifies
the simulations in
\Cref{fig:intro_simulations}. We prove it in
\Cref{sec:Grothendieck_asymptotics}
by realizing $H(x,y)$ as an observable of a
discrete time
TASEP with parallel geometrically distributed jumps started from the densely packed
(step) initial configuration.
Here we omit the technical details of this coupling, and refer to
\Cref{thm:vertex_TASEP_matching} in the text.
By employing known techniques
from Integrable Probability and passing back to the
Grothendieck random permutation, we obtain the limit shape and fluctuations.

\medskip

The law of large numbers in the first part of
\Cref{thm:Grothendieck_permutations_asymptotics_intro}
means that the Grothendieck random permutations $\mathbf{w}\in S_n$
converge to a deterministic \emph{permuton}.
Recall that permutons are Borel probability measures $\mu$
on $[0,1]^2$ with uniform marginals, that
is,
$\mu\left( [0,1]\times [a,b] \right)=
\mu\left( [a,b]\times [0,1] \right)=b-a$.
For details on permutons
and their connection to pattern frequencies in permutations, we refer to
\cite{hoppen2013limits},
\cite{bassino2019universal}, and the survey
\cite{grubel2023ranks}.
Our limiting \emph{Grothendieck permuton} is completely determined
by the limiting height function
via
\begin{equation*}
	\mu\left( [\mathsf{x},1]\times [\mathsf{y},1] \right)
	=\mathsf{h}^\circ(\mathsf{x},\mathsf{y}).
\end{equation*}
Note that the ``complementary''
function
$\mu\left(
	[0,\mathsf{x}]\times [0,\mathsf{y}]
\right)$ is often referred to as \emph{copula} in the literature, cf.
\cite{grubel2023ranks}.
Our limiting permuton has a singular part, that is, a positive portion of its mass
is concentrated along the curve
\begin{equation*}
	\mathscr{E}_p\coloneqq
	\left\{
		(\mathsf{x},\mathsf{y})
		\colon
		(\mathsf{y}-\mathsf{x})^2/p+(\mathsf{y}+\mathsf{x}-1)^2/(1-p)=1,\
		1-p\le \mathsf{x}\le 1
	\right\}\subset [0,1]^2.
\end{equation*}
The total mass supported on this curve is equal to
(\Cref{prop:delta_mass})
\begin{equation}
	\label{eq:delta_mass_intro}
	\gamma_p\coloneqq 1-\sqrt{\frac{1-p}{p}}\ssp \arccos\sqrt{1-p}.
\end{equation}
In particular, $\gamma_{\frac{1}{2}}=1-\frac{\pi}{4}$.

As a corollary of the permuton convergence of \Cref{thm:Grothendieck_permutations_asymptotics_intro},
one can also obtain laws of large numbers for
arbitrary pattern counts in Grothendieck random permutations.
The simplest example is the number of inversions:

\begin{proposition}[\Cref{prop:Grothendieck_permuton_inversions}]
	\label{prop:Grothendieck_permuton_inversions_intro}
	Let
	$\mathbf{w}=\mathbf{w}(n)\in S_n$ be the Grothendieck random permutations with a
	fixed parameter $p\in(0,1)$.
	We have
	\begin{equation}
		\label{eq:Grothendieck_permuton_inversions_intro}
		\lim_{n\to\infty}\frac{\inv\left( \mathbf{w}(n) \right)}{\binom n2}=\gamma_p,
	\end{equation}
	where $\gamma_p$ is given by \eqref{eq:delta_mass_intro}.
\end{proposition}
The fact that the scaled number of inversions converges to $\gamma_p$,
the singular part of the Grothen\-dieck permuton, is surprising.
Besides exact computations given in \Cref{sub:Grothendieck_permuton_density},
we do not have a conceptual explanation for this phenomenon.

\begin{remark}
	\label{remark:Grothendieck_and_ASM}
	In \Cref{sub:square_ice,sub:square_ice_aztec_asymptotics}, we connect
	Grothendieck random permutations
	to
	2-enumerated Alternating Sign Matrices
	(bijectively corresponding to the six-vertex model with
	domain wall boundary conditions and free-fermion weights, and also to
	domino tilings of the Aztec diamond).
	The connection between Grothendieck polynomials and the
	six-vertex model goes back to
	\cite{Lascoux02ice}, see also \cite{Weigandt2020_bumpless}.
	Equivalences between 2-enumerated Alternating Sign Matrices,
	six-vertex models, and domino tilings of the Aztec diamond
	are well-known in statistical mechanics and integrable probability,
	see, for example,
	\cite{cohn-elki-prop-96}, \cite{CohnKenyonPropp2000}, and
	\cite{elkies1992alternating}, \cite{zinn2000six}, \cite{ferrari2006domino}.

	Via this connection,
	our asymptotic results (\Cref{thm:Grothendieck_permutations_asymptotics_intro})
	can be recast as results about a random permutation coming from the free-fermion
	six-vertex model
	with domain wall boundary conditions.
	Moreover, there is a vast amount of asymptotic results
	about domino tilings and free-fermion six-vertex models,
	and in \Cref{sub:square_ice_aztec_asymptotics}
	we show how some of them lead to further asymptotic properties
	of Grothendieck random permutations.
\end{remark}

\subsection{A deformation and random non-reduced pipe dreams}
\label{sub:variant_without_reduction}

Let us discuss a deformation of the Grothendieck random permutation
by means of a new parameter
$q\in[0,1]$.
Take a pipe dream of order~$n$ obtained
by randomly placing the tiles in the same way as in \eqref{eq:pipe_dream_measure}
(depending on $p\in (0,1)$). However, instead of reduction (\Cref{def:reduction_pipe_dream}),
let us apply a modified ($q$-randomized) procedure:
\begin{definition}[$q$-Reduction of a pipe dream]
	\label{def:q_reduction_pipe_dream}
	Assume that $D$ is a pipe dream
	that is not necessarily reduced.
	For each crossing tile, consider the incoming pipes.
	There are two cases:
	\begin{enumerate}[$\bullet$]
		\item If the lower colored pipe enters from the left (and the higher colored one from the bottom),
			then the pipes cross through each other with probability $1$.
		\item Otherwise, if the lower colored pipe enters from the bottom (and the higher colored one from the left),
			then the pipes cross through each other with probability $q$.
			That is, with probability $1-q$ we replace the crossing with a bump.
	\end{enumerate}
	The random $q$-reduction steps are done independently at each crossing tile
	when
	reading the staircase diagonal by diagonal starting from the bottom left.
	See \Cref{sub:matching_pipe_dreams} below for a detailed definition
	of the Markov process for $q=0$, which has the same order of reduction steps.
	Denote the random permutation obtained by the $q$-reduction of a random pipe dream
	with the measure \eqref{eq:pipe_dream_measure} by $\mathbf{w}^{(q)}\in S_n$.
	See
	\Cref{fig:q_reduction} for simulations.
\end{definition}

\begin{remark}
	\label{rmk:q_reduction}
	For $q=0$, the $q$-reduction procedure becomes the usual reduction from \Cref{def:reduction_pipe_dream}.
	For $q>0$, the $q$-reduction no longer tracks whether the pipes
	have crossed before. Instead, it simply applies the local $q$-reduction rule at each
	crossing tile,
	based on the relative colors of the incoming pipes.
\end{remark}

\begin{remark}
	The $q$-reduction of a pipe dream
	can be viewed as applying the $q$-Hecke product
	instead of the Demazure (0-Hecke)
	product as in \Cref{rmk:Demazure_product}.
	Now, instead of the symmetric group
	$S_n$, we need to pass to the corresponding (Iwahori-)Hecke
	algebra $\mathcal{H}_n(q)$, and the $q$-reduction of a pipe dream
	becomes a linear combination of the Hecke elements $T_w$
	corresponding to various permutations $w\in S_n$.
	By choosing the Hecke product in a ``stochastic'' way (as in
	\cite{bufetov2020interacting}), namely,
	\begin{equation*}
		T_{s_i}T_w=
		\begin{cases}
			(1-q)\ssp T_w+q\ssp T_{ws_i}, & \textnormal{if}\ \ell(ws_i)>\ell(w),\\
			T_w, & \textnormal{otherwise},
		\end{cases}
	\end{equation*}
	the linear combination of the elements $T_w$ becomes convex, and
	thus corresponds to a probability distribution on $S_n$.
	This distribution is the law of the inverse $(\mathbf{w}^{(q)})^{-1}$.
\end{remark}

The asymptotic analysis of $\mathbf{w}^{(q)}$ for $q\in (0,1]$
may be performed similarly to the case $q=0$.
The underlying particle system is a certain
$q$-deformation of TASEP which allows both left and right jumps (governed by different rules).
Setting $q=1$ removes the asymmetry, which changes the
normalization in the law of large numbers and fluctuations (as well as
the fluctuation distribution).
This behavior
parallels the different
scales of laws of large numbers
in the
KPZ and the
Edwards--Wilkinson (EW) universality classes.
The KPZ class asymptotics arise in TASEP and its asymmetric deformations,
while symmetric versions of TASEP fall into the EW class.
We refer to
\cite{CorwinKPZ} for a detailed discussion of these two universality classes.

In this paper, we mainly focus on the case $q=0$, and obtain
results outlined in \Cref{sub:asymptotics_permutations}.  In
\Cref{sec:non-reduced}, we also treat the case $q=1$.  In
the latter model, the $q$-reduction does not change the pipe
dream at all. Probabilistically, this is the most natural
way to associate a permutation to a random pipe dream with
distribution \eqref{eq:pipe_dream_measure}. On the other
hand, the $q=1$ model apparently lacks a rich algebraic
underpinning.

Finding the asymptotic number of inversions of
$\mathbf{w}^{(1)}=\mathbf{w}^{(1)}(n)
\in S_n$ as $n\to\infty$ was posed as an open
problem by Colin
Defant at the Richard Stanley's 80th birthday
conference.\footnote{Conference ``The Many Combinatorial Legacies of Richard P. Stanley: Immense Birthday Glory of the Epic Catalonian Rascal,'' \url{https://www.math.harvard.edu/event/math-conference-honoring-richard-p-stanley/}. Accessed: 07/23/2024.}
We show
the following
asymptotics:
\begin{theorem}[\Cref{thm:random_permutation_bounds}
	and
	\Cref{prop:id_permuton_limit}]
	\label{thm:non_reduced_intro}
	For $p\in[0,1)$ and any $\varepsilon>0$,
	the expected number of inversions satisfies for all sufficiently large $n$:
	\begin{equation}
		\label{eq:non_reduced_intro_bounds}
		\frac{2}{3\sqrt{\pi}} (1-\varepsilon)\ssp  n^{3/2} \sqrt{\frac{p}{1-p}}\leq
		\operatorname{\mathbb{E}}\bigl[\inv\bigl(\mathbf{w}^{(1)}(n)\bigr)\bigr] \leq
		\frac{4}{3\sqrt{\pi}} (1+\varepsilon) \ssp n^{3/2} \sqrt{\frac{p}{1-p}}.
	\end{equation}
	Moreover,
	the random permutations $\mathbf{w}^{(1)}(n)$ converge in distribution
	to the identity permuton supported
	on the diagonal of $[0,1]^2$.
\end{theorem}

Based on simulations (see \Cref{fig:q_reduction} for examples)
and supporting numerics, we make
further conjectures about the behavior of $\mathbf{w}^{(q)}$.
\eqref{eq:non_reduced_intro_bounds}:
\begin{conjecture}[\Cref{conj:natural_model_number_of_inversions}; settled in \cite{defant2024randomsubwordspipedreams}]
	\label{conj:q_reduction_q1}
	For any $p\in[0,1)$, we have convergence in probability:
	\begin{equation*}
		n^{-3/2}\ssp \inv(\mathbf{w}^{(1)})\to \varkappa\ssp \sqrt{\frac{p}{1-p}},\qquad n\to\infty,
	\end{equation*}
	where $\varkappa$ is a constant independent of $p$ whose value is near $0.53$.\footnote{After this article appeared on the \texttt{arXiv}, Defant settled this conjecture in \cite{defant2024randomsubwordspipedreams} and showed that $\varkappa=\dfrac{2\sqrt{2}}{3\sqrt{\pi}}$.}
\end{conjecture}
\Cref{thm:non_reduced_intro} implies that
$\varkappa$ (if it exists) must satisfy
$2/(3\sqrt\pi)\le \varkappa \le 4/(3\sqrt\pi)$, and the approximate value $0.53$
is conjectured based on numerical data.

We also expect that as $p=p(n)\to1$ with $n$, the random
permutations $\mathbf{w}^{(1)}(n)$ converge in
distribution to a nontrivial deterministic permuton (which
depends on the speed and the parameters in $p(n)\to1$).

\medskip

In the intermediate cases $q\in(0,1)$, let $H^{(q)}(x,y)$ be the height function of $\mathbf{w}^{(q)}$,
defined in the same way as in \Cref{sub:asymptotics_permutations} above.
\begin{conjecture}
	\label{conj:q_reduction_ASEP_like}
	For any fixed $p,q\in(0,1)$, the random permutations
	$\mathbf{w}^{(q)}(n)\in S_n$ converge to a permuton
	determined by a height function
	$\mathsf{h}^{(q)}(\mathsf{x},\mathsf{y})$ on $[0,1]^2$.
	That is, we have the limit in
	probability:
	\begin{equation*}
		\lim_{n\to\infty} n^{-1}\ssp H^{(q)}(\lfloor n\ssp\mathsf{x} \rfloor,
		\lfloor n\ssp\mathsf{y} \rfloor)=\mathsf{h}^{(q)}(\mathsf{x},\mathsf{y}),
		\qquad
		(\mathsf{x},\mathsf{y})\in[0,1]^2.
	\end{equation*}
\end{conjecture}
Fluctuations of $H^{(q)}(x,y)$ around $\mathsf{h}^{(q)}$ should
be governed by one of the laws appearing in the
KPZ universality class in the presence of a boundary
(Tracy--Widom GUE distribution,
its GOE/GSE counterparts and crossovers,
or Gaussian fluctuations on scale $n^{1/2}$).
We refer to \cite{baik2018pfaffian}, \cite{barraquand2018stochastic} for further discussion of
fluctuations in half-space models from the KPZ universality class.

\begin{figure}[htpb]
	\centering
	\includegraphics[width=.24\textwidth, angle=90]{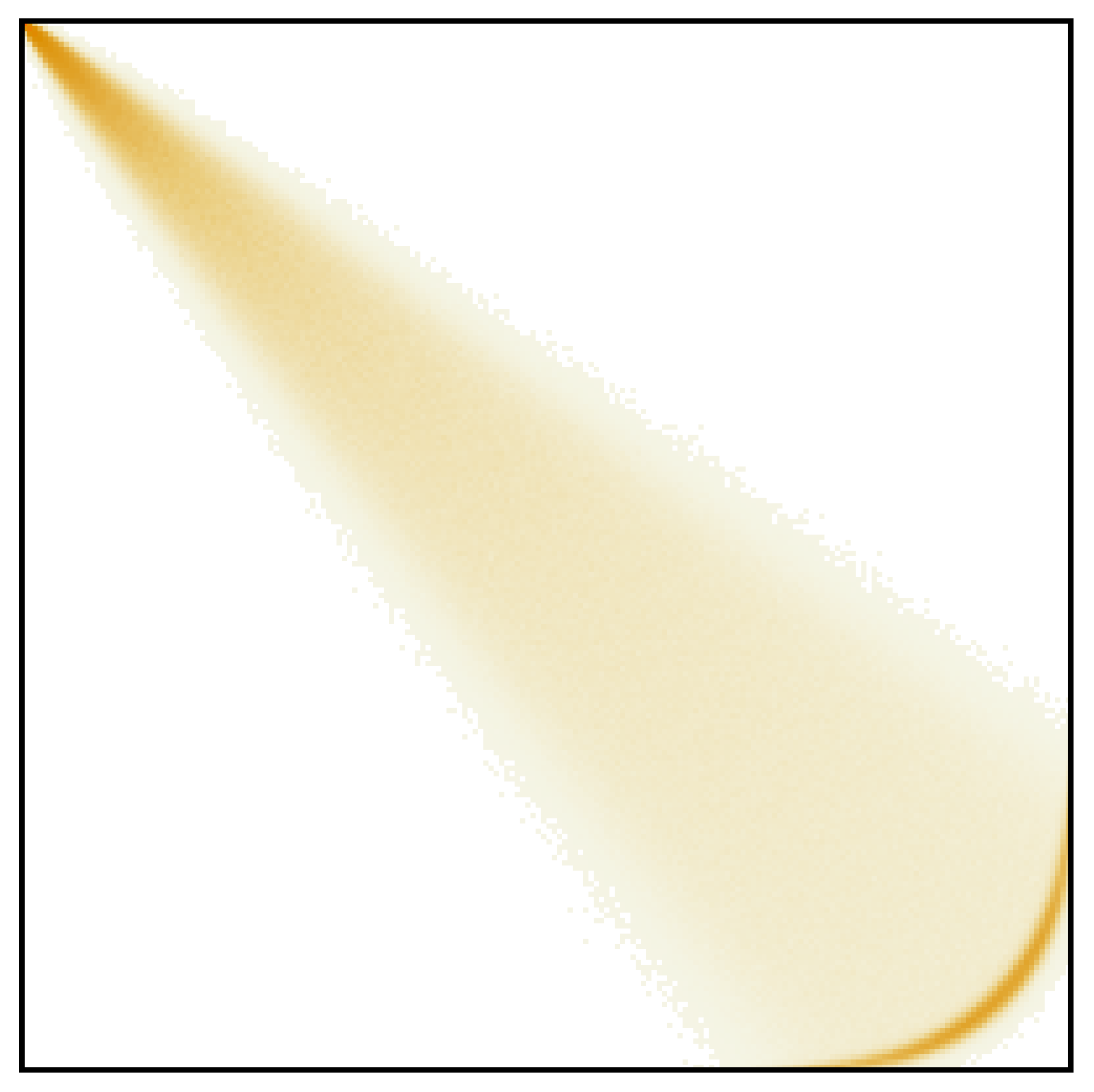}\
	\includegraphics[width=.24\textwidth, angle=90]{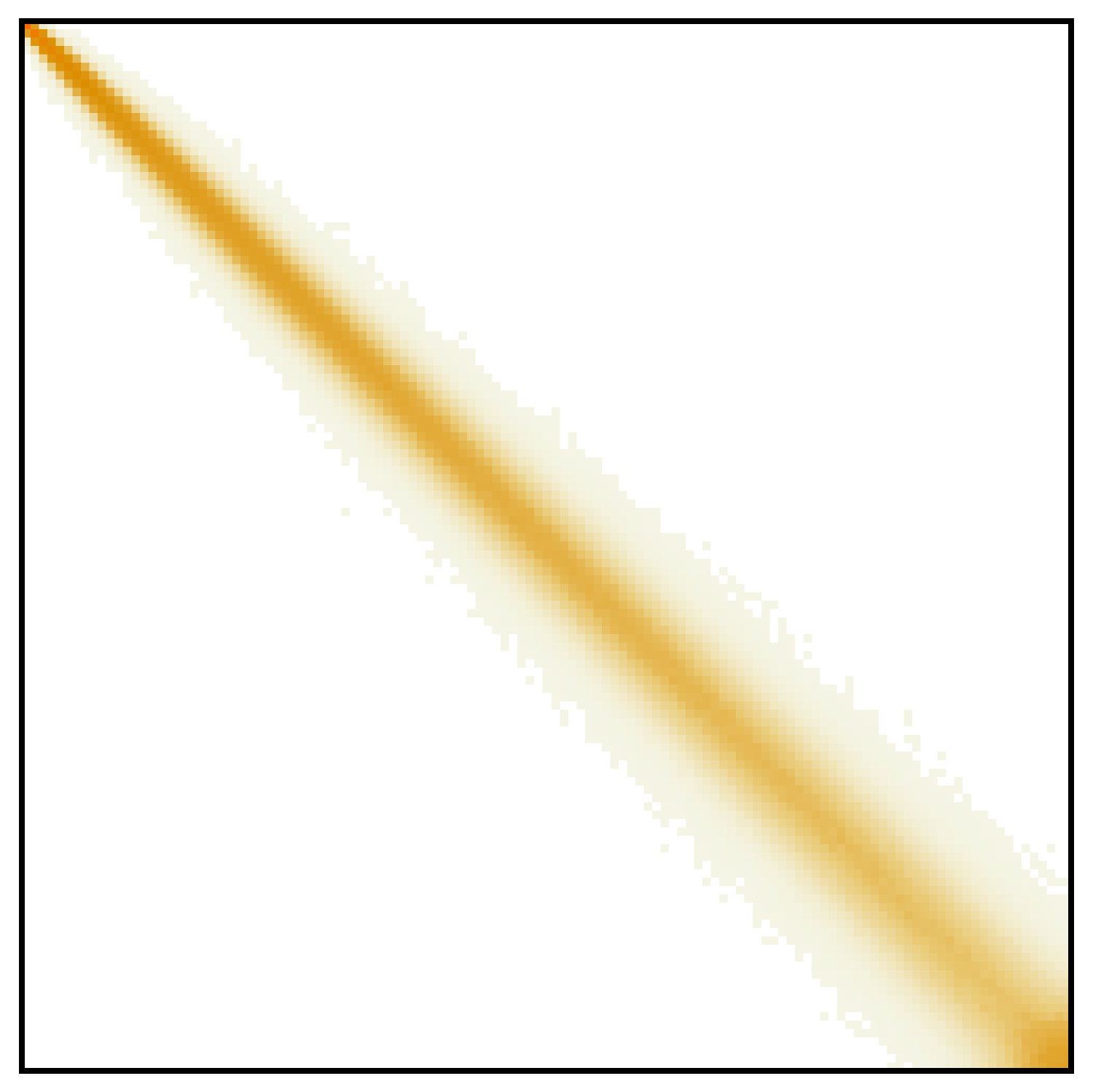}\
	\includegraphics[width=.24\textwidth, angle=90]{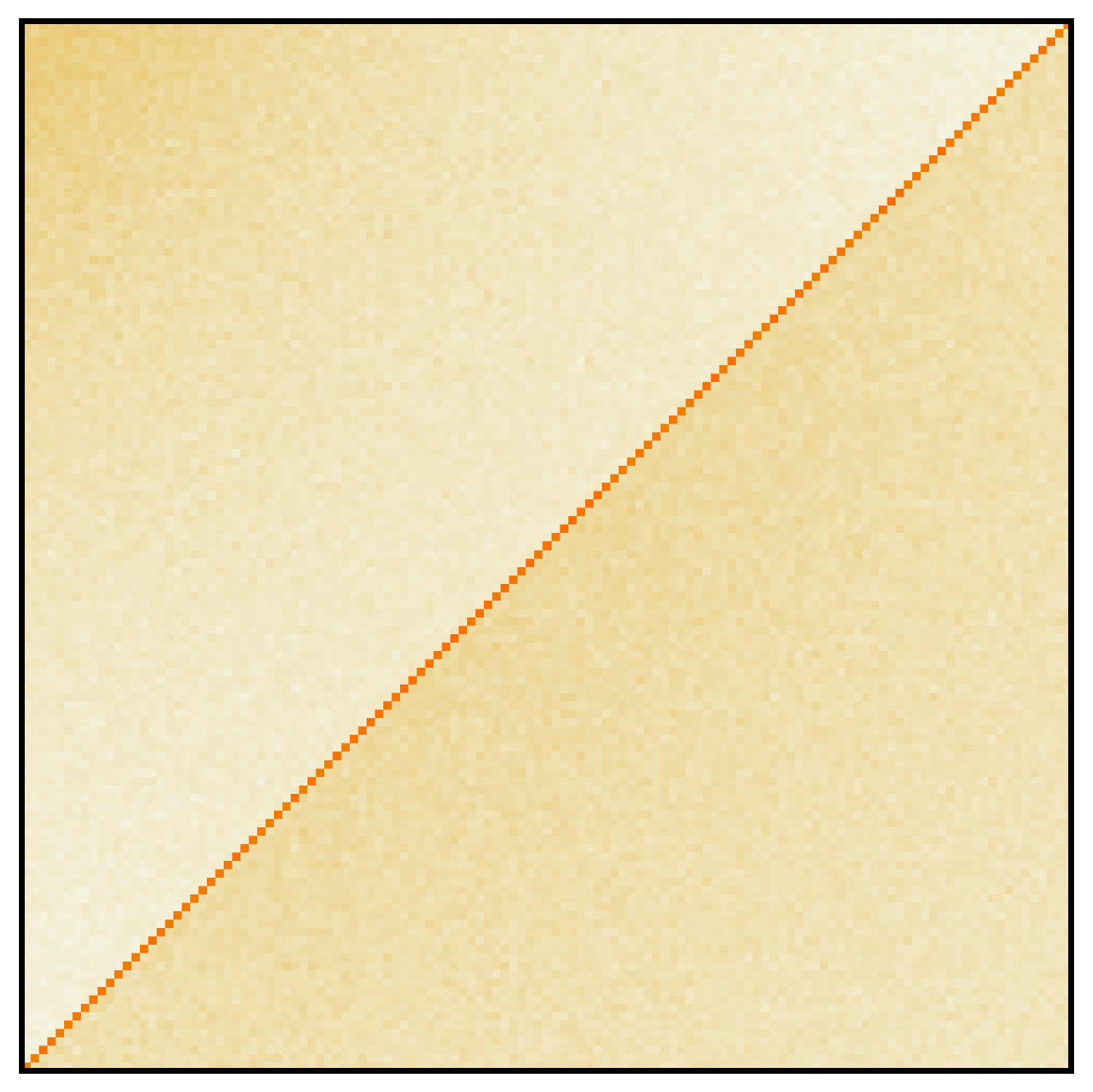}\
	\includegraphics[width=.232\textwidth]{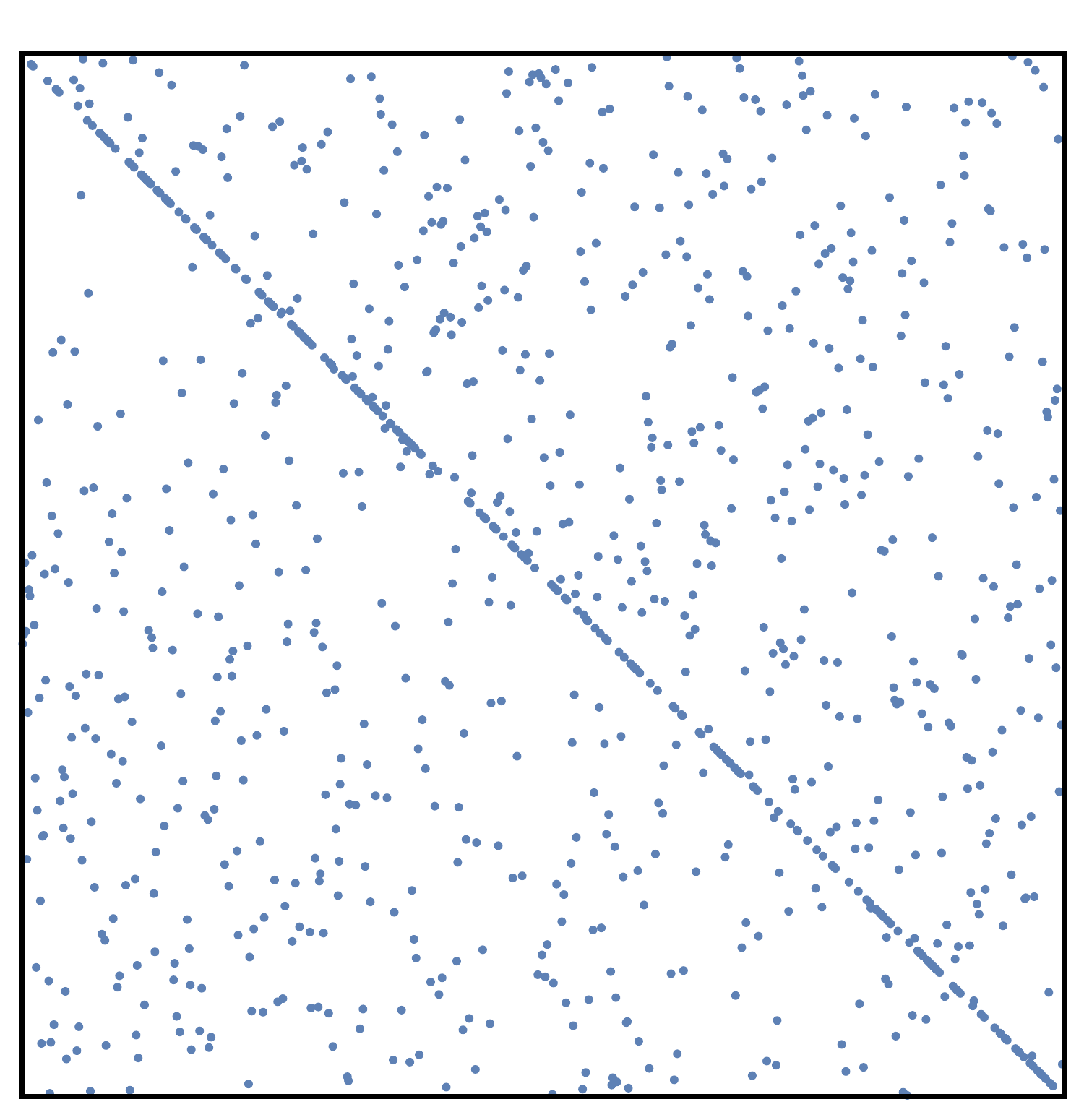}
	\caption{Simulations of $\mathbf{w}^{(q)}$, from left to right:
		(1) An average over $2000$ samples with $n=1000$, $p=0.7$, and $q=0.5$. We see curved boundary
		in the top right, similarly to $q=0$.
		(2) An average over $2000$ samples with $n=1000$, $p=0.5$, and $q=1$. The random permutation $\mathbf{w}^{(1)}$
		is
		close to the identity (and converges to it, see \Cref{thm:non_reduced_intro}).
		(3) An average over $2000$ samples with $n=1000$, $p=0.9985$, and $q=1$. The permutation has a positive mass
		close to the anti-diagonal.
		(4) A single sample with $n=1000$, $p=0.9985$, and $q=1$, with the anti-diagonal clearly visible. The latter two simulations suggest a permuton limit as $p=p(n)\to 1$.
	The \texttt{C} code for these simulations is available on the \texttt{arXiv} as an ancillary file.}
	\label{fig:q_reduction}
\end{figure}

\subsection{Asymptotics and maxima of Grothendieck principal specializations}
\label{sub:intro_Groth_spec}

We are interested in the following \emph{principal specializations} of the
Grothendieck polynomials:
\begin{equation}
	\label{eq:Upsilon_w_definition_intro}
	\Upsilon_w(\beta) \coloneqq \mathfrak{G}_w^{\beta}\bigl (\underbrace{1,1,\ldots,1}_n\bigr),
	\qquad w\in S_n.
\end{equation}
In particular, $\Upsilon_w(0)$ is the principal specialization of the Schubert polynomial.
For $\beta=1$, these quantities are the (unnormalized) probability weights
of the Grothendieck random permutation~$\mathbf{w}$ with $p=\frac{1}{2}$,
see \eqref{eq:pipe_dream_specialization_p_Grothendieck_intro}.

Set
\begin{equation}
	\label{eq:v_n_u_n_notation_intro}
	v_n(\beta)\coloneqq \sum_{w \in S_n} \Upsilon_w(\beta),\qquad
	u_n(\beta)\coloneqq \max_{w\in S_n} \Upsilon_w(\beta).
\end{equation}
From \Cref{thm:Grothendieck_pipe_dreams_intro}, we have:
\begin{equation}
	\label{eq:Grothendieck_specialization_v_n_u_n_intro}
	v_n(1)=2^{\binom{n}{2}},\qquad  u_n(1)=2^{\binom{n}{2}-o(n^2)}, \qquad n\to\infty.
\end{equation}
The first equality is exact, while the second one is asymptotic.
This asymptotic behavior
follows from
the cardinality
$n!\sim e^{n\log n+O(n)}\ll 2^{\binom{n}{2}}$ of the set over which we take the maximum.

For $u \in S_k, w \in S_m$, we denote
\begin{equation}
	\label{eq:cross_product_of_permutations}
	u \times w \coloneqq (u(1), \ldots, u(k), w(1)+k, \ldots, w(m) + k)\in S_{k+m}.
\end{equation}
For a composition $b = (b_{\ell}, \ldots, b_1)$ of $n = b_1 + \cdots + b_{\ell}$, the {\it layered permutation} $w(b) \in S_n$ is defined as follows:
\begin{equation}
	\label{eq:layered_permutation_definition_intro}
w(b) = w_0(b_{\ell}) \times \cdots \times w_0(b_{1}),
\end{equation}
where $w_0(k)$ is the full reversal permutation of order $k$.
Denote by $L_n\subset S_n$ the subset of layered permutations,
and let
\begin{equation}
	\label{eq:u_n_prime_definition_intro}
	u_n'(1)\coloneqq \max_{w\in L_n} \Upsilon_w(1).
\end{equation}

\begin{figure}[htbp]
	\raisebox{-113pt}{\includegraphics[width=.5\textwidth]{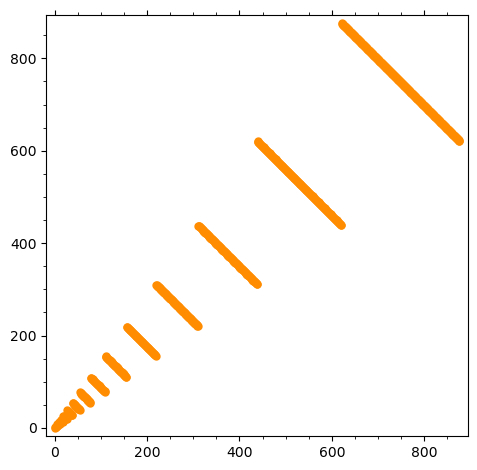}}
  $
	\qquad \qquad
\begin{array}{rrr}
\hline
\sum_{i=1}^k b_i & b_1  & \qquad \quad f(n)\\ \hline
4&2&0.175460\\
6&2&0.215599\\
9& 3&0.279931\\
13& 4&0.328982\\
19&6&0.375242\\
27&8&0.406540\\
39&12&0.433385\\
55&16&0.450935\\
78&23&0.464651\\
110&32&0.474448\\
156&46&0.481823\\
220&64&0.486974\\
311&91&0.490735 \\
439&128&0.493404\\
621&182 & 0.495329 \\
877&256 & 0.496684\\ \hline
\end{array}
$
\caption{\textbf{Left}: Permutation matrix of a layered permutation $w(b)\in S_{877}$, where
		the composition is
		$b=(256, 182, 128, 91, 64, 46, 32, 23, 16, 12, 8, 6, 4, 3, 2, 2, 1, 1)$.
		Note that
		$b_i /b_{i+1} \approx 1/\sqrt{2}$.
		\textbf{Right}: Table of exact values for $3\le k\le 19$ of layered permutations $w(b)$ with $b_i/b_{i+1} \approx 1/\sqrt{2}$. The third column is
	$f(n)\coloneqq  \frac{1}{n^2} \log_2 \Upsilon_{w(b)}(1)$, where $n=\sum_i b_i$.}
	\label{fig:value_ps_Groth_layered}
\end{figure}

In \Cref{sub:asymptotically_maximal_specializations}, we prove that
on layered permutations, the
$\beta=1$ Grothendieck polynomials
attain their asymptotic maximum:

\begin{theorem}[\Cref{thm:gmax}]
	\label{thm:gmax_intro}
There are sequences of layered permutations $w(b^{(n)})\in S_n$ so that
$$
\lim_{n \to \infty} \frac{1}{n^2} \log_2
\Upsilon_{w(b^{(n)})} (1)
= \frac{1}{2}.
$$
\end{theorem}
\Cref{thm:gmax_intro} implies that in leading order,
$u_n'(1)$ asymptotically behaves in the same way as
$v_n(1)$ and $u_n(1)$ \eqref{eq:Grothendieck_specialization_v_n_u_n_intro}.

Explicit constructions of such sequences of layered
permutations are given in \Cref{thm:gmax}. In particular, we
can take the parts of the compositions $b^{(n)} = (\ldots, b_2, b_1)$
to be
geometric $b_i \sim
(1- \alpha)\alpha^{i-1} n$ for any $\alpha \in [1/\sqrt{2}, 1)$.
See
\Cref{fig:value_ps_Groth_layered} for an illustration.
Note, however that we do not know in the limit what
compositions $b$ of size $n$ yield the global maximum of
$\Upsilon_{w(b)}$ over all layered permutations. For
Schubert polynomials, the analogous question was settled in
\cite{MoralesPakPanova2019}. We present numerics
for principal specializations of the $\beta=1$ Grothendieck polynomials in
\Cref{app:data_Groth}, and explain the difference between the (numerically)
optimal layered permutations and the ones constructed in the proof of \Cref{thm:gmax}.

Regarding Stanley's problem on permutations achieving the
maximal Schubert specialization $u_n(0)$ (the case
$\beta=0$), the Merzon--Smirnov conjecture
\cite{merzon2016determinantal} states that \emph{the maximum
is attained on layered permutations}. On one hand, our
\Cref{thm:gmax_intro}
establishes that an asymptotic analog of this conjecture holds for
the $\beta=1$
Grothendieck polynomials.
On the other hand, the typical shape of Grothendieck
random permutations
(\Cref{fig:intro_simulations}, center and right; see also
\Cref{thm:Grothendieck_permutations_asymptotics_intro})
suggests that the
global maximum $u_n(1)$
may be achieved on permutations whose shape is far from layered.
In \Cref{sub:bounds_Groth}, we obtain new bounds on the quantities
$v_n(\beta)$ and $u_n(\beta)$ for general values of $\beta$,
but do not consider the relation
of maximal principal specializations to layered permutations.

\subsection{Outline}

In the Introduction, we defined the model and formulated our main results.
In \Cref{sec:pipe_dreams_colored_six_vertex}, we explain how the Grothendieck random permutation
$\mathbf{w}$ from \Cref{sub:random_permutations}
is sampled by running a Markov chain, more precisely, the colored stochastic six-vertex model.
In \Cref{sec:six_vertex_TASEP}, we connect the vertex model
to a discrete time TASEP with parallel update and geometrically distributed jumps.
This process is well-studied within Integrable Probability,
which leads to law of large numbers and Tracy--Widom fluctuation
results.
In \Cref{sec:Grothendieck_asymptotics}, we recast the asymptotic results about
TASEP into
limiting properties of
Grothendieck random permutations.
In \Cref{sec:non-reduced}, we consider a variant of our
model coming from non-reduced pipe dreams (described in \Cref{sub:variant_without_reduction}).
In \Cref{sec:maximal_specializations_Grothendieck}, we
address the question of asymptotics of  principal specializations
of $\beta=1$ Grothendieck polynomials, and show that they are asymptotically the largest
on layered permutations.
We also obtain new bounds on the principal specializations for general $\beta$,
based on an interpretation of Grothendieck polynomials in terms of
Alternating Sign Matrices.
In \Cref{sec:add_remarks_and_open_problems}, we discuss further directions and a few
open problems.

In \Cref{app:data_Groth}, we provide tables of maximal  principal specializations of $\beta=1$ Grothendieck polynomials on layered permutations.
In \Cref{sec:app_TASEP_asymptotics}, we provide the steepest descent computations
leading to the asymptotics in TASEP. While these computations are standard by
now, we include them to justify the exact values of the constants
in TASEP limit shape and fluctuations.

\subsection*{Acknowledgments}

We are very grateful to Jehanne Dousse for her valuable insights in the development of this project and doing experiments in Sage. We thank Igor Pak for bringing us together and initiating our collaboration.
We thank Daoji Huang,
Mark Shimozono and Tianyi Yu for sharing an early version of \cite{HSY}. We would also like to thank Dave Anderson, Philippe Biane, Jacopo Borga, Filippo Colomo, Colin Defant, Philippe DiFrancesco, Rick Kenyon, Matthew Nicoletti, Khaydar Nurligareev, Mikhail Tikhonov, and Anna Weigandt for helpful discussions.

The project was begun at the SQuaRE ``Young Tableaux Asymptotics''
at the American Institute of Mathematics. The authors thank
AIM for providing a supportive and mathematically rich
environment.
Part of this research was performed in Spring 2024 while
some of the authors were
visiting the
program
``Geometry, Statistical Mechanics, and Integrability''
at the Institute for Pure and Applied Mathematics
(IPAM), which is supported by the NSF grant DMS-1925919.

AM was partially supported by the NSF grant DMS-2154019, the NSERC Discovery Grant
RGPIN-2024-06246, and the FRQNT Team grant 341288. GP was partially supported by the NSF grant CCF-2302174. LP was partially supported by the NSF grant DMS-2153869 and the Simons Collaboration Grant for Mathematicians 709055. DY was partially supported by the \mbox{MSHE RK} grant AP23489146.

\section{From pipe dreams to the colored stochastic six-vertex model}
\label{sec:pipe_dreams_colored_six_vertex}

\subsection{Colored stochastic six-vertex model}
\label{sub:colored_six_vertex}

Introduce the vertex weights
$$
w_p(a,b;c,d)
=
w_p
\bigl(
	\begin{tikzpicture}[baseline=-3,scale=.7,very thick]
			\draw[fill] (0,0) circle [radius=0.025];
			\draw (0.5,0) -- (0.05,0);
			\draw (-0.5,0) -- (-0.05,0);
			\draw (0,0.05) -- (0, 0.5);
			\draw (0,-0.05) -- (0,-0.5);
			\node at (.25,-.4) {\scriptsize $a$};
			\node at (-.7,0) {\scriptsize $b$};
			\node at (-.25,.4) {\scriptsize $c$};
			\node at (.7,0) {\scriptsize $d$};
\end{tikzpicture}
\bigr),$$ where
$a,b,c,d\in \left\{0, 1,\ldots,n  \right\}$.
Here~$0$ represents
the absence of a pipe, and
positive numbers indicate the
pipes' colors.
We view $(a,b)$ and $(c,d)$ as incoming and outgoing pipes, respectively.
The weights are
defined as follows:
\begin{equation}
	\label{eq:vertex_weights}
	\begin{split}
		w_p(a,a;a,a) &=1;
		\\
		w_p(b,a;b,a)&=p, \qquad  w_p(b,a;a,b)=1-p;
		\\
		w_p(a,b;a,b)&=0, \qquad  w_p(a,b;b,a)=1,
	\end{split}
\end{equation}
where $0\le a<b\le n$.
The weights of all other configurations not listed in \eqref{eq:vertex_weights}
are zero. Let us indicate a few crucial properties of the weights:
\begin{enumerate}[$\bullet$]
	\item The weights \emph{conserve} the pipes:
		$w_{p}(a,b;c,d)=0$ unless $\{a,b\}=\{c,d\}$ as sets.
	\item The weights are \emph{stochastic}:
		\begin{equation*}
			w_p(a,b;c,d)\ge0,\qquad \sum_{c,d=0}^n w_p(a,b;c,d)=1 \quad \textnormal{for all}\ a,b.
		\end{equation*}
\end{enumerate}

\begin{remark}
	The weights $w_p$ come from the
	fundamental stochastic $R$-matrix
	for the quantum group $U_q(\widehat{\mathfrak{sl}_{n+1}})$
	\cite{Jimbo:1985ua}, see also \cite[Section~2.1]{borodin_wheeler2018coloured}
	for a brief overview.
	For the application to random pipe dreams, we specialize
	the \emph{quantum parameter} $q$ to zero
	(equivalently, to infinity in the normalization of \cite{borodin_wheeler2018coloured} --- this choice
	depends only on
	the order in which colors correspond to basis vectors of
	the fundamental representation of $\mathfrak{sl}_{n+1}$).
	The remaining parameter $p$ is related to the \emph{spectral parameter}
	in an integrable vertex model.
\end{remark}

\subsection{Matching random pipe dreams to the stochastic vertex model}
\label{sub:matching_pipe_dreams}

Attach coordinates
$(i,j)\in \mathbb{Z}_{\ge1} ^{2}$
to boxes of the
staircase shape, with~$i$ increasing down,
and~$j$ increasing to the right.
The staircase shape is then
$\updelta_n\coloneqq \left\{ (i,j)\colon i+j\le n \right\}$.

Place a stochastic vertex with the weight $w_p$ at each box $(i,j)\in \updelta_n$.
Let the initial condition along the left boundary of $\updelta_n$ be
the \emph{rainbow} one, with colors $1$ to $n$ from top to bottom.
Then we can sample a random configuration
of pipes as in \Cref{fig:pipe_dream}, right,
by running a discrete time Markov chain with time $\tau=j-i$, where
$-(n-1)\le \tau\le n-1$. At each step $\tau\to \tau+1$,
the configuration with $j-i\le \tau$
is already sampled, which determines the incoming colors at all boxes
$(i,j)\in \updelta_n$ with $j-i=\tau+1$. The next
step $\tau\to \tau+1$ consists in an independent update of the outgoing colors
at all boxes $(i,j)\in \updelta_n$ with $j-i=\tau+1$,
using the stochastic vertex weights
$w_p(a,b;\cdot,\cdot)$
\eqref{eq:vertex_weights}, $1\le a,b\le n$.
Here we view each $w_p(a,b;\cdot,\cdot)$ as a probability distribution on possible outputs,
$(a,b)$ or $(b,a)$.
Reading off the outgoing colors at the top boundary of $\updelta_n$,
we arrive at a random permutation $\mathbf{w}\in S_n$.

\begin{proposition}
	\label{prop:pipe_dreams_to_vertex_model}
	The random permutation $\mathbf{w}\in S_n$ obtained from the colored stochastic six-vertex model
	as described above has the same distribution as the Grothendieck random permutation
	defined in \Cref{sub:random_permutations}.
\end{proposition}
\begin{proof}
	The Markov evolution of the stochastic vertex model
	is equivalent to a simultaneous
	exploration of a random pipe dream (that is, determining the state
	\raisebox{-2pt}{\includegraphics[scale=0.4]{cross}}
	or
	\raisebox{-2pt}{\includegraphics[scale=0.4]{elbows}}
	at each box),
	and its reduction by following the pipes.
	Indeed, in
	the evolution of the stochastic vertex model,
	in each box $(i,j)$, two strands of different colors meet as incoming pipes.
	There are two possibilities:
	\begin{enumerate}[$\bullet$]
		\item If the pipe of the lower numbered color
			is below, then these pipes have already met and crossed through each other
			(which is allowed only once in a reduced pipe dream).
			Thus, regardless of the state
			\raisebox{-2pt}{\includegraphics[scale=0.4]{cross}}
			or
			\raisebox{-2pt}{\includegraphics[scale=0.4]{elbows}}
			at $(i,j)$, the pipes \emph{must} bump off each other.
			In
			\eqref{eq:vertex_weights},
			this corresponds to the vertex weight $w_p(a,b;b,a)=1$, where $1\le a<b\le n$.
		\item
			If the pipe of the lower numbered color is to the left, then the pipes have not yet crossed
			through each other (but they may have bumped off at an elbow).
			Then, we place the tile
			\raisebox{-2pt}{\includegraphics[scale=0.4]{cross}}
			or
			\raisebox{-2pt}{\includegraphics[scale=0.4]{elbows}}
			at $(i,j)$ with probability $p$ or $1-p$, respectively.
			After placing the tile,
			the strands deterministically follow the paths drawn on it.
			The choice of the tile is equivalent
			to using the stochastic vertex weights
			$w_p(b,a;b,a)=p$ or $w_p(b,a;a,b)=1-p$, where $1\le a<b\le n$.
	\end{enumerate}
	We see by induction that the vertex model
	\eqref{eq:vertex_weights}
	produces the same random permutation
	as reducing the random pipe dream.
\end{proof}

\subsection{Height function and color forgetting}
\label{sub:color_forgetting}

Let $x,y\in \left\{ 1,\ldots,n  \right\}$.
For a fixed pipe dream $D$ and the corresponding permutation $w=w(D)$
(see \Cref{def:permutation_from_pipe_dream}),
define the \emph{permutation height function} as
\begin{equation}
	\label{eq:permutation_height_function}
	\begin{split}
		H(x,y)
		\coloneqq &
		\operatorname{\#}\bigl\{\textnormal{pipes of colors $\ge x$
		which exit through positions $j\ge y $ at the top}\bigr\}
		\\=&
		\operatorname{\#}\left(
			\left\{ w^{-1}(x),w^{-1}(x+1),\ldots,w^{-1}(n)  \right\}
			\cap
		\left\{ y,y+1,\ldots,n  \right\} \right).
	\end{split}
\end{equation}
Observe that $H(x,y)$ depends only on the permutation
$w$ and not on the structure of crossings and bumps in the pipe dream.
For a fixed $w$, the function $H(x,y)$ decreases in $x$ and $y$.
When the pipe dream and the permutation $\mathbf{w}$
are random, the height function $H(x,y)$ becomes a random variable.

In the permutation matrix of $w\in S_n$,
$H(x,y)$ is the number of entries in the rectangle
$[y,n]\times[x,n]$. Recall that the pipe of color $i$ exits at $w^{-1}(i)$, so the
rectangle is transposed.
See \Cref{fig:height_function_permutation_matrix} for an illustration.
This interpretation of $H(x,y)$ implies that
\begin{equation}
	\label{eq:height_function_bounds}
	H(x,y)\le \min(n-x,n-y).
\end{equation}

In colored stochastic six-vertex models, quantities like $H(x,y)$ are
referred to as \emph{colored height functions}, cf. \cite{borodin2020observables}.

\begin{figure}[htpb]
	\centering
	\begin{tikzpicture}[scale=.6]
  \foreach \x in {1,2,3,4,5,6} {
    \foreach \y in {1,2,3,4,5,6} {
        \draw (\x-.5,\y-.5) rectangle (\x+.5,\y+.5);
    }
  }

	\fill (1,2) circle (0.2);
  \fill (2,4) circle (0.2);
  \fill (3,1) circle (0.2);
  \fill (4,6) circle (0.2);
  \fill (5,5) circle (0.2);
  \fill (6,3) circle (0.2);

  \foreach \i in {1,2,3,4,5,6} {
    \node at (\i,0) {\i};
    \node at (0,\i) {\i};
  }
	\draw[blue,opacity=.3, fill] (2.5,3.5)--++(4,0)--++(0,3)--++(-4,0)--cycle;
\end{tikzpicture}
	\caption{Permutation matrix of $w=(2,4,1,6,5,3)$
		coming from the pipe dream in
		\Cref{fig:pipe_dream}, right (dots indicate 1's).
		The highlighted rectangle has
		$H(4,3)=2$ entries.}
	\label{fig:height_function_permutation_matrix}
\end{figure}
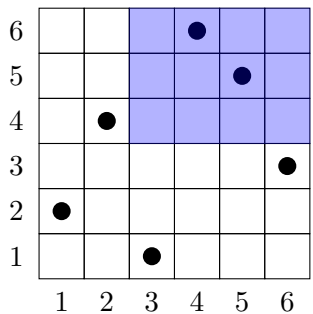

\begin{remark}
	\label{rmk:color_position}
	The color-position symmetry \cite{BorodinBufetov2021ColorPosition}
	of the colored stochastic six-vertex model
	implies that the distribution of $H(x,y)$ corresponding to the
	Grothendieck random permutation $\mathbf{w}$
	is symmetric in $x,y$.
	Indeed, by the color-position symmetry, $\mathbf{w}$ and $\mathbf{w}^{-1}$
	have the same distribution. Therefore,
	\begin{equation*}
		H(x,y)
		=
		\sum_{i=x}^{n}\sum_{j=y}^{n}\mathbf{1}_{\mathbf{w}^{-1}(i)=j}
		\stackrel{d}{=}
		\sum_{i=x}^{n}\sum_{j=y}^{n}\mathbf{1}_{\mathbf{w}(i)=j}
		= H(y,x).
	\end{equation*}
	However, we can access the distribution of $H(x,y)$ directly
	using color forgetting (\Cref{prop:color_forgetting} below), and do not
	rely on
	color-position symmetry.
\end{remark}

Consider an \emph{uncolored} (\emph{color-blind}) stochastic
vertex model
with all pipes of the same color.
We indicate pipes and empty edges by $0$ and $1$, respectively.
The weights of the color-blind model are defined as
\begin{equation}
	\label{eq:w_color_blind}
	\begin{split}
		w_p^{\bullet}(0,0;0,0)&=1,\qquad w_p^{\bullet}(1,1;1,1)=1,
		\\
		w_p^{\bullet}(1,0;1,0)&=p,\qquad
		w_p^{\bullet}(1,0;0,1)=1-p,
		\\
		w_p^{\bullet}(0,1;0,1)&=0,\qquad
		w_p^{\bullet}(0,1;1,0)=1.
	\end{split}
\end{equation}

\begin{proposition}
	\label{prop:color_forgetting}
	Fix $1\le x\le n$. In the colored stochastic vertex model of
	\Cref{sub:colored_six_vertex,sub:matching_pipe_dreams},
	erase all pipes of colors $<x$,
	and identify all the remaining pipes for colors $\ge x$.
	The resulting random configuration of
	uncolored pipes evolves
	according to
	the color-blind stochastic
	vertex model $w_p^{\bullet}$ \eqref{eq:w_color_blind}.

\end{proposition}
\begin{proof}
	This is a standard property of the colored stochastic six-vertex model,
	see \cite[Section~2.4]{borodin_wheeler2018coloured} and references therein.
	In interacting particle systems,
	this property is known as the basic coupling, cf.
	\cite{liggett1976coupling}, \cite{Liggett1985}.
	Let us provide an idea of the proof.

	The crucial property is
	that the weights $w_p$ \eqref{eq:vertex_weights}
	depend only on the relative order of colors of the incoming pipes, and not the actual colors.
	Therefore, when we identify all colors $\ge x$, the resolution of the pipe
	crossings does not affect the behavior of the color-blind system.
	In detail, one can consider four cases
	for the incoming pipes $(a,b)$,
	$1\le a,b\le n$,
	depending on whether $a\ge x$ and/or $b\ge x$.
	The resulting probability distribution on possible outputs,
	$(a,b)$ or $(b,a)$, depends only on these cases, and not on
	the actual values of~$a,b$.
	These four cases correspond to the color-blind incoming pipes
	in $w_p^{\bullet}$ \eqref{eq:w_color_blind}.
\end{proof}

By \Cref{prop:color_forgetting},
the permutation height function $H(x,y)$
can be identified with an observable of the model
$w_p^{\bullet}$ with initial occupied configuration
$\left\{ x,x+1,\ldots,n   \right\}$ along the left boundary
(that is, the sites $\{1,\ldots,x-1\}$ are empty).
Indeed, in this color-blind model, $H(x,y)$ is simply
the
number of pipes exiting through the top boundary at positions $\ge y$.
We see that under the color forgetting, the color parameter $x$ in $H(x,y)$ became a parameter of the initial configuration.

\begin{remark}
	One can similarly forget the colors $>x$, and the resulting
	color-blind system evolves according to different
	weights (obtained from $w_p^{\bullet}$ by swapping $0\leftrightarrow 1$).
	It is more convenient for us to work with the model $w_p^{\bullet}$,
	which we directly relate
	to an interacting particle system in the next \Cref{sub:identification_vertex_TASEP}.
\end{remark}

\section{From vertex models to TASEP}
\label{sec:six_vertex_TASEP}

\subsection{TASEP with moving exit boundary}
\label{sub:TASEP_exit}

\begin{definition}
	\label{def:TASEP}
	Let $k\ge 1$.
	Let
	$\xi(t)\coloneqq \bigr(\xi_1(t)>\ldots>\xi_k(t) \bigr)\subset \mathbb{Z}$
	be the
	$k$-particle
	discrete time TASEP (Totally Asymmetric Simple Exclusion Process)
	having
	parallel updates
	and geometrically distributed jumps.
	In detail, $\xi(t)$, $t\in \mathbb{Z}_{\ge0}$, is a discrete time Markov chain
	on particle configurations in $\mathbb{Z}$ which at each time step $t\to t+1$ evolves as follows:
	\begin{equation}
		\label{eq:TASEP_parallel_def}
		\xi_i(t+1)=\xi_i(t)+\min \left( G_{i}(t+1), \xi_{i-1}(t)-\xi_i(t)-1 \right),\qquad
		1\le i\le k,
	\end{equation}
	where $G_i(t+1)$ are independent geometric random variables with parameter $p\in(0,1)$,
	that is,
	\begin{equation}
		\label{eq:geometric_random_variable}
		\Prob
		\left( G_i(t+1)=m \right)=(1-p)\ssp p^m,\qquad m\in \mathbb{Z}_{\ge0}.
	\end{equation}
	The update \eqref{eq:TASEP_parallel_def} occurs in parallel for all
	particles $1\le i\le k$, that is,
	the new positions $\xi_i(t+1)$ depend only on the configuration
	$\xi(t)$ at the previous time step, and new independent random variables.
	By agreement, we have $\xi_0(t)\equiv +\infty$,
	so that the first particle $\xi_1(t)$ performs an independent random walk
	with geometrically distributed jumps.
	See \Cref{fig:TASEP} for an illustration.
\end{definition}

\begin{figure}[htpb]
	\centering
	\begin{tikzpicture}
		[scale=1, very thick]
		\draw[->] (-1,0)--++(10,0) node [above right] {$\mathbb{Z}$};
		\foreach \x in {0,1,2,3,4,5,6,7,8}
		{
			\draw (\x,0.1)--++(0,-0.2);
		}
		\foreach \x in {0,2,3,8}
		{
			\draw[fill] (\x,0) circle [radius=0.15];
		}
		\node at (8,-0.5) {$\xi_1$};
		\node at (3,-0.5) {$\xi_2$};
		\node at (2,-0.5) {$\xi_3$};
		\node at (0,-0.5) {$\xi_4$};
		\draw[->] (0,0) to [out=90,in=-180] ++(.5,.7) to[out=0,in=90] ++(.5,-.5);
		\draw[->] (3,0) to [out=90,in=-180] ++(1.5,.7+.4) to[out=0,in=90] ++(1.5,-.5-.4);
	\end{tikzpicture}
	\caption{One step of the TASEP with parallel updates and geometric jumps.
		Here the independent geometric random variables could be
		$G_1=0$, $G_2=3$, $G_3=1$, $G_4=2$ (there is more than one choice
		of the $G_i$'s leading to the same update).
		For these $G_i$'s, the jumps of $\xi_3$ and $\xi_4$ are blocked by the preceding particles.}
	\label{fig:TASEP}
\end{figure}
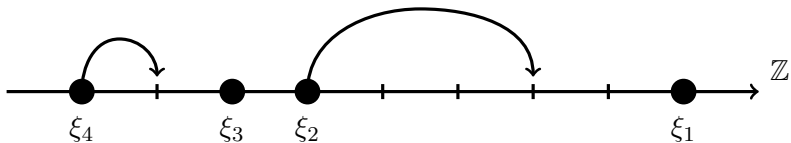

Start the TASEP from
the \emph{densely packed} (also called \emph{step})
initial configuration
$\left\{ 1,2,\ldots,k  \right\}$, that is,
$\xi_i(0)=k+1-i$, $1\le i\le k$.
Fix $n\ge k$, and introduce a \emph{moving exit boundary}
which starts at $n+\frac12$, and deterministically
moves to the left with speed $1$.
The \emph{exit time} of a particle $\xi_i(t)$ is defined by
\begin{equation}
	\label{eq:exit_time}
	T_{\mathrm{exit}}(i)\coloneqq \min\left\{ t\colon \xi_i(t)\ge n+1-t \right\}.
\end{equation}
For example, if $n=k$, then the first particle starting at
$n$ exits at time $t=1$ (but have not exited at the initial
time $t=0$). The exit times are almost surely ordered:
\begin{equation}
	\label{eq:exit_times_are_ordered}
	1\le T_{\mathrm{exit}}(1)<T_{\mathrm{exit}}(2)<\ldots<T_{\mathrm{exit}}(k)\le n.
\end{equation}
Note that tor any $n\ge k$,
all particles have exited the system by the time $t=n$.
See \Cref{fig:TASEP_exit_boundary,fig:TASEP_sims} for an illustration and simulations.

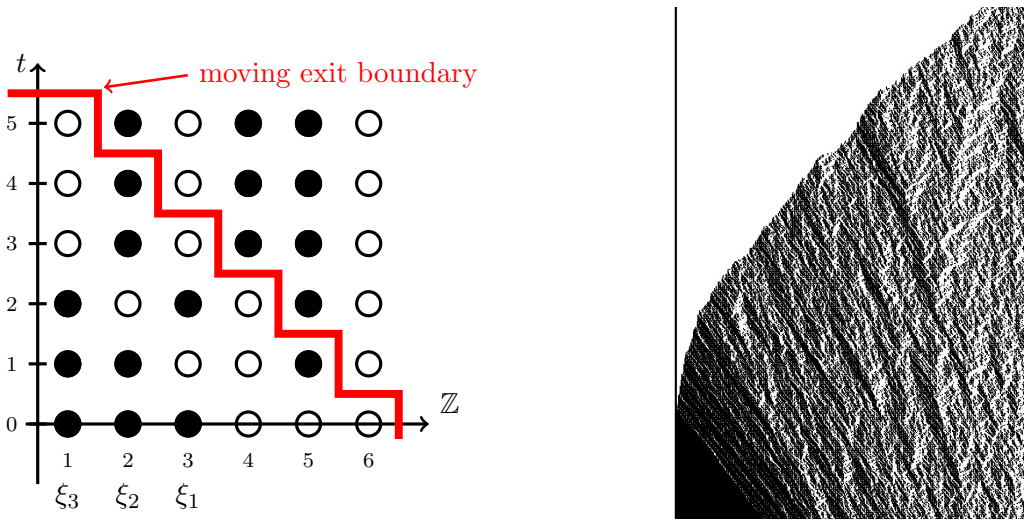
\begin{figure}[htpb]
	\centering
	\begin{tikzpicture}
		[scale=.8, very thick]
		\draw[->] (.3,0)--++(6.7,0) node [above right] {$\mathbb{Z}$};
		\draw[->] (.5,-1)--++(0,7) node [left] {$t$};
		\foreach \t in {0,1,2,3,4,5}
		{
		\draw (.65,\t)--++(-.3,0) node [left] {\scriptsize$\t$};
		\foreach \x in {1,2,3,4,5,6}
			{
				\draw (\x,\t) circle[radius=0.2];
			}
		}
		\draw[red,line width=3pt] (6.5,-.25)--++(0,.75)
		--++(-1,0)--++(0,1)
		--++(-1,0)--++(0,1)
		--++(-1,0)--++(0,1)
		--++(-1,0)--++(0,1)
		--++(-1,0)--++(0,1)
		--++(-1.5,0)
		;
		\node at (1,-1.2) {$\xi_3$};
		\node at (2,-1.2) {$\xi_2$};
		\node at (3,-1.2) {$\xi_1$};
		\node at (1,-.6) {\scriptsize$1$};
		\node at (2,-.6) {\scriptsize$2$};
		\node at (3,-.6) {\scriptsize$3$};
		\node at (4,-.6) {\scriptsize$4$};
		\node at (5,-.6) {\scriptsize$5$};
		\node at (6,-.6) {\scriptsize$6$};
		\foreach \pt in {
			(3,0),(2,0),(1,0),
			(5,1),(2,1),(1,1),
			(5,2),(3,2),(1,2),
			(5,3),(4,3),(2,3),
			(5,4),(4,4),(2,4),
			(5,5),(4,5),(2,5)
		}
		{
			\draw[fill] \pt circle[radius=0.2];
		}
		\node[red] (l) at (5.5,5.8) {moving exit boundary};
		\draw[red,line width=1pt,->] (l.west) -- (1.6,5.6);
	\end{tikzpicture}
	\qquad \qquad \qquad
	\includegraphics[height=.42\textwidth]{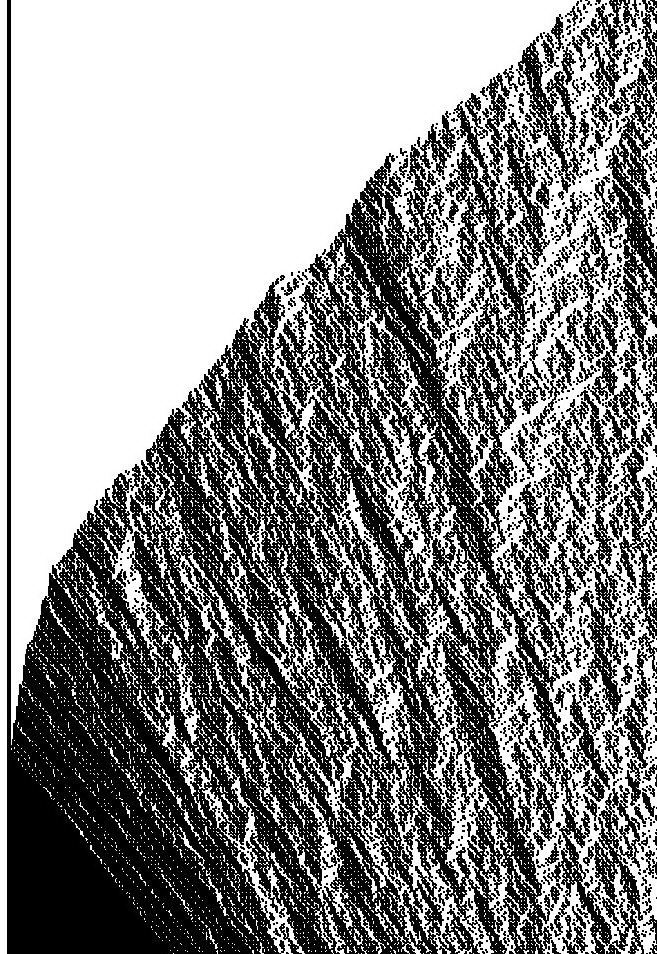}
	\caption{\textbf{Left}:
		Trajectory of TASEP with moving exit boundary
		for $k=3$ and $n=6$.
		Particles do not jump anymore after they exit
		(though
		the behavior beyond the boundary
		is irrelevant).
		The exit times are
		$T_{\mathrm{exit}}(1)=2$,
		$T_{\mathrm{exit}}(2)=3$, and
		$T_{\mathrm{exit}}(3)=5$.
		\textbf{Right}: A fragment (close to the left boundary) of the simulation with
		$n=2000$, $k=400$, and
		$p=0.8$.
		The last particle $\xi_k(t)$ globally follows a nonlinear trajectory
		that is a quadratic parabola near the point of tangence with the left boundary.}
	\label{fig:TASEP_exit_boundary}
\end{figure}

\begin{figure}[htpb]
	\centering
	\includegraphics[height=.45\textwidth]{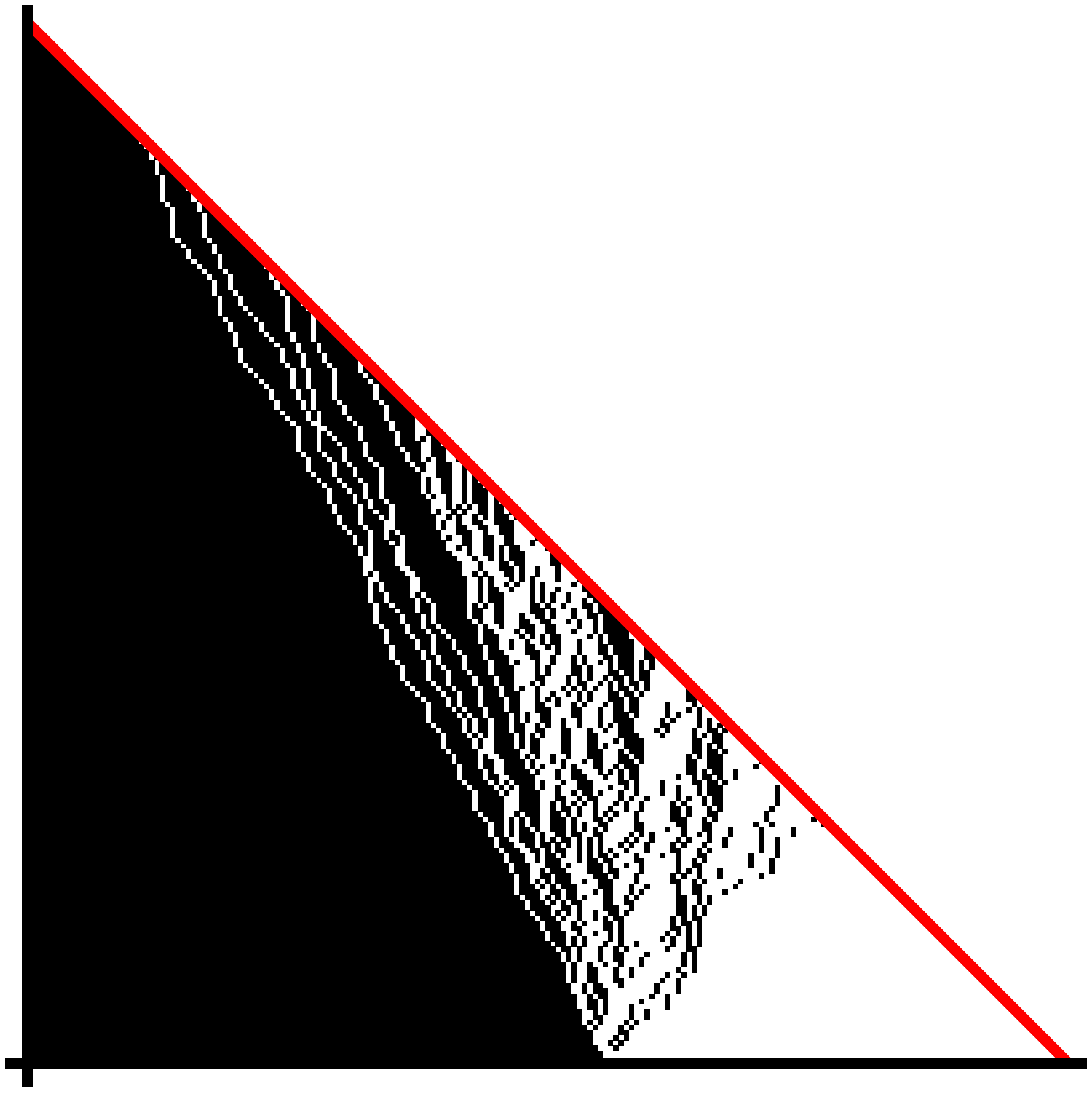}
	\qquad
	\includegraphics[height=.45\textwidth]{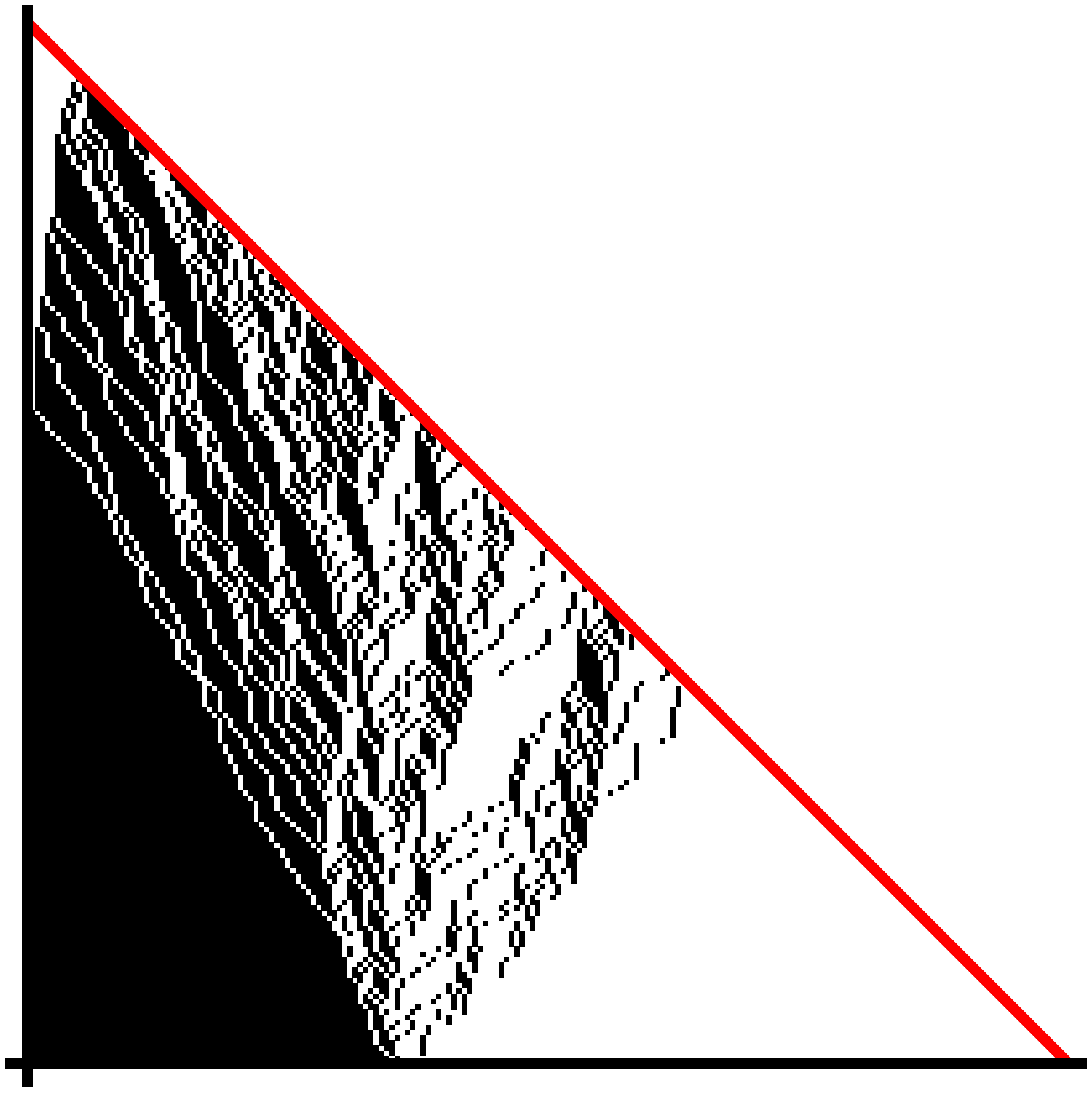}
	\caption{Simulations
		of TASEP with moving exit boundary for $n=200$, $p=0.5$,
		and $k=110$ (\textbf{left}) or $k=70$ (\textbf{right}).
		Particles beyond the diagonal exit boundary are not shown.
		In both cases, the first particle $\xi_1(t)$ (performing a random walk)
		globally follows a linear trajectory.
		In the left simulation, the last particle $\xi_k$ did not
		have time to move before the exit boundary caught up with it.
		In the right simulation, the last particle $\xi_k$ follows a nonlinear trajectory.
		See also \Cref{fig:TASEP_exit_boundary}, right, for a close-up in a larger simulation.}
	\label{fig:TASEP_sims}
\end{figure}

Integrability of TASEP with parallel updates, geometric jumps,
and densely packed initial configuration
can be traced back to
\cite{Vershik1986}, see also
\cite{draief2005queues},
\cite[Case C]{diekerWarren2008determinantal},
and \cite{BorFerr2008DF}.
We recall the necessary
integrability
results in \Cref{sub:TASEP_integrability} below
after connecting TASEP to vertex models and
Grothendieck random permutations
in the next \Cref{sub:identification_vertex_TASEP}.

\subsection{Identification with vertex models}
\label{sub:identification_vertex_TASEP}

Fix $n$ and $1\le x,y\le n$. Recall the
observables $H(x,y)$
defined by \eqref{eq:permutation_height_function}.
That is, $H(x,y)$ are random variables which are functions of the Grothendieck
random permutation (equivalently, of the colored stochastic six-vertex model).
In \Cref{prop:color_forgetting},
we identified $H(x,y)$ with observables of a color-blind vertex model with weights
$w_p^{\bullet}$ \eqref{eq:w_color_blind}.

\begin{theorem}
	\label{thm:vertex_TASEP_matching}
	For any $1\le x,y\le n$ and $0\le h\le n-x+1$, we have
	\begin{equation}
		\label{eq:vertex_TASEP_matching_theorem}
		\begin{split}
			\Prob_{\mathbf{w}}\bigl(H(x,y)\le h\bigr)
			&=
			\Prob_{\mathrm{TASEP}}\bigl(T_{\mathrm{exit}}(n-x+1-h)\le y-1\bigr)
			\\&=
			\Prob_{\mathrm{TASEP}}\bigl(
				\xi_{n-x+1-h}(y-1)\ge n-y+2
			\bigr).
		\end{split}
	\end{equation}
	Here
	$\Prob_{\mathbf{w}}$
	corresponds to the Grothendieck random permutation of order $n$,
	and
	$\Prob_{\mathrm{TASEP}}$
	is the probability distribution of the TASEP
	with $k=n-x+1$ particles and moving exit boundary.
	In \eqref{eq:vertex_TASEP_matching_theorem}
	we have, by agreement,
	$T_{\mathrm{exit}}(m)= 0$ and
	$\xi_m(t)= +\infty$ for all $m\le 0$.
\end{theorem}
\begin{proof}
	The second identity in \eqref{eq:vertex_TASEP_matching_theorem}
	immediately follows from the definition of the exit times
	\eqref{eq:exit_time}. Let us focus on the first identity.

	By \Cref{prop:color_forgetting}, let us consider the color-blind vertex model
	with $k=n-x+1$ identical pipes entering the left boundary
	of the staircase shape $\updelta_n$
	at locations $\left\{ x,x+1,\ldots,n  \right\}$.
	The event $\left\{ H(x,y)\le h \right\}$ means that
	at most $h$ of these pipes exit the top boundary at positions $\ge y$.
	Note that the color-blind model has indistinguishable pipes, and thus the ``resolution''
	of their crossings
	does not change the behavior of the system (see the proof of \Cref{prop:color_forgetting} for more detail).

	View the horizontal coordinate $j$
	as time $t=0,1,\ldots,n $,
	and record the pipes' coordinates as
	\begin{equation*}
		\eta_1(t)<\ldots<\eta_k(t),\qquad t=0,1,\ldots,n .
	\end{equation*}
	The initial condition is $\eta_{k-m+1}(0)=n-m+1$, $1\le m\le k$.

	At each time step, according to the $w_p^{\bullet}$ \eqref{eq:w_color_blind},
	each pipe first deterministically turns up.
 	When a pipe faces up, it can travel up a further
	distance distributed as a geometric random variable with parameter $p$,
	see \eqref{eq:geometric_random_variable}.
	However, during a time step $t\to t+1$, each pipe $\eta_m$ cannot
	travel further than the previous pipe' location $\eta_{m-1}(t)$
	(this is the exclusion mechanism).
	When a pipe stops moving up, it turns right and waits till the next time step.
	Once a pipe reaches the top boundary, it exits the system.
	See \Cref{fig:vertex_TASEP_matching}, left, for an illustration.

	\begin{figure}[htpb]
		\centering
		\includegraphics[height=.35\textwidth]{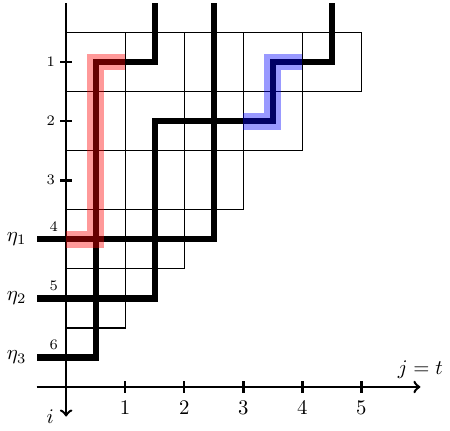}
	\qquad \qquad
	\scalebox{.9}{
	\begin{tikzpicture}
		[scale=.8, very thick]
		\draw[->] (.3,0)--++(6.7,0) node [above right] {$\mathbb{Z}$};
		\draw[->] (.5,-1)--++(0,7) node [left] {$t$};
		\foreach \t in {0,1,2,3,4,5}
		{
		\draw (.65,\t)--++(-.3,0) node [left] {\scriptsize$\t$};
		}
		\foreach \pt in {(1,1),(1,2),(1,3),(1,4),(1,5),
			(2,1),(2,2),(3,1),(3,2),(3,3),
			(4,0),(4,1),(4,2),
			(5,0),(5,1),
			(6,0)
			}
		{
      \draw \pt circle[radius=0.2];
		}
		\draw[red,line width=3pt] (6.5,-.25)--++(0,.75)
		--++(-1,0)--++(0,1)
		--++(-1,0)--++(0,1)
		--++(-1,0)--++(0,1)
		--++(-1,0)--++(0,1)
		--++(-1,0)--++(0,1)
		--++(-1.5,0)
		;
		\node at (1,-1.2) {$\xi_3$};
		\node at (2,-1.2) {$\xi_2$};
		\node at (3,-1.2) {$\xi_1$};
		\node at (1,-.6) {\scriptsize$1$};
		\node at (2,-.6) {\scriptsize$2$};
		\node at (3,-.6) {\scriptsize$3$};
		\node at (4,-.6) {\scriptsize$4$};
		\node at (5,-.6) {\scriptsize$5$};
		\node at (6,-.6) {\scriptsize$6$};
		\foreach \pt in {
			(3,0),(2,0),(1,0),
			(2,1),(1,1),
			(1,2),
			(2,3),(3,2),
			(2,4),(5,1)
		}
		{
			\draw[fill] \pt circle[radius=0.2];
		}
		\draw [line width=8pt,opacity=0.4,red] (3,0)--++(2,1);
		\draw [line width=8pt,opacity=0.4,blue] (2,3)--++(0,1);
	\end{tikzpicture}}
	\caption{\textbf{Left}: The evolution of the uncolored pipes $\eta(t)$.
		Here $n=6$ and $x=4$, so $k=3$. At time $t=2$, we have
		$\eta_2(2)=2$, $\eta_3(2)=4$, and the pipe $\eta_1$ has exited before $t=2$.
		We have $H(4,3)=2$. \textbf{Right}:
		The evolution of the process $\xi(t)$
		which is in bijection
		with the pipe configuration on the left (that is, $\xi_m(t)= n+1-t-\eta_m(t)$).
		In detail, a move of the particle $\xi_\ell$ by $r\ge0$ steps
		corresponds to the pipe moving $r+1$ steps up (at the same time increment).
		Two pairs of corresponding moves are highlighted.}
		\label{fig:vertex_TASEP_matching}
	\end{figure}
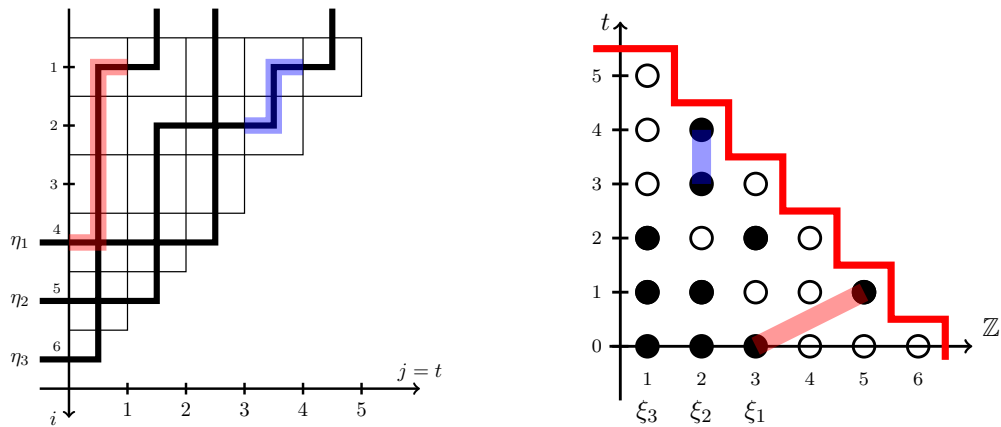

	Setting $\xi_m(t)\coloneqq n+1-t-\eta_m(t)$,
	one readily verifies that the evolution of $\xi_1(t)> \ldots > \xi_k(t)$
	is the same as that of the TASEP
	with $k=n-x+1$ particles
	defined in
	\Cref{sub:TASEP_exit}.
	In particular, subtracting $t$ from $n+1-\eta_m(t)$
	eliminates the deterministic up movement of the pipes by one at each time step.

	The top boundary for pipes (as in \Cref{fig:vertex_TASEP_matching}, left)
	becomes the moving exit boundary for TASEP. The event
	$\{H(x,y)\le h\}$ for pipes is equivalent to the event
	that at least $k-h=n-x+1-h$ of the TASEP particles
	have exited before time $y-1$. Because the
	exit times are ordered as in~\eqref{eq:exit_times_are_ordered},
	we get the desired first identity in \eqref{eq:vertex_TASEP_matching_theorem}.
\end{proof}

\subsection{Integrability of TASEP}
\label{sub:TASEP_integrability}

Fix $m\ge 0$, and let $\lambda=(\lambda_1\ge \ldots\ge \lambda_m\ge 0 )$
be a partition with at most $m$ parts.
The \emph{Schur symmetric polynomial} in $m$ variables indexed by $\lambda$ is defined as
\begin{equation*}
	s_\lambda(a_1,\ldots,a_m)\coloneqq
	\frac{\det[a_i^{\lambda_j+m-j}]_{i,j=1}^m}{\det[a_i^{m-j}]_{i,j=1}^m},
\end{equation*}
where $a_i$ are variables. The denominator is the Vandermonde determinant
$\prod_{1\le i<j\le m}(a_i-a_j)$.

Let $m,t\ge1$.
The \emph{Schur measure} \cite{okounkov2001infinite}
is a probability measure on partitions with at most $\min(m,t)$ parts
depending on parameters $a_1,\ldots,a_m,b_1,\ldots,b_t$. Its probability weights
are defined as
\begin{equation}
	\label{eq:Schur_measure}
	\Prob(\lambda)\coloneqq
	\prod_{i=1}^{m}
	\prod_{j=1}^{t}
	(1-a_ib_j)
	\cdot
	s_\lambda(a_1,\ldots,a_m)\ssp
	s_\lambda(b_1,\ldots,b_t).
\end{equation}
When $m=t=1$, the weights \eqref{eq:Schur_measure} define the geometric distribution
on $\lambda=(\lambda_1\ge0)$ with parameter $a_1b_1$.
In general, the form of the normalizing constant in \eqref{eq:Schur_measure}
follows from the Cauchy summation identity for Schur polynomials.
The infinite sum
for the normalizing constant runs
over all $\lambda$ of length $\le \min(m,t)$ and converges if $|a_ib_j|<1$ for all $i,j$.

The next statement connects TASEP to Schur measure,
which provides the necessary integrability structure for the former.

\begin{proposition}
	\label{prop:TASEP_Schur_measure}
	For any $t,m\ge1$, the displacement of the $m$-th particle
	in the TASEP with parallel updates, geometric jumps with parameter $p$,
	and densely packed initial configuration (as in \Cref{sub:TASEP_exit})
	has the same distribution as $\lambda_m$, the last part of a partition
	under the Schur measure
	\eqref{eq:Schur_measure}
	with $a_1=\ldots=a_m=1$ and $b_1=\ldots=b_t=p$:
	\begin{equation}
		\label{eq:TASEP_Schur_measure_concrete}
		\xi_m(t)-\xi_m(0)\stackrel{d}{=}\lambda_m,
		\qquad
		\lambda \sim
		(1-p)^{m t} s_\lambda(\underbrace{1,\ldots,1 }_m)\ssp
		s_\lambda(\underbrace{p,\ldots,p }_t).
	\end{equation}
\end{proposition}

\Cref{prop:TASEP_Schur_measure} is a well-known
result in Integrable Probability.
It follows either from the
Robinson--Schensted--Knuth (RSK) correspondence with column insertion,
or by representing the TASEP as a
marginally Markovian evolution of a process on the space of semistandard Young tableaux
(equivalently, interlacing arrays / Gelfand--Tsetlin patterns).

For the simpler continuous time TASEP (when each particle waits an exponentially
distributed time and jumps to the right by $1$ if the destination
is unoccupied), the connection to Schur measures is
exploited in the celebrated work \cite{johansson2000shape}
via last-passage percolation, which is related to the row RSK.
The RSK approach can be traced back to \cite{Vershik1986}.
Dynamics on interlacing arrays are constructed later in
\cite{BorFerr2008DF} using a non-RSK approach.

For our discrete time TASEP with geometric jumps,
the RSK was explicitly used in
\cite[Section 5]{draief2005queues}
to establish \Cref{prop:TASEP_Schur_measure}.
A systematic treatment of discrete time
dynamics on interlacing arrays connected to Schur measures
leading to various TASEPs and PushTASEPs may be found in
\cite[Section~4]{MatveevPetrov2014}.
It includes
RSK-type and other dynamics on interlacing arrays under one roof,
as well as a generalization of the Schur structure to the level of $q$-Whittaker
symmetric polynomials.

\subsection{Determinantal structure}
\label{sub:determinantal_structure}

The connection
to Schur measures (\Cref{prop:TASEP_Schur_measure})
allows to extract law of large numbers and
fluctuation results for the TASEP particles.
The key property of Schur measures which makes this possible is
the fact that they form a \emph{determinantal point process}.

Fix $m,t\ge 1$, and define a random point configuration
$X(\lambda)\coloneqq \{\lambda_j+m-j\}_{j=1}^{m}$ on $\mathbb{Z}_{\ge0}$,
where $\lambda$ is distributed as \eqref{eq:TASEP_Schur_measure_concrete}.
If $\lambda$ has less parts than $m$, append it by zeros.
These zeros translate into a part of $X(\lambda)$ of the form
$\{0,1,\ldots,l \}$ for some $l$.

\begin{proposition}
	\label{prop:det_Schur}
	The random point configuration $X(\lambda)$ on $\mathbb{Z}_{\ge0}$ is a
	determinantal point process. This means that for
	any $r\ge1$ and pairwise distinct points $u_1,\ldots,u_r \in \mathbb{Z}_{\ge0}$,
	the correlation functions have the form
	\begin{equation}
		\label{eq:determinantal_process_definition}
		\Prob(\textnormal{$X(\lambda)$ contains all of the points $u_1,\ldots,u_r $} )=
		\det\left[ K(u_i,u_j) \right]_{i,j=1}^{r}.
	\end{equation}
	The correlation kernel $K$ is given by a double contour integral:
	\begin{equation}
		\label{eq:kernel_Schur}
		K(u_1,u_2)=\frac{1}{(2\pi\mathbf{i})^2}\oiint \frac{dz \ssp dw}{z-w}\frac{w^{u_2-m}}{z^{u_1-m+1}}
		\frac{(1-p/z)^m}{(1-z)^t}\frac{(1-w)^t}{(1-p/w)^m},
		\qquad
		u_1,u_2\in \mathbb{Z}_{\ge0},
	\end{equation}
	where the contours are positively oriented simple closed curves satisfying
	$p<|w|<|z|<1$.
\end{proposition}

\Cref{prop:det_Schur} is due to
\cite{okounkov2001infinite},
see also
\cite{borodin2005eynard}.
General discussions and many examples of determinantal point processes
may be found
in \cite{Soshnikov2000}, \cite{Borodin2009},
or
\cite{peres2006determinantal}.

\subsection{Asymptotics of TASEP}
\label{sub:TASEP_asymptotics}

Determinantal structure powers the asymptotic behavior of
the TASEP particles.
We are interested in
probabilities of the form
\begin{equation}
	\label{eq:xi_bar_and_TASEP_probability}
	\Prob_{\mathrm{TASEP}}(\bar\xi_{m}(t)\ge u),
	\qquad
	\textnormal{where}\quad
	\bar \xi_m(t)\coloneqq \xi_m(t)-\xi_m(0).
\end{equation}
Indeed, since
\Cref{thm:vertex_TASEP_matching}
reduces the distribution of $H(x,y)$ to
probabilities of the form \eqref{eq:xi_bar_and_TASEP_probability},
we can forget about the moving exit boundary in TASEP.
Indeed, the parameter $n$ of the moving boundary is absorbed into $m$ and $u$.

\medskip

Let the parameters
$t,m,u$ grow to infinity proportionally to each other:
\begin{equation}
	\label{eq:parameter_scaling}
	m=\lfloor L\ssp \mathsf{m} \rfloor , \qquad
	t=\lfloor L\ssp \mathsf{t} \rfloor , \qquad
	u=\lfloor L\ssp \mathsf{u} \rfloor ,
	\qquad  L\to \infty.
\end{equation}
First, let us heuristically discuss the law of large numbers and fluctuations
of
$\Prob_{\mathrm{TASEP}}(\bar\xi_{m}(t)\ge u)$,
without providing exact constants.
The latter are given in \Cref{def:TASEP_constants} below,
and we formulate the asymptotic results in \Cref{thm:TASEP_limit}.
The proof of \Cref{thm:TASEP_limit} using the determinantal structure
is a standard application of the steepest descent method for double contour integrals.
We provide the computations (leading to the constants in \Cref{def:TASEP_constants})
in
\Cref{sec:app_TASEP_asymptotics}.

\medskip

At the ``hydrodynamic'' scale \eqref{eq:parameter_scaling},
TASEP has a limit shape
(first observed in the continuous time TASEP in \cite{Rost1981}, see also
\cite{aldous1995hammersley}). That is, there exists
a function $\mathsf{c}(\mathsf{m},\mathsf{t})$ such that
for any $\varepsilon>0$, we have
\begin{equation}
	\label{eq:TASEP_hydrodynamic}
	\lim_{L\to\infty}
	\Prob_{\mathrm{TASEP}}
	\bigl(\bar\xi_{m}(t)\ge L\ssp \mathsf{c}(\mathsf{m},\mathsf{t})+\varepsilon\bigr)
	=0,
	\qquad
	\lim_{L\to\infty}
	\Prob_{\mathrm{TASEP}}
	\bigl(\bar \xi_{m}(t)\ge L\ssp \mathsf{c}(\mathsf{m},\mathsf{t})-\varepsilon\bigr)
	=1.
\end{equation}
In other words, we have convergence in probability:
\begin{equation}
	\label{eq:TASEP_LLN}
	\lim_{L\to\infty}
	L^{-1}\ssp
	\bar\xi_{\lfloor L\ssp \mathsf{m} \rfloor }(\lfloor L\ssp \mathsf{t} \rfloor )
	=\mathsf{c}(\mathsf{m},\mathsf{t}).
\end{equation}
Let us note a certain a priori property of the limit shape:
\begin{lemma}
	\label{lemma:a_priori_zero_of_limit_shape}
	If $\mathsf{t}<\mathsf{m}/p$, we have $\mathsf{c}(\mathsf{m},\mathsf{t})=0$.
\end{lemma}
\begin{proof}
	This follows from the fact that the leftmost hole in the TASEP configuration
	(which starts immediately to the right of $\xi_1(0)$ at time $t=0$)
	travels left by at most $1$ at each time step.
	The probability that the hole moves is $p=\Prob(G>0)$,
	where $G$ is the geometric random variable \eqref{eq:geometric_random_variable}.
	Therefore, if $\mathsf{t}<\mathsf{m}/p$, then the hole has not reached the
	$m$-th particle by time $t$ with probability $1-e^{-C\ssp L}$ for some $C>0$.
	This implies the claim.
\end{proof}

Next, since TASEP belongs to the Kardar--Parisi--Zhang (KPZ) universality class
\cite{CorwinKPZ}, we expect
fluctuations of order $L^{1/3}$ provided that $\mathsf{t}>\mathsf{m}/p$.
More precisely,
modifying $m$ and $u$ on the scale
$L^{\frac13}$ probes the scaling window in the law of large numbers
\eqref{eq:TASEP_hydrodynamic}. That is,
there exist constants
$\mathsf{v}_1(\mathsf{m},\mathsf{t}),
\mathsf{v}_2(\mathsf{m},\mathsf{t})>0$, such that
\begin{equation}
	\label{eq:TASEP_fluctuations_result}
	\lim_{L\to\infty}
	\Prob_{\mathrm{TASEP}}
	\left( \bar\xi_{\lfloor L\ssp \mathsf{m}- L^{1/3}\ssp \alpha\ssp \mathsf{v}_1(\mathsf{m},\mathsf{t})\rfloor}
		(\lfloor L\ssp \mathsf{t} \rfloor )\ge
	L\ssp \mathsf{c}(\mathsf{m},\mathsf{t})-
	L^{1/3}\ssp \beta \ssp \mathsf{v}_2(\mathsf{m},\mathsf{t})\right)
	=F_2\left( \alpha +\beta\right)
\end{equation}
for all $\alpha,\beta\in\mathbb{R}$.
Here, $F_2$ denotes the cumulative distribution function of
the Tracy--Widom GUE distribution. The plus signs
by $\alpha$ and $\beta$ in the right-hand side of \eqref{eq:TASEP_fluctuations_result}
are straightforward from the monotonicity of the pre-limit probabilities.

The distribution function $F_2$
was first discovered
in connection with the fluctuations of the largest
eigenvalue in large complex Hermitian random matrices
exhibiting unitary symmetry \cite{tracy1993level}.
Subsequently,
$F_2$ was put into a larger class of limiting distributions arising in random
growth models and interacting particle systems, known as the KPZ universality class.
We refer to the surveys \cite{CorwinKPZ},
\cite{QuastelSpohnKPZ2015},
\cite{halpin2015kpzCocktail} for a detailed exposition.
The appearance of $F_2$ in the fluctuations (out of several
candidates in the KPZ universality class)
is a feature of the step initial
configuration in our TASEP.

\begin{definition}
	\label{def:TASEP_constants}
	Let $\mathsf{m},\mathsf{t}>0$. Define
	\begin{equation}
		\label{eq:TASEP_c_constant_LLN}
		\mathsf{c}(\mathsf{m},\mathsf{t})
		\coloneqq
		\begin{cases}
			0,
			& \textnormal{if $\mathsf{t}\le \mathsf{m}/p$};
			\\
			\displaystyle\frac{(\sqrt{p\ssp \mathsf{t}}-\sqrt{\mathsf{m}})^2}{1-p},
			& \textnormal{if $\mathsf{t}\ge\mathsf{m}/p$}.
		\end{cases}
	\end{equation}
	If $\mathsf{t}>\mathsf{m}/p$, set
	\begin{equation}
		\label{eq:TASEP_fluctuations_constants}
		\mathsf{v}_1(\mathsf{m},\mathsf{t})\coloneqq
		\frac{\sqrt{p}\ssp\mathsf{m}^{1/3}}{\mathsf{t}^{1/6}}
		\frac{(\sqrt{\mathsf{t}/p}-\sqrt{\mathsf{m}})^{2/3}}
		{(\sqrt{p\mathsf{t}}-\sqrt{\mathsf{m}})^{1/3}},
		\qquad
		\mathsf{v}_2(\mathsf{m},\mathsf{t})\coloneqq
		\frac{\sqrt{p}\ssp
		(\sqrt{p\mathsf{t}}-\sqrt{\mathsf{m}})^{2/3}
		(\sqrt{\mathsf{t}/p}-\sqrt{\mathsf{m}})^{2/3}
		}{(\mathsf{m}\mathsf{t})^{1/6}(1-p)}
		.
	\end{equation}
\end{definition}

\begin{theorem}
	\label{thm:TASEP_limit}
	Let the constants $\mathsf{c}(\mathsf{m},\mathsf{t})$ and $\mathsf{v}_{1,2}(\mathsf{m},\mathsf{t})$ be
	given in \Cref{def:TASEP_constants}. Then
	\begin{enumerate}[$\bullet$]
		\item
			The law of large numbers \eqref{eq:TASEP_LLN} holds for all
			$\mathsf{m},\mathsf{t}>0$.
		\item
			The Tracy--Widom fluctuation result \eqref{eq:TASEP_fluctuations_result} holds
			under the condition $\mathsf{t}>\mathsf{m}/p>0$.
	\end{enumerate}
\end{theorem}
The proof of \Cref{thm:TASEP_limit} is given in \Cref{sec:app_TASEP_asymptotics}.

\section{Limit behavior of Grothendieck random permutations}
\label{sec:Grothendieck_asymptotics}

\subsection{Law of large numbers for the permutation height function}
\label{sub:LLN_for_H}

In this section, we obtain the asymptotic behavior
of Grothendieck random permutations of growing order
via the
the pre-limit identity in distribution (\Cref{thm:vertex_TASEP_matching})
and
TASEP asymptotics (\Cref{thm:TASEP_limit}).
As a result, we complete the proof of \Cref{thm:Grothendieck_permutations_asymptotics_intro}
formulated in the Introduction.
First, we focus on
computations involving the law of large numbers, and
then add the fluctuation terms.
By \Cref{thm:vertex_TASEP_matching}, we have
\begin{equation*}
	\Prob_{\mathbf{w}}\bigl(H(x,y)\le h\bigr)
	=
	\Prob_{\mathrm{TASEP}}\bigl(
		\xi_{n-x+1-h}(y-1)\ge n-y+2
	\bigr),
\end{equation*}
where $H$ is the height function of the Grothendieck random
permutation~$\mathbf{w}$, and
$\xi_i(t)$ are the TASEP particles
with
the step initial configuration
$\xi_i(0)=n-x-i+2$, $1\le i\le n-x+1$.
Therefore,
\begin{equation}
	\label{eq:H_and_TASEP_centered_matching}
	\Prob_{\mathbf{w}}\bigl(H(x,y)\le h\bigr)
	=
	\Prob_{\mathrm{TASEP}}\bigl(
		\bar\xi_{n-x+1-h}(y-1)\ge n-y+1-h
	\bigr),
\end{equation}
where $\bar \xi_i(t)=\xi_i(t)-\xi_i(0)$ is the displacement.

\medskip
Let $h=\lfloor n\ssp \mathsf{h} \rfloor $, where $\mathsf{h}$ is the
scaled permutation height function.
Let also
$x=\lfloor n\ssp \mathsf{x} \rfloor $, $y=\lfloor n\ssp \mathsf{y} \rfloor $,
where $\mathsf{x},\mathsf{y}\in[0,1]$ are the scaled coordinates in the permutation matrix.
Note that
by \eqref{eq:height_function_bounds}, we have
$0\le \mathsf{h}\le \min(1-\mathsf{x},1-\mathsf{y})$.

For each $\mathsf{x},\mathsf{y},\mathsf{h}$, the
asymptotic location
of $n^{-1}\bar\xi_{n-x+1-h}(y-1)$ is determined
from the TASEP law of large numbers \eqref{eq:TASEP_LLN}, \eqref{eq:TASEP_c_constant_LLN}.
Namely, this location is equal to
\begin{equation*}
	\mathsf{c}(1-\mathsf{x}-\mathsf{h},\mathsf{y})=
	\frac{(\sqrt{p\ssp \mathsf{y}}-\sqrt{1-\mathsf{x}-\mathsf{h}})^2}{1-p}
	\ssp
	\mathbf{1}_{p\ssp \mathsf{y}\ge 1-\mathsf{x}-\mathsf{h}}
	.
\end{equation*}
There are two cases.
If
\begin{equation}
	\label{eq:h_inequality_initial}
	\mathsf{c}(1-\mathsf{x}-\mathsf{h},\mathsf{y})> 1-\mathsf{y}-\mathsf{h},
\end{equation}
then the probability \eqref{eq:H_and_TASEP_centered_matching} goes to~$1$.
If the inequality is reversed, then this probability goes to~$0$.
As $\mathsf{h}$ increases,
the left- and the right-hand sides of \eqref{eq:h_inequality_initial}
increase and decrease, respectively.
Therefore, we should define
\begin{equation}
\label{eq:h_circ}
	\mathsf{h}^\circ(\mathsf{x},\mathsf{y})\coloneqq
	\inf\left\{ \mathsf{h}\in[0,\min(1-\mathsf{x},1-\mathsf{y})]\colon
	\mathsf{c}(1-\mathsf{x}-\mathsf{h},\mathsf{y})> 1-\mathsf{y}-\mathsf{h} \right\}.
\end{equation}
If the set in \eqref{eq:h_circ} is empty, we
take the maximal possible value,
$\mathsf{h}^\circ(\mathsf{x},\mathsf{y})=\min(1-\mathsf{x},1-\mathsf{y})$.

To describe $\mathsf{h}^\circ$ more explicitly,
introduce the northeast part of the ellipse
\begin{equation}
	\label{eq:ellipse}
	\mathscr{E}_p\coloneqq
	\left\{
		(\mathsf{x},\mathsf{y})
		\colon
		(\mathsf{y}-\mathsf{x})^2/p+(\mathsf{y}+\mathsf{x}-1)^2/(1-p)=1,\
		1-p\le \mathsf{x}\le 1
	\right\}\subset [0,1]^2.
\end{equation}

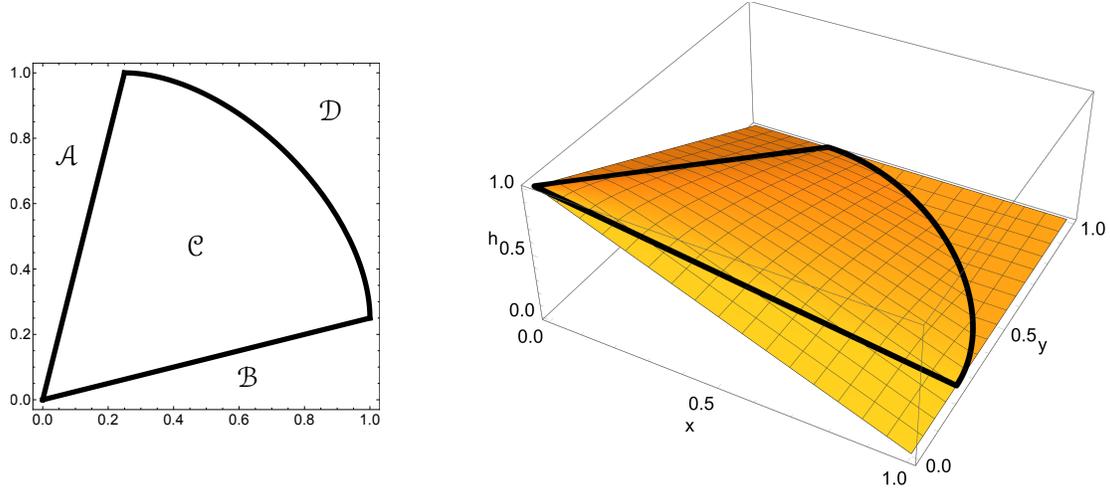
\begin{figure}[htpb]
	\centering
	\begin{tikzpicture}
		\node at (0,0) {\includegraphics[height=.3\textwidth]{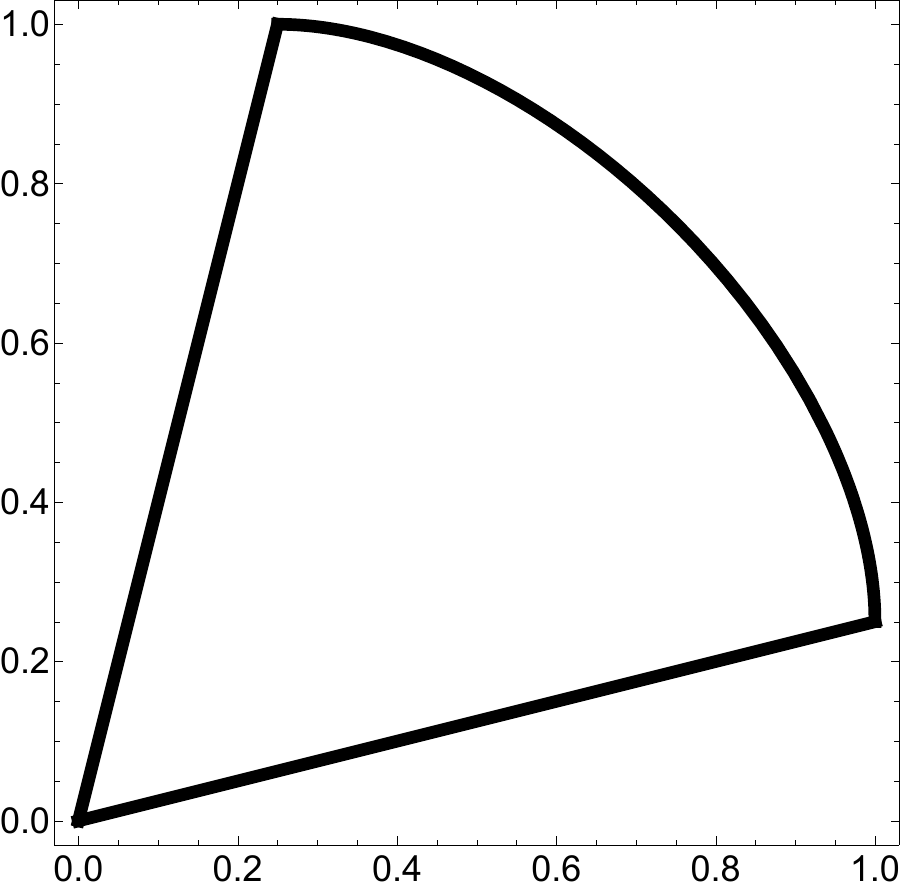}};
		\node at (8,0) {\includegraphics[height=.4\textwidth]{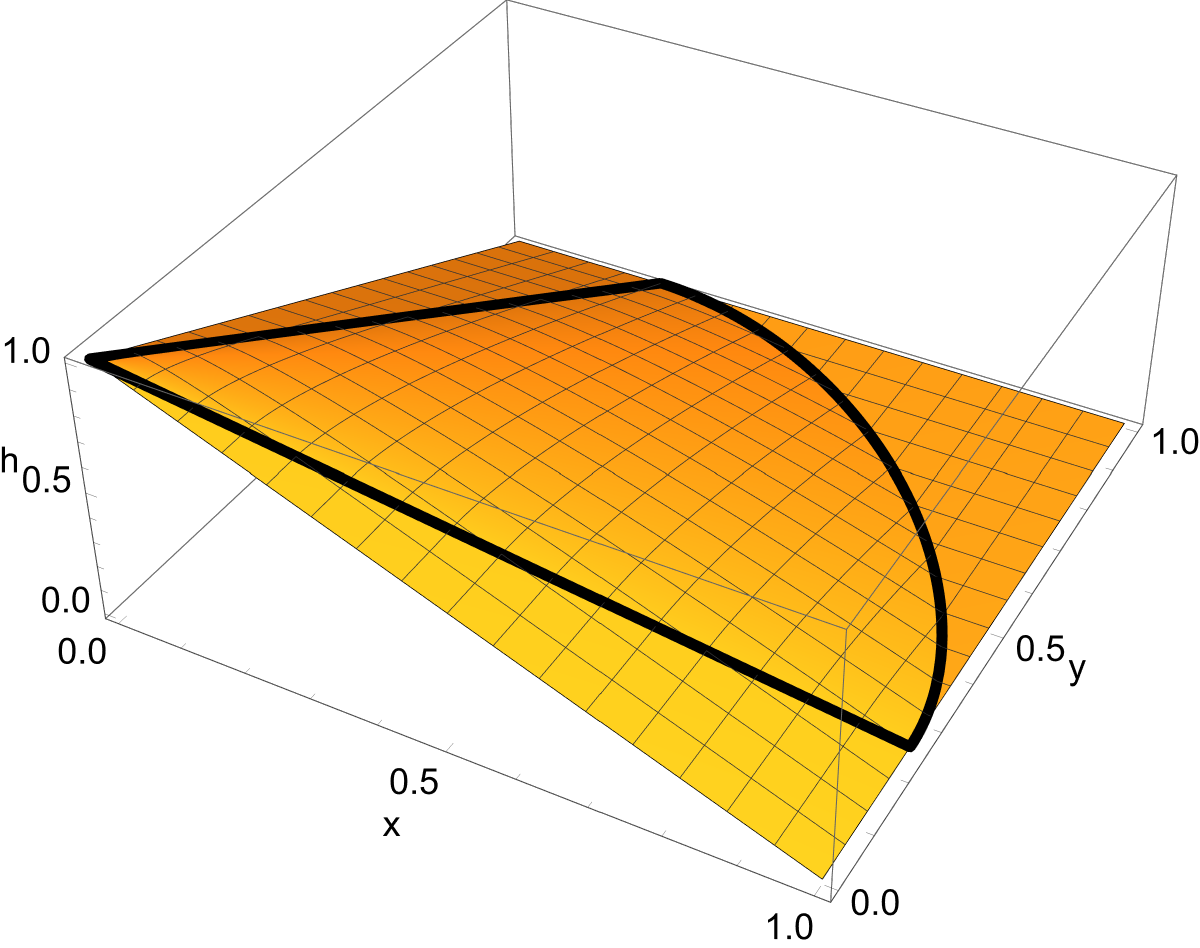}};
		\node at (-1.7,1.2) {\large$\mathscr{A}$};
		\node at (.7,-1.75) {\large$\mathscr{B}$};
		\node at (0,0) {\large$\mathscr{C}$};
		\node at (1.8,1.8) {\large$\mathscr{D}$};
	\end{tikzpicture}
	\caption{\textbf{Left:} The subsets \eqref{eq:h_circ_by_zones}
	of the unit square.
	\textbf{Right:}
	Graph of the limit shape $\mathsf{h}^\circ(\mathsf{x},\mathsf{y})$
	defined by \eqref{eq:h_circ} and given explicitly by
	\eqref{eq:h_circ_by_zones}.
	The function $\mathsf{h}^\circ$ is linear outside of the curved triangle
	$\mathscr{C}$.
	In both figures, the parameter is $p=3/4$.}
	\label{fig:Grothendieck_limit_shape_zones}
\end{figure}

\begin{lemma}
	\label{lemma:Grothendieck_limit_shape_equation_on_h}
	The function
	$\mathsf{h}^\circ$
	\eqref{eq:h_circ} is given by
	\begin{equation}
		\label{eq:h_circ_by_zones}
		\mathsf{h}^\circ(\mathsf{x},\mathsf{y})
		=
		\begin{cases}
			1-\mathsf{y},&(\mathsf{x},\mathsf{y})\in
			\mathscr{A}\coloneqq\left\{ 0<\mathsf{x}<1-p,\ \mathsf{x}/(1-p)<\mathsf{y}<1 \right\}
			;\\
			1-\mathsf{x},&
			(\mathsf{x},\mathsf{y})\in
			\mathscr{B}\coloneqq\left\{ 0<\mathsf{x}<1,\ 0<\mathsf{y}<(1-p)\ssp \mathsf{x} \right\}
			;\\
			\displaystyle
			1+
			\frac{2}{p} \sqrt{(1-p) \mathsf{x}\ssp \mathsf{y}}
			-\frac{\mathsf{x}+\mathsf{y}}{p}
			,&(\mathsf{x},\mathsf{y})\in
			\mathscr{C}\coloneqq\Biggl\{
				\parbox{.38\textwidth}{$0<\mathsf{x}<1,\ (1-p)\ssp \mathsf{x}<\mathsf{y}<\mathsf{x}/(1-p)$,
			\\
				$(\mathsf{x},\mathsf{y})$ \textnormal{below} $\mathscr{E}_p$} \Biggr\}
			;\\
			0,&
			(\mathsf{x},\mathsf{y})\in
			\mathscr{D}\coloneqq\left\{  \textnormal{$(\mathsf{x},\mathsf{y})$ above $\mathscr{E}_p$} \right\}
			.
		\end{cases}
	\end{equation}
	It
	is continuous on $[0,1]^2$.
\end{lemma}
See \Cref{fig:Grothendieck_limit_shape_zones}, left, for an illustration
of the subsets $\mathscr{A},\mathscr{B},\mathscr{C},\mathscr{D}$
in \eqref{eq:h_circ_by_zones}. The graph of the function
$\mathsf{h}^\circ$ is given in
\Cref{fig:Grothendieck_limit_shape_zones}, right.
\begin{proof}
	In $\mathscr{A}$ and $\mathscr{C}$,
	we have
	$\mathsf{c}(1-\mathsf{x}-\mathsf{h}^\circ,\mathsf{y})= 1-\mathsf{y}-\mathsf{h}^\circ$.
	Solving this for $\mathsf{h}^\circ$ leads to the result.
	In $\mathscr{D}$, we have
	$\mathsf{c}(1-\mathsf{x}-\mathsf{h},\mathsf{y})> 1-\mathsf{y}-\mathsf{h}$
	for all $\mathsf{h}\ge0$, so $\mathsf{h}^\circ(\mathsf{x},\mathsf{y})=0$.
	In $\mathscr{B}$,
	we have
	$\mathsf{c}(1-\mathsf{x}-\mathsf{h},\mathsf{y})< 1-\mathsf{y}-\mathsf{h}$
	for all $0\le \mathsf{h}\le 1-\mathsf{x}$.
	This means that
	the set over which we take the infimum in
	\eqref{eq:h_circ} is empty.
	This implies that
	$\mathsf{h}^\circ(\mathsf{x},\mathsf{y})=
	1-\mathsf{x}$ in $\mathscr{B}$, which
	completes the proof.
\end{proof}

\Cref{lemma:Grothendieck_limit_shape_equation_on_h}
and identity \eqref{eq:H_and_TASEP_centered_matching}
immediately imply the law of large numbers for the
height function:
\begin{equation}
	\label{eq:Grothendieck_permutations_LLN_in_the_text}
	\lim_{n\to\infty}n^{-1}
	H
	\left( \lfloor n\ssp \mathsf{x} \rfloor,
	\lfloor n\ssp \mathsf{y} \rfloor \right)
	=\mathsf{h}^\circ(\mathsf{x},\mathsf{y}),
	\qquad \mathsf{x},\mathsf{y}\in[0,1],
\end{equation}
with convergence in probability.
This completes the proof of the first part of
\Cref{thm:Grothendieck_permutations_asymptotics_intro}
from the Introduction.

\subsection{Properties of the limiting permuton}
\label{sub:Grothendieck_permuton_density}

The law of large numbers
\eqref{eq:Grothendieck_permutations_LLN_in_the_text}
implies that the Grothendieck random permutations $\mathbf{w}\in S_n$
converge, as $n\to\infty$, to a (deterministic) \emph{permuton}
prescribed by
$\mathsf{h}^\circ(\mathsf{x},\mathsf{y})$.
Recall that a permuton is a probability measure
on $[0,1]^2$ with uniform marginals.
We refer to
\cite{hoppen2013limits},
\cite{bassino2019universal},
\cite{grubel2023ranks} for detailed treatment of permutons.

In detail, the permuton is connected to the height function as
\begin{equation*}
	\mathsf{h}^\circ(\mathsf{x},\mathsf{y}) =
	\Prob\left( X>\mathsf{x},\ssp Y>\mathsf{y} \right),
	\qquad
	\mathsf{x},\mathsf{y}\in[0,1],
\end{equation*}
where $(X,Y)\in[0,1]^2$ is a random point distributed according to the permuton.
Note that both $X$ and $Y$ are uniformly distributed on $[0,1]$,
while their joint distribution is nontrivial.
Indeed, from \eqref{eq:h_circ_by_zones}
if follows that
$\mathsf{h}^\circ(\mathsf{x},0) = 1-\mathsf{x}$,
$\mathsf{h}^\circ(0,\mathsf{y}) = 1-\mathsf{y}$,
as it should be for uniform marginals of a permuton.

\medskip
Let is find the part of the permuton
that is concentrated on
the curve
$\mathscr{E}_p$ \eqref{eq:ellipse}
(this concentration is visible in simulations, see \Cref{fig:intro_simulations}).
For small $\varepsilon>0$, we have
\begin{equation*}
	\mathsf{h}^\circ(\mathsf{x},\mathsf{y}-\varepsilon)-\mathsf{h}^\circ(\mathsf{x},\mathsf{y})
	=\Prob\left( X>\mathsf{x},\ \mathsf{y}-\varepsilon<Y\le \mathsf{y} \right).
\end{equation*}
Dividing this by $\varepsilon=\Prob(\mathsf{y}-\varepsilon<Y\le \mathsf{y})$,
taking the limit as $\varepsilon\to0$,
and passing to the complement event,
we obtain the conditional cumulative distribution function (cdf) of $X$ given $Y=\mathsf{y}$:
\begin{equation}
	\label{eq:Grothendieck_permuton_conditional_cdf}
	\Prob(X\le \mathsf{x}\mid Y = \mathsf{y})=
	1+
	\frac{\partial}{\partial \mathsf{y}}\mathsf{h}^\circ(\mathsf{x},\mathsf{y})=
	\begin{cases}
		0,& (\mathsf{x},\mathsf{y})\in\mathscr{A};\\
		1,& (\mathsf{x},\mathsf{y})\in\mathscr{B}\cup \mathscr{D};\\
		\dfrac{p-1}{p}+\dfrac{\sqrt{(1-p)\ssp \mathsf{x}\mathsf{y}}}{p\ssp \mathsf{y}},& (\mathsf{x},\mathsf{y})\in\mathscr{C}.
	\end{cases}
\end{equation}
This function has a jump discontinuity along the ellipse $\mathscr{E}_p$ \eqref{eq:ellipse},
that is, at the point
\begin{equation}
	\label{eq:x_p_on_ellipse}
	\mathsf{x}_p(\mathsf{y})
	\coloneqq
	\left( \sqrt{p(1-\mathsf{y})}+\sqrt{\mathsf{y}(1-p)} \right)^2,
	\qquad
	\mathsf{y}>1-p.
\end{equation}
We conclude:
\begin{proposition}
	\label{prop:Grothendieck_permuton_atom}
	Let $\mathsf{y}>1-p$.
	The unique discontinuity of the conditional distribution of $X$ given $Y=\mathsf{y}$ is
	an atom
	at $\mathsf{x}=\mathsf{x}_p(\mathsf{y})$, and the atom's mass is
	\begin{equation}
		\label{eq:Grothendieck_permuton_atom}
			\Prob( X = \mathsf{x}_p(\mathsf{y})\mid Y = \mathsf{y})=
			\frac{1}{p}-
			\dfrac{\sqrt{(1-p)\ssp \mathsf{x}_p(\mathsf{y})\ssp \mathsf{y}}}{p\ssp \mathsf{y}}
			=
			\frac{1}{p}
			-\frac{\sqrt{1-p}
				\left( \sqrt{p(1-\mathsf{y})}+\sqrt{\mathsf{y}(1-p)} \right)
				}{p\sqrt{\mathsf{y}}},
	\end{equation}
	where
	$1-p<\mathsf{y}\le 1$.
\end{proposition}

In particular, from \eqref{eq:Grothendieck_permuton_atom} we have
$\Prob( X = 1-p\mid Y = 1)
=1$, as it should be.

\begin{remark}
	\label{rmk:no_delta_in_2d}
	While conditionally on $Y=\mathsf{y}$ (where $\mathsf{y}>1-p$),
	the distribution of $X$ has an atom,
	the permuton (that is, the joint distribution of $X$ and $Y$)
	does not have atoms. This follows from the fact that
	$\mathsf{h}^\circ(\mathsf{x},\mathsf{y})$
	\eqref{eq:h_circ_by_zones} is continuous on $[0,1]^2$.
\end{remark}

By integrating
\eqref{eq:Grothendieck_permuton_atom}
in $\mathsf{y}$ from $1-p$ to $1$ (which is straightforward),
we obtain the ``singular'' mass of the permuton,
that is, the
mass
concentrated on the curve $\mathscr{E}_p$ \eqref{eq:ellipse}:
\begin{proposition}
	\label{prop:delta_mass}
	Let $(X,Y)\in[0,1]^2$ be a random point distributed according to our permuton.
	Then we have
	\begin{equation*}
		\Prob\left[ (X,Y)\in \mathscr{E}_p \right]
		=
		\Prob\left[ X=\mathsf{x}_p(Y) \right]
		=1-\sqrt{\frac{1-p}{p}}\ssp \arccos\sqrt{1-p},
		\qquad p\in[0,1].
	\end{equation*}
	In particular, for $p=\frac{1}{2}$,
	this expression is equal to $1-\frac{\pi}{4}$.
\end{proposition}

As an application, let us consider the law of large numbers of the number of inversions
$\inv(\mathbf{w})\coloneqq \#\{i<j\colon \mathbf{w}(i)>\mathbf{w}(j)\}$, where
$\mathbf{w}\in S_n$ is the Grothendieck random permutation.
By a result of
\cite{hoppen2013limits} (see also Theorem~1 in the survey \cite{grubel2023ranks} or
\cite[Theorem~2.5]{bassino2019universal}),
the convergence of $\mathbf{w}\in S_n$ to a permuton
is equivalent to the convergence (in probability) of all pattern counting statistics
of $\mathbf{w}$.
The limiting pattern counts are determined by the permuton.
In particular,
for inversions, we have
\begin{equation}
	\label{eq:Grothendieck_permuton_inversions_convergence_in_probability}
	\lim_{n\to\infty}\frac{\inv(\mathbf{w})}{\binom{n}{2}}=
	\gamma_p
	\coloneqq\Prob
	\left[ (X_1>X_2,\, Y_1<Y_2) \textnormal{ or } (X_1<X_2,\, Y_1>Y_2) \right],
	\qquad p\in(0,1),
\end{equation}
where $(X_1,Y_1)$ and $(X_2,Y_2)$ are independent
points of $[0,1]^2$
sampled from the limiting permuton.

\begin{proposition}[\Cref{prop:Grothendieck_permuton_inversions_intro} in the 	Introduction]
	\label{prop:Grothendieck_permuton_inversions}
	We have $\gamma_{p}=1-\sqrt{\frac{1-p}{p}}\ssp \arccos\sqrt{1-p}$.
\end{proposition}
The fact that the scaled number of inversions is
the same as the singular mass of the permuton is surprising,
but we do not have a conceptual explanation for this coincidence.
\begin{proof}[Proof of \Cref{prop:Grothendieck_permuton_inversions}]
	Fix $(\mathsf{x},\mathsf{y})\in[0,1]^2$, and let $(X,Y)\in[0,1]^2$
	be a random point sampled from the permuton.
	We have
	\begin{equation}
		\label{eq:Grothendieck_permuton_inversions_proof}
		\Prob
		\left[ (X>\mathsf{x},\, Y<\mathsf{y}) \textnormal{ or } (X<\mathsf{x},\, Y>\mathsf{y}) \right]
		=1-\mathsf{h}^\circ(\mathsf{x},\mathsf{y})-
		\mathsf{h}^\bullet(\mathsf{x},\mathsf{y}),
	\end{equation}
	where $\mathsf{h}^\circ(\mathsf{x},\mathsf{y})=\Prob(X>\mathsf{x},\, Y>\mathsf{y})$ is given by
	\eqref{eq:h_circ_by_zones}, and
	\begin{equation*}
		\mathsf{h}^\bullet(\mathsf{x},\mathsf{y})\coloneqq
		\Prob(X\le \mathsf{x},\, Y\le \mathsf{y})
	\end{equation*}
	is the \emph{copula} \cite{grubel2023ranks} corresponding to the permuton.

	In \eqref{eq:Grothendieck_permuton_inversions_proof},
	we used the absence of atoms
	(\Cref{rmk:no_delta_in_2d}) and
	ignored the difference between strict and weak inequalities.
	The function $\mathsf{h}^\bullet$ can be computed using the
	conditional distribution \eqref{eq:Grothendieck_permuton_conditional_cdf}:
	\begin{equation*}
		\mathsf{h}^\bullet(\mathsf{x},\mathsf{y})=
		\int_0^{\mathsf{y}}
		\Prob(X\le \mathsf{x}\mid Y = w)\ssp dw=
		\begin{cases}
			\mathsf{x},& (\mathsf{x},\mathsf{y})\in \mathscr{A};\\
			\mathsf{y},& (\mathsf{x},\mathsf{y})\in \mathscr{B};\\
			\displaystyle
			\frac{p-1}{p}(\mathsf{x}+\mathsf{y})+
			\frac{2}{p}\sqrt{(1-p)\ssp \mathsf{x}\mathsf{y}},
			& (\mathsf{x},\mathsf{y})\in \mathscr{C};\\
			\mathsf{x}+\mathsf{y}-1,& (\mathsf{x},\mathsf{y})\in \mathscr{D}.
		\end{cases}
	\end{equation*}

	By \eqref{eq:Grothendieck_permuton_inversions_proof},
	the constant $\gamma_{p}$ (for general $p$) is given by
	\begin{equation*}
		\gamma_p=
		\operatorname{\mathbb{E}}
		\left[ 1 - \mathsf{h}^\circ(X,Y) - \mathsf{h}^\bullet(X,Y) \right],
	\end{equation*}
	where the expectation is taken with respect to the permuton.
	To evaluate this expectation, we use the conditional density
	of $X$ given $Y=y$ which is obtained from
	\eqref{eq:Grothendieck_permuton_conditional_cdf}:
	\begin{equation*}
		f_{X\mid Y=\mathsf{y}}(\mathsf{x})=
		\frac{1}{2p}\ssp\sqrt{\frac{1-p}{\mathsf{x}\mathsf{y}}}\ssp
		\mathbf{1}_{(\mathsf{x},\mathsf{y})\in\mathscr{C}}.
	\end{equation*}
	Recalling that for $Y=\mathsf{y}>1-p$ the conditional distribution of $X$ has an atom,
	we can express the expectation as follows:
	\begin{equation}
		\label{eq:Grothendieck_permuton_inversions_expectation_computation}
		\begin{split}
			\gamma_p=
			1&-
			\iint_{\mathscr{C}}
			\left(
				\mathsf{h}^\circ(\mathsf{x},\mathsf{y})+\mathsf{h}^\bullet(\mathsf{x},\mathsf{y})
			\right)
			f_{X\mid Y=\mathsf{y}}(\mathsf{x})\ssp d\mathsf{x}\ssp d\mathsf{y}
			\\&-
			\int_{1-p}^1
			\bigl(
				\mathsf{h}^\circ(\mathsf{x}_p(\mathsf{y}),\mathsf{y})+
				\mathsf{h}^\bullet(\mathsf{x}_p(\mathsf{y}),\mathsf{y})
			\bigr)
			\Prob
			\left( X=\mathsf{x}_p(\mathsf{y})\mid Y=\mathsf{y} \right)\ssp d\mathsf{y}.
		\end{split}
	\end{equation}
	In the single integral in
	\eqref{eq:Grothendieck_permuton_inversions_expectation_computation},
	one can check that
	$\mathsf{h}^\circ(\mathsf{x}_p(\mathsf{y}),\mathsf{y})+
	\mathsf{h}^\bullet(\mathsf{x}_p(\mathsf{y}),\mathsf{y})=\mathsf{x}_p(\mathsf{y})+\mathsf{y}-1$,
	and so the integral takes the form
	\begin{equation*}
		\int_{1-p}^1
		(\mathsf{x}_p(\mathsf{y})+\mathsf{y}-1)
		\left(
		\frac{1}{p}-
		\dfrac{\sqrt{(1-p)\ssp \mathsf{x}_p(\mathsf{y})\ssp \mathsf{y}}}{p\ssp \mathsf{y}}\right)
		d\mathsf{y}=
		\frac{1}{2} \left( p-1+\sqrt{\frac{1}{p}-1} \Bigl(\frac{\pi}{2} -\arctan\sqrt{\frac{1}{p}-1}\ssp\Bigr)\right).
	\end{equation*}
	The integral is expressed in a closed form since the integrand has an explicit
	antiderivative.\footnote{Antiderivatives throughout this proof
		are expressed through elementary functions. They are tedious but explicit,
	and were obtained by a computer algebra system (namely, \texttt{Mathematica}).}
	In the double integral, we have
	\begin{multline*}
		\left(
		\mathsf{h}^\circ(\mathsf{x},\mathsf{y})
		+
		\mathsf{h}^\bullet(\mathsf{x},\mathsf{y})
		\right)\ssp f_{X\mid Y=\mathsf{y}}(\mathsf{x})=
		\left( 1+\frac{4\sqrt{1-p}}{p}\sqrt{\mathsf{x}\mathsf{y}}
		+\left( 1-\frac{2}{p} \right)(\mathsf{x}+\mathsf{y})\right)
		\frac{1}{2p}\ssp\sqrt{\frac{1-p}{\mathsf{x}\mathsf{y}}}
		\\=
		\frac{\sqrt{1-p}}{2p\sqrt{\mathsf{x}\mathsf{y}}}
		+
		\frac{2(1-p)}{p^2}
		+
		\left( 1-\frac{2}{p} \right)\frac{\sqrt{1-p}}{2p}
		\cdot
		\frac{\mathsf{x}+\mathsf{y}}{\sqrt{\mathsf{x}\mathsf{y}}}.
	\end{multline*}
	For integrating over $\mathscr{C}$, we split
	the $d\mathsf{y}$ integral into two, and can readily
	compute the $d\mathsf{x}$ integral in both cases.
	Denote the $\mathsf{x}$-antiderivative by
	\begin{equation*}
		I(t)\coloneqq
		\frac{\sqrt{1-p}\sqrt{t}}{p\sqrt{\mathsf{y}}}
		+
		\frac{2(1-p)}{p^2}t
		+
		\left( 1-\frac{2}{p} \right)\frac{\sqrt{1-p}}{p}
		\cdot
		\frac{\frac13 t^{\frac32}+\mathsf{y}\sqrt{t}}{\sqrt{\mathsf{y}}}.
	\end{equation*}
	Then the double integral in \eqref{eq:Grothendieck_permuton_inversions_expectation_computation}
	is equal to
	\begin{multline*}
		\int_0^{1-p}\Bigl( I(\mathsf{y}/(1-p))-I(\mathsf{y}(1-p)) \Bigr)d\mathsf{y}
		+
		\int_{1-p}^1\Bigl( I(\mathsf{x}_p(\mathsf{y}))-I(\mathsf{y}(1-p)) \Bigr)d\mathsf{y}=
		\frac{1}{3} \left(p+\frac{2}{p}-3\right)
		\\
		+\frac{3 \pi  \sqrt{(1-p) p^3}-3 \sqrt{(1-p) p^3} \left(\arctan\sqrt{\frac{1}{p}-1}+
		\frac{\pi }{2}\right)+p ((9-5 p) p-4)}{6 p^2},
	\end{multline*}
	where the second line is the integral $\int_{1-p}^1 I(\mathsf{x}_p(\mathsf{y}))d\mathsf{y}$,
	whose integrand has an explicit antiderivative.
	Putting all the computed integrals together
	and using basic trigonometric identities, we arrive at the result.
\end{proof}

\subsection{Fluctuations of the permutation height function}
\label{eq:TW_for_H}

Let us now apply the second part of
\Cref{thm:TASEP_limit} to the Grothendieck random permutations, and
obtain fluctuations of the height function
around its limit shape $\mathsf{h}^\circ(\mathsf{x},\mathsf{y})$.
Throughout this subsection, we assume that the point $(\mathsf{x},\mathsf{y})$
belongs to the curved triangle $\mathscr{C}$
where the height function $\mathsf{h}^\circ$ is not linear.

Recall that by $H(x,y)$ we denote the random height function of the
Grothendieck random permutation $\mathbf{w}\in S_n$.
By \eqref{eq:H_and_TASEP_centered_matching}, we have
for all $r\in \mathbb{R}$:
\begin{equation}
	\label{eq:Grothendieck_permutations_fluctuations_first_formula}
	\begin{split}
		&
		\Prob_{\mathbf{w}}\bigl(
			H(\lfloor n\ssp \mathsf{x} \rfloor,\lfloor n\ssp \mathsf{y} \rfloor)
			\le
			n\ssp \mathsf{h}^\circ+ r\ssp n^{1/3}
		\bigr)
		\\&\hspace{60pt}=
		\Prob_{\mathrm{TASEP}}\bigl(
			\bar\xi_{
			\lfloor n\ssp(1-\mathsf{x}-\mathsf{h}^\circ-r\ssp n^{1/3}) \rfloor }
				(\lfloor n\ssp \mathsf{y} \rfloor )
				\ge
				n\ssp(1-\mathsf{y}-\mathsf{h}^\circ-r\ssp n^{1/3})
		\bigr)+o(1),
	\end{split}
\end{equation}
where we used the shorthand notation
\begin{equation}
	\label{eq:h_circ_curved}
	\mathsf{h}^\circ
	=
	\mathsf{h}^\circ(\mathsf{x},\mathsf{y})
	=
	1+
	\frac{2}{p} \sqrt{(1-p) x y}
	-\frac{x+y}{p}
	.
\end{equation}
The term $o(1)$ in \eqref{eq:Grothendieck_permutations_fluctuations_first_formula}
is due to the fact that we removed
shifts by $(+1)$, and combined the integer parts. These
modifications are negligible because fluctuations live on the scale $n^{1/3}$.

By the fluctuation result for TASEP,
the probability in the right-hand side of
\eqref{eq:Grothendieck_permutations_fluctuations_first_formula}
has the following behavior:
\begin{equation}
	\label{eq:Grothendieck_permutations_fluctuations_second_formula}
	\begin{split}
		&
	\lim_{n\to\infty}
		\Prob_{\mathrm{TASEP}}\bigl(
			\bar\xi_{
			\lfloor n\ssp(1-\mathsf{x}-\mathsf{h}^\circ-r\ssp n^{1/3}) \rfloor }
				(\lfloor n\ssp \mathsf{y} \rfloor )
				\ge
				n\ssp(1-\mathsf{y}-\mathsf{h}^\circ-r\ssp n^{1/3})
			\bigr)\\&\hspace{180pt}=
		F_2\left( r\Bigl( \frac{1}{\mathsf{v}_1(1-\mathsf{x}-\mathsf{h}^\circ,\mathsf{y})}
		+
		\frac{1}{\mathsf{v}_2(1-\mathsf{x}-\mathsf{h}^\circ,\mathsf{y})} \Bigr) \right),
	\end{split}
\end{equation}
where $F_2$ is the Tracy--Widom GUE cumulative distribution function,
and the constants $\mathsf{v}_1,\mathsf{v}_2$ are given by
\eqref{eq:TASEP_fluctuations_constants}.
We have
\begin{equation*}
	\frac{1}{\mathsf{v}_1(\mathsf{m},\mathsf{t})}
	+
	\frac{1}{\mathsf{v}_2(\mathsf{m},\mathsf{t})}=
	\frac{(\sqrt{\mathsf{t}}-\sqrt{\mathsf{m}\ssp p}) \ssp \mathsf{t}^{1/6}}{\mathsf{m}^{1/3}
	(\sqrt{\mathsf{t}/p}-\sqrt{\mathsf{m}})^{2/3}
	(\sqrt{\mathsf{t}\ssp p}-\sqrt{\mathsf{m}})^{2/3}}.
\end{equation*}
Let us denote the above expression with
$\mathsf{m}=1-\mathsf{x}-\mathsf{h}^\circ(\mathsf{x},\mathsf{y})$
and $\mathsf{t}=\mathsf{y}$ by
\begin{equation}
	\label{eq:V_constant_for_H}
	\mathsf{v}(\mathsf{x},\mathsf{y})\coloneqq
	\frac{(\sqrt{\mathsf{y}}-\sqrt{p(1-\mathsf{x}-\mathsf{h}^\circ(\mathsf{x},\mathsf{y}))})
	\ssp \mathsf{y}^{1/6}}{(1-\mathsf{x}-\mathsf{h}^\circ(\mathsf{x},\mathsf{y}))^{1/3}
	(\sqrt{\mathsf{y}/p}-\sqrt{1-\mathsf{x}-\mathsf{h}^\circ(\mathsf{x},\mathsf{y})})^{2/3}
	(\sqrt{\mathsf{y}\ssp p}-\sqrt{1-\mathsf{x}-\mathsf{h}^\circ(\mathsf{x},\mathsf{y})})^{2/3}},
\end{equation}
where $\mathsf{h}^\circ(\mathsf{x},\mathsf{y})$ is given in \eqref{eq:h_circ_curved}.
While it is not evident from \eqref{eq:V_constant_for_H},
inside $\mathscr{C}$
the function $\mathsf{v}(\mathsf{x},\mathsf{y})$
is symmetric in $\mathsf{x},\mathsf{y}$.

With the constant $\mathsf{v}(\mathsf{x},\mathsf{y})$, the convergence
in \eqref{eq:Grothendieck_permutations_fluctuations_second_formula}
implies the desired
Tracy--Widom GUE fluctuations of the height function
$H(x,y)$.
This completes the proof of
\Cref{thm:Grothendieck_permutations_asymptotics_intro}.

\section{Random permutations from non-reduced pipe dreams}
\label{sec:non-reduced}

Here we consider a different family of random permutations
which is
obtained by sampling a non-reduced pipe dream and not resolving any double crossings.
This corresponds to setting $q=1$ in \Cref{def:q_reduction_pipe_dream}
from the Introduction.
For shorter notation, throughout this section we denote these random permutations
$\mathbf{w}^{(1)}$ by $\mathbf{w}$.

\subsection{Exact formula}
\label{sub:exact_formulas}

Consider the staircase shape
$\updelta_n= \left\{ (i,j)\colon i+j\le n \right\}$
(where $i$ and $j$ are the row and column indices, respectively)
and
place tiles
\raisebox{-2pt}{\includegraphics[scale=0.4]{cross}}
or
\raisebox{-2pt}{\includegraphics[scale=0.4]{elbows}}
in each box of $\updelta_n$ independently with probabilities $p$ and $1-p$, respectively.
Let the random permutation $\mathbf{w}\in S_n$
is obtained by following the pipes starting at $(i,0)$, and ending at $(0,\mathbf{w}_i^{-1})$,
$i=1,\ldots,n$.
For example, for the pipe dream in \Cref{fig:pipe_dream}, this procedure gives
$w=241635$.

Observe that each pipe is a random walk in the bulk of
$\updelta_n$, equipped with the mandatory turns by
$90^\circ$ at the elbows at the diagonal boundary $i+j=n$.
We begin by deriving an exact formula for the probability
distribution of $\mathbf{w}_i^{-1}$, the outgoing position
of the $i$-th pipe. For general~$p$, this distribution
involves two Gauss hypergeometric functions ${}_2F_1$, which
for $p=\frac{1}{2}$ simplifies to a sum of binomial
coefficients. We do not use these exact formulas for
asymptotic analysis, but instead, in
\Cref{sub:random_walks_estimates} below, we obtain bounds on
the expected number of inversions using the random walk
approach.

Recall the Gauss hypergeometric function
$${}_2F_1(a,b;c;z) = \sum_{k\geq 0} \frac{ (a)_k (b)_k}{(c_k)} \frac{z^k}{k!},$$
where $(a)_k = a(a+1)\cdots (a+k-1)$. Set
\begin{equation*}
	\begin{split}
		F_1(i,j)&\coloneqq
		(1-p)p^{i+j-2} {}_2F_1\left(-i+1,-j+1;1; (1/p-1)^2\right)
		\\&\hspace{80pt}= \sum_{k=0}^{\infty} \binom{i-1}{k}\binom{j-1}{k} (1-p)^{2k+1} p^{i+j-2-2k};\\
		F_2(s,n) &\coloneqq n (1-p)^2p^{n+s-3}{}_2F_1\left(-s+2,-n+1,2, (1/p-1)^2 \right)
		\\&\hspace{80pt}= \sum_{k=0}^{\infty} \binom{s-2}{k}\binom{n}{k+1} (1-p)^{2k+2} p^{n+s-3-2k}.
	\end{split}
\end{equation*}
Note that both hypergeometric series terminate after finitely many terms.
For $p=\frac{1}{2}$, these functions simplify to binomial coefficients:
\begin{equation*}
	F_1(i,j)\big\vert_{p=1/2}
	=2^{-i-j+1}\binom{i+j-2}{j-1},
	\qquad
	F_2(s,n)\big\vert_{p=1/2}
	=2^{-s-n+1}\binom{n+s-2}{n-1}.
\end{equation*}

\begin{proposition}
	\label{prop:pipe_dream_distribution_natural}
	We have for all $i,j=1,\ldots,n $:
	\begin{equation}
		\label{eq:pipe_dream_distribution_natural}
		\Prob(\mathbf{w}^{-1}_i=j) = F_1(i,j) +
			p^n \mathbf{1}_{i+j=n+1}
			+
			F_2(i+j-n,n)\ssp \mathbf{1}_{i+j>n+1}.
	\end{equation}
\end{proposition}
\begin{proof}
    We proceed by induction on $i$.
		For $i=1$, we have
		$$\Prob(\mathbf{w}^{-1}_1=j) =
    \begin{cases}
    (1-p)\ssp p^{j-1}, & \textnormal{ if } 1\leq j<n;\\
		p^{n-1}, & \textnormal{ if } j=n,\\
    \end{cases}
    $$
    which coincides with the proposed formula. Now let $i>1$.
		Let $X_i^{(n)}$ denote the exit position of the pipe started at $(i,0)$,
		assuming that all squares on the diagonal $i+j=n+1$ are elbows.
		By considering where this pipe crosses the first row,
		we can represent
		$X_i^{(n)}=X_{i-1}^{(n-1)} + Y$,
		where $Y$ depends on $X_{i-1}^{(n-1)}$ and has the following
		conditional distribution
		given $X_{i-1}^{(n-1)}$:
		\begin{equation*}
			\Prob\bigl(Y=m\mid X_{i-1}^{(n-1)}\bigr) =
			\begin{cases}
				p,& m=0;\\
				(1-p)^2 \ssp p^{m-1},& 1\le m< n-X_{i-1}^{(n-1)};\\
				(1-p) \ssp p^{n-X_{i-1}^{(n-1)}-1},& m=n-X_{i-1}^{(n-1)}.
			\end{cases}
		\end{equation*}
		Since $\mathbf{w}^{-1}_i=X_i^{(n)}$,
		we can write
    \begin{equation}
			\label{eq:pipe_dream_main_recursion}
			\begin{split}
					&\Prob(X_{i-1}^{(n-1)}=j)\cdot p+\sum_{r< j} \Prob(X_{i-1}^{(n-1)}=r) \cdot
					\Prob(\mathbf{w}^{-1}_i =j)
					\\&\hspace{120pt}
					= \begin{cases}
					(1-p)^2 p^{j-r-1}, & j<n;\\
					\sum_{r< n} \Prob(X_{i-1}^{(n-1)}=r)\cdot (1-p) \ssp p^{n-r-1}, &  j=n.
        \end{cases}
			\end{split}
    \end{equation}
		Using the induction hypothesis,
		we know the distribution of $X_{i-1}^{(n-1)}$
		in terms of $F_1$ and $F_2$.
		Therefore, it remains to verify the recursion
		\eqref{eq:pipe_dream_main_recursion}
		for the answer \eqref{eq:pipe_dream_distribution_natural}.
		This is easily done by grouping binomial coefficients,
		summing over $r$, and using the hockey-stick identity.
\end{proof}

\subsection{Estimates from random walks}
\label{sub:random_walks_estimates}

We are interested in the asymptotics of the (expected) number of inversions
$\inv(\mathbf{w})$ in the random permutation $\mathbf{w}$ obtained from the non-reduced pipe dream model.
First, we use the following known bound in terms of a
\emph{displacement} (\emph{disarray})
$\operatorname{\mathrm{dis}}(w) \coloneqq \sum_{i=1}^n |i-w_i|
=
\sum_{i=1}^n|i-w_i^{-1}|$:

\begin{proposition}[Diaconis--Graham~\cite{diaconis1977spearman}]\label{prop:diaconis graham inversions}
	For any permutation $w$, we have
	 \begin{equation}
		 \label{eq:diaconis_graham}
			\frac12\operatorname{\mathrm{dis}}(w) \leq \inv(w) \leq \operatorname{\mathrm{dis}}(w).
	 \end{equation}
\end{proposition}
Taking the expectation of \eqref{eq:diaconis_graham} and using linearity,
we see that it suffices to understand the asymptotic behavior
of $\operatorname{\mathbb{E}}\bigl[\bigl|i-\mathbf{w}_i^{-1}\bigr|\bigr]$ for all $i$.
In principle, the exact distribution of $\mathbf{w}_i^{-1}$
from \Cref{prop:pipe_dream_distribution_natural}
should allow us to compute this expectation.
However, this is not straightforward.
Instead, we use the random walk interpretation of each individual
pipe:

\begin{proposition}\label{prop:abs_value_sigma_i}
	Fix $\varepsilon>0$ and let $i<(1-\varepsilon)n$. Then
	\begin{equation*}
		\operatorname{\mathbb{E}}
		\bigl[\bigl|i-\mathbf{w}_i^{-1}\bigr|\bigr]
		=
		\sqrt{\frac{4i}{\pi} \ssp \frac{p}{1-p}
		} + O(1),\qquad n\to\infty.
	\end{equation*}
\end{proposition}
\begin{proof}
	Fix $i$. Let $J_i,J_{i-1},\ldots,J_1,J_0 $
	be the (random) column coordinates of the positions of the $i$-th pipe
	in row $i,i-1,\ldots,1,0$.
	In particular, $J_i=0$ and $J_0=\mathbf{w}^{-1}_i$.
	For example, in the pipe dream in \Cref{fig:pipe_dream}, left, we have
	$J_3=0$, $J_2=3$, $J_1=3$, and $J_0=5$ for $i=3$.

	We aim to upper bound the probability that
	the pipe reaches the diagonal $i+j=n$ first time at row $k_0$, that is,
	$J_{k_0}=n-k_0$. Denote this event by $R_{k_0}$,
	and observe that $\Prob(R_{a})\ge \Prob(R_{b})$ for $a<b$.
	Before the pipe reaches the diagonal,
	the horizontal displacements $J_{k}-J_{k+1}$, $k_0<k<i-1$, are iid (independent
	identically distributed), and are distributed as
	\begin{equation}
		\Prob\left( \{J_{k}-J_{k+1}=m\}\cap R_{k_0} \right)=p\ssp \mathbf{1}_{m=0}
		+
		(1-p)^2p^{m-1}\ssp \mathbf{1}_{m\geq 1},\qquad m\in \mathbb{Z}_{\ge 0},
		\quad k_0<k<i-1.
	\end{equation}
	One readily checks that the
	expectation and variance of the random variable in the right-hand side are equal to $1$
	and $\frac{2p}{1-p}$, respectively.
	We can represent
	\begin{equation*}
		J_{k_0}=
		J_{k_0}-J_{k_0+1}+J_{i-1}+\sum_{k=k_0+1}^{i-2}\left( J_k-J_{k+1} \right).
	\end{equation*}
	For the event $R_{k_0}$ to occur, the above sum must be
	equal to $n-k_0=i-k_0+(n-i)>i-k_0+n \varepsilon$.
	Since the expectations of the iid summands are equal to $1$,
	by a standard large deviation estimate, this probability
	is upper bounded by $e^{-c_\varepsilon n}$ for suitable
	$c_\varepsilon(k_0)>0$.\footnote{Note that if $i-k_0\ll n$, the actual bound is even stronger than this, but we do not need this precision.}
	By monotonicity, we can choose these constants such that $c_\varepsilon(k_0)\ge C_\varepsilon\coloneqq c_\varepsilon(0)>0$
	for all $k_0$.
	Taking the union over all $k_0$, we see that the probability
	that a pipe started at $i$ ever reaches the diagonal
	is exponentially small.

	Therefore, $J_0=\mathbf{w}^{-1}_i$
	is close (with exponentially small error in probability)
	to a sum of $i$ iid random variables
	$J_k-J_{k+1}$.
	By the Central Limit Theorem,
	this sum is approximately normal with mean $i$ and variance
	$2ip/(1-p)$. Subtracting $i$ from $\mathbf{w}^{-1}_i$, taking the absolute value,
	and using the fact that $\operatorname{\mathbb{E}}\left( |Z| \right)=\sqrt{2/\pi}$
	for standard normal $Z$, we obtain the result.
\end{proof}

We can now bound the expected number of inversions in $\mathbf{w}$ on the scale
$n^{3/2}$:
\begin{theorem}
	\label{thm:random_permutation_bounds}
	Fix $p\in[0,1)$.
	For every $\varepsilon
		>0$ and sufficiently large $n$, we have
		\begin{equation}
			\label{eq:non_red_bounds}
			\frac{2}{3\sqrt{\pi}} (1-\varepsilon)\ssp  n^{3/2} \sqrt{\frac{p}{1-p}}\leq
			\operatorname{\mathbb{E}}[\inv(\mathbf{w})] \leq
			\frac{4}{3\sqrt{\pi}} (1+\varepsilon) \ssp n^{3/2} \sqrt{\frac{p}{1-p}}.
		\end{equation}
\end{theorem}
\begin{proof}
		For $p=0$, the permutation is identity with probability $1$,
		so the bounds
		\eqref{eq:non_red_bounds} hold trivially.

    Let $p>0$ and let us calculate $\operatorname{\mathbb{E}}[\operatorname{\mathrm{dis}}(\mathbf{w})]$ to apply \Cref{prop:diaconis graham inversions}.
		With \Cref{prop:abs_value_sigma_i} for $i< (1-\varepsilon)^{2/3}n$, we obtain
		for large $n$:
    \begin{align*}
			\operatorname{\mathbb{E}}[\operatorname{\mathrm{dis}}(\mathbf{w})] =\sum_{i=1}^n \operatorname{\mathbb{E}}\bigl[\bigl|\mathbf{w}_i^{-1}-i\bigr|\bigr] \geq \sum_{i=1}^{(1-\varepsilon)^{2/3}n} \frac{2}{\sqrt{\pi}} \sqrt{\frac{p}{1-p}} \sqrt{i} \geq \frac{2}{\sqrt{\pi}} \sqrt{\frac{p}{1-p}} \frac{2}{3} n\sqrt{n}\ssp (1-\varepsilon),
    \end{align*}
    where the last inequality is a simple Riemann sum approximation. The upper bound follows similarly, this time extending the Riemann sum to $n$.
\end{proof}

Numerical simulations suggest the following behavior of the
number of inversions:
\begin{conjecture}
	\label{conj:natural_model_number_of_inversions}
	For any $p\in[0,1)$, we have convergence in probability:
	\begin{equation}
		\label{eq:p_fit}
		\lim_{n\to\infty} \frac{\inv(\mathbf{w})}{n^{3/2}} = \varkappa\sqrt{\frac{p}{1-p}}.
	\end{equation}
\end{conjecture}
Note that the bounds
$2/(3\sqrt\pi)\approx 0.376$ and $4/(3\sqrt\pi)\approx 0.752$
in \Cref{thm:random_permutation_bounds}
are at the same time bounds on $\varkappa$.
Simulations show that $\varkappa$ is close to $0.5$, but is not exactly equal to it.\footnote{After this article was posted to the \texttt{arXiv}, Defant settled this conjecture in \cite{defant2024randomsubwordspipedreams} and showed that $\varkappa=\dfrac{2\sqrt{2}}{3\sqrt{\pi}}$.}

The fact that the number of inversions lives on scale
$n^{3/2}$
implies a trivial permuton limit of~$\mathbf{w}$.
Denote
by $\mathrm{id}$ the
deterministic identity permuton
supported on the diagonal of $[0,1]^2$.

\begin{proposition}
	\label{prop:id_permuton_limit}
	For any fixed $p\in [0,1)$, the random permutations $\mathbf{w}\in S_n$ converge in distribution
	to $\mathrm{id}$ as $n\to\infty$.
\end{proposition}
\begin{proof}
	For any $\tau\in S_k$ and $w\in S_n$, $k< n$, denote by
	$t(\tau,w)$
	\cite[Section~3]{grubel2023ranks}
	the relative frequency of
	the pattern $\tau$ in $w$.
	This is simply the number of times the pattern $\tau$ appears in $w$,
	divided by $\binom nk$.
	In particular, $t( 21, w)=\frac{\inv(w)}{\binom n2}$.
	We have for any $2<\ell < n$
	\cite[(9)]{grubel2023ranks}:
	\begin{equation}
		\label{eq:pattern_recursion_proof}
		t(21, w) = \sum_{\tau\in S_\ell}t(21, \tau)\ssp t(\tau, w).
	\end{equation}
	The term in the sum with $\tau=\mathrm{id}$ vanishes since $t(21, \mathrm{id})=0$.
	For our random permutations
	$\mathbf{w}\in S_n$, the left-hand side of \eqref{eq:pattern_recursion_proof}
	converges to zero in probability by \Cref{prop:abs_value_sigma_i}
	and Markov inequality:
	\begin{equation*}
		\Prob\left(
			\inv(\mathbf{w}) \ge \delta n^{2}
		\right)\le C_\delta\ssp n^{-\frac{1}{2}}\to 0,\qquad n\to\infty
	\end{equation*}
	for all $\delta>0$.
	For a fixed $\ell$, this implies that
	$t(\tau,\mathbf{w})\to0$,
	$n\to\infty$,
	in probability
	for all $\mathrm{id}\ne \tau\in S_\ell$, $\ell\ge3$. We see that all relative frequencies
	converge to zero except for $t(\mathrm{id}, \mathbf{w})\to 1$.
	Therefore,
	by
	\cite[Theorem~2.5]{bassino2019universal},
	we get the result.
\end{proof}

This completes the proof of \Cref{thm:non_reduced_intro} from the Introduction.

\section{Maximal principal specializations of Grothendieck polynomials}
\label{sec:maximal_specializations_Grothendieck}

\subsection{Grothendieck polynomials via bumpless pipe dreams}
\label{sub:square_ice}

In this subsection, we outline another
combinatorial model for Schubert and Grothendieck polynomials
based on bumpless pipe dreams. Some of these definitions and results
go back to
\cite{Lascoux02ice}, and are described in detail in
\cite{Weigandt2020_bumpless} and \cite{LamLeeShimozono}.

A \emph{bumpless pipe dream} is a
tiling $D$ of the $n\times n$ square
with six types of tiles:
\begin{equation}
	\label{eq:bumpless_tiles}
\emph{SE bump}  \,\, \raisebox{-2pt}{\includegraphics[scale=0.4]{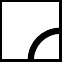}}\ssp, \quad \emph{NW bump} \,\, \raisebox{-2pt}{\includegraphics[scale=0.4]{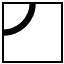}}\ssp, \quad \emph{cross}  \,\, \raisebox{-2pt}{\includegraphics[scale=0.4]{cross}}\ssp, \quad
\emph{empty}  \,\, \raisebox{-2pt}{\includegraphics[scale=0.4]{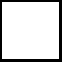}}\ssp, \quad
\emph{horizontal}   \,\, \raisebox{-2pt}{\includegraphics[scale=0.4]{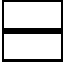}}\ssp, \quad \emph{vertical}  \,\, \raisebox{-2pt}{\includegraphics[scale=0.4]{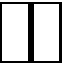}}\ssp.
\end{equation}
These tiles form a set of pipes labeled $1,\ldots,n$ going
from the bottom boundary (left to right) to the right
boundary. See \Cref{fig:bumpless_pipe_dream}, left, for an illustration.
We denote the set of all bumpless pipe dreams of size $n$ by $\BPD(n)$.

\begin{remark}
	\label{rmk:bumpless_pipe_dreams_six_vertex_ASMs}
	Bumpless pipe dreams are the same as
	configurations
	of the six-vertex model in the $n\times n$ square
	with domain wall boundary conditions (see \cite[Section
	3]{Weigandt2020_bumpless} or \cite{Lascoux02ice}).
	We refer to \cite[Section~2]{ZinnJustin20096Vertex}
	or \cite{reshetikhin2010lectures} for background and many properties
	on the six-vertex model. This object is also sometimes called
	the \emph{oscullating lattice paths} model \cite{Behrend2008}.

	The same configurations can be identified with
	alternating sign matrices (ASMs) of size $n$,
	see \cite{Bressoud1999} for a detailed exposition,
	and \cite{kuperberg1996another}, \cite[Section~2.5.6]{ZinnJustin20096Vertex} for applications
	of the six-vertex model to the enumeration of ASMs.
\end{remark}

\begin{figure}
  \centering
  \includegraphics{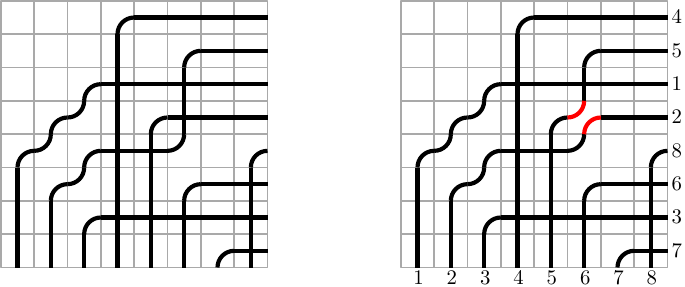}
	\caption{\textbf{Left}: Example of a (non reduced)
	bumpless pipe dream $D$. \textbf{Right}: The associated
	reduced bumpless pipe dream $D'$ of $D$.
	In this case, $D'$ is obtained by
	ignoring the second crossing of the pipes $2$ and $5$.
	The permutation is $w(D)=w(D')=45128637$.}
	\label{fig:bumpless_pipe_dream}
\end{figure}

\begin{definition}[{From bumpless pipe dreams to permutations; see \cite[Section~2.3]{Weigandt2020_bumpless}}]
	\label{def:bumpless_pipe_dream_permutation}
	A permutation $w(D)$ is associated to a bumpless pipe dream $D$
	by tracing the pipes from left to right, starting from the bottom boundary,
	and ignoring the second and subsequent crossings of each pair of pipes.
	See \Cref{fig:bumpless_pipe_dream}, right, for a reading of the
	bumpless pipe dream in the left part of the figure
	which leads to $w(D)$. The local rules for labeling the
	tiles in a bumpless pipe dream
	(which produce the permutation $w(D)$)
	are given in
	\Cref{fig:rules_label_bumpless}.
	Denote by $\BPD(w)\subset\BPD(n)$ the set of bumpless pipe dreams
	whose associated permutation is $w$.
\end{definition}

\begin{figure}
    \centering
    \includegraphics{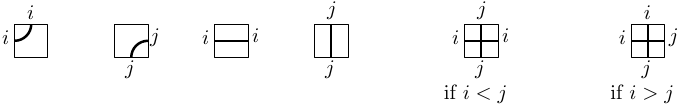}
    \caption{Labeling tiles in a bumpless pipe dream
		to read the associated permutation $w(D)$.}
		\label{fig:rules_label_bumpless}
\end{figure}

\begin{definition}
	A bumpless pipe dream is called \emph{reduced} if no two pipes
	cross twice.
	Denote the set of reduced bumpless pipe dreams of size $n$ by $\RBPD(n)$,
	and by $\RBPD(w)\subset \RBPD(n)$ the subset
	whose associated permutation is $w$.
\end{definition}

\begin{remark}
\label{rmk:Demazure_product_bpds}
	Alternatively, just as in pipe dreams (see \Cref{rmk:Demazure_product}), the permutation $w(D)$
	from a non-reduced bumpless pipe dream
	can be defined using the Demazure product, see \cite[\S 2.3]{Weigandt2020_bumpless} for details.
\end{remark}

For a bumpless pipe dream $D$ (reduced or not), denote by
$\NWbump(D)$ and $\emptytile(D)$ the set of NW bumps and empty tiles in $D$, respectively
(see \eqref{eq:bumpless_tiles} for the notation).
For a tile in a bumpless pipe dream, we denote its
horizontal and vertical coordinates by $(k,l)$, $1\le k,l\le n$.
Here $k$ increases from left to right, and $l$ increases from top to bottom.

\begin{theorem}[{\cite[Thm. 1.1]{Weigandt2020_bumpless}}] \label{thm:GrothPipeDreams_weigandt}
	For any permutation $w\in S_n$, we have
	\begin{equation} \label{eq: groth in terms of bpd}
		\mathfrak{G}^\beta_w(x_1,\ldots,x_n ) = \beta^{-\ell(w)} \sum_{D \in \BPD(w)} \prod_{(k,l)\in \emptytile(D)} \beta x_k \prod_{(k,l)\in \NWbump(D)} (1+\beta x_k).
	\end{equation}
	In particular, for $\beta=0$, we have
	for the Schubert polynomials:
	\begin{equation} \label{eq: schubs in terms of rbpd}
		\mathfrak{S}_w(x_1,\ldots,x_n ) = \sum_{D\in \RBPD(w)} \prod_{(k,l)\in \emptytile(D)} x_k.
	\end{equation}
\end{theorem}

\begin{corollary} \label{cor:partition function Grothendieck}
	Recall the notation \eqref{eq:v_n_u_n_notation_intro}.
	We have
	\begin{align}
	v_n(\beta) &=   \sum_{D \in \PD(n)} \beta^{\#\cross(D)-\ell(w(D))}, \label{eq:Groth pd partition function}\\
			&=
			\sum_{P \in \BPD(n)} \beta^{\#\emptytile(P)-\ell(w(P))}(1+\beta)^{\# \NWbump(P)},  \label{eq:Groth bpd partition function}
	\end{align}
	where $w(D)$ and $w(P)$ are the
	permutations associated to the pipe dream $D$ and bumpless pipe dream $P$, respectively.
\end{corollary}
\begin{proof}
	These identities follow from the formulas for Grothendieck polynomials in \Cref{thm:Grothendieck_pipe_dreams_intro,thm:GrothPipeDreams_weigandt}.
For \eqref{eq:Groth pd partition function}, we immediately get
\begin{align*}
	v_n(\beta) = \sum_w \mathfrak{G}_w^\beta(1^n) = \sum_{D \in PD(n)} \beta^{\#\cross(D)-\ell(w(D))}.
\end{align*}
For \eqref{eq:Groth bpd partition function}, we have
\begin{align*}
	v_n(\beta)&= \sum_w \mathfrak{G}_w^\beta(1^n) \\&=
	\sum_{P \in BPD(n)} \beta^{-\ell(w(P))} \sum_{D \in \BPD(w)} \prod_{(k,l)\in \emptytile(D)} \beta  \prod_{(k,l)\in \NWbump(D)} (1+\beta )\\&=
		\sum_{P \in BPD(n)}\beta^{\#\emptytile(P)-\ell(w(P))}(1+\beta)^{\# \NWbump(P)}.
\end{align*}
This completes the proof.
\end{proof}

\begin{remark}
Gao and Huang  \cite{GaoHuang} gave a bijection between reduced pipe dreams and reduced bumpless pipe dreams that preserves the corresponding weight of the diagrams in the corresponding expansions \eqref{eq: schubs in terms of rpd} and \eqref{eq: schubs in terms of rbpd} of Schubert polynomials. This bijection was recently generalized to the case of pipe dreams and bumpless pipe dreams with marked NW-bumps and Grothendieck polynomials by Huang--Shimozono--Yu \cite{HSY}.
\end{remark}

\begin{remark} \label{rem:size ASMs vs PD}
By the product
formula for the number of alternating sign matrices of size $n$
\cite{zeilberger1996proof}, \cite{kuperberg1996another}, we have
$\#\BPD(n) = \prod_{k=0}^{n-1} \frac{(3k+1)!}{(n+k)!}$. Moreover,
since $\RBPD(n)\subset \BPD(n)$, we have
\[
u_n(0) \leq v_n(0) \leq  \prod_{k=0}^{n-1} \frac{(3k+1)!}{(n+k)!}.
\]
From the asymptotics of the number of AMSs of order $n$
(which is relatively straightforward using the product formula),
we conclude that
\[
	\limsup_{n\to\infty} \frac{1}{n^2}\ssp \log_2 u_n(0)   \leq 0.37.
\]
The asymptotic behavior of the number of bumpless pipe dreams
(six-vertex model configurations)
with periodic boundary is originally due to
Lieb \cite{Lieb1967SixVertex}. The domain wall case
with general six-vertex parameters is treated
by Korepin--Zinn-Justin \cite{korepin2000thermodynamic}.
\end{remark}

\subsection{Asymptotics of Grothendieck random permutations via domino tilings}
\label{sub:square_ice_aztec_asymptotics}

Consider a random bumpless pipe dream $\mathbf{D}\in \BPD(n)$
with the distribution
\begin{equation}
	\label{eq:2_ASM_distribution}
	\Prob_{\textnormal{2-ASM}}(\mathbf{D}=D) = 2^{-\binom{n}{2}}\ssp 2^{\# \NWbump(D)},\qquad
	D\in \BPD(n).
\end{equation}
Via the bijection with Alternating Sign Matrices
(\Cref{rmk:bumpless_pipe_dreams_six_vertex_ASMs}),
this distribution is equivalent to choosing each alternating
sign matrix of size $n$ with probability
proportional to $2^{\#(\textnormal{number of $-1$'s})}$.
We call this distribution the \emph{2-enumerated ASM model},
and it is equivalent (via a two-to-one bijection) to
the \emph{uniform} probability distribution on
domino tilings of the Aztec diamond of size $n-1$.
This connection, as well as
the asymptotics of uniformly random domino tilings,
has been a subject of extensive research in
statistical mechanics and integrable probability,
see, for example,
\cite{cohn-elki-prop-96}, \cite{CohnKenyonPropp2000}, and
\cite{elkies1992alternating}, \cite{zinn2000six}, \cite{ferrari2006domino}.

The results from the previous
\Cref{sub:square_ice} immediately imply:
\begin{corollary}
	\label{cor:Grothendieck_bumpless_matching}
	The Grothendieck random permutation $\mathbf{w}\in S_n$
	with the parameter $p=\frac{1}{2}$
	(see \Cref{sub:random_permutations} for the definition)
	coincides with the permutation $w(\mathbf{D})$
	associated
	(as in \Cref{def:bumpless_pipe_dream_permutation})
	to the
	random bumpless pipe dream $\mathbf{D}\in \BPD(n)$
	with the distribution \eqref{eq:2_ASM_distribution}.
\end{corollary}

\begin{remark}
	\label{rmk:similar_general_p}
	For general $p$, an analog of \Cref{cor:Grothendieck_bumpless_matching}
	would involve a more complicated weighting of bumpless pipe dreams
	(six-vertex configurations). This weighting is still in the
	free-fermion family of six-vertex weights.
	At the level of domino tilings, the $p$-weighting corresponds to a
	deformation of the uniform model
	in which vertical and horizontal dominoes have different weights.
	This model is also well-studied,
	for example, see \cite{chhita2015asymptotic}.
	For simplicity, here we focus only on the case $p=\frac{1}{2}$.
\end{remark}

Our asymptotic results for
Grothendieck random permutations (\Cref{thm:Grothendieck_permutations_asymptotics_intro})
readily
translate into results about the behavior of the permutations
$w(\mathbf{D})$ coming from the 2-enumerated ASM model.
On the other hand, the asymptotic behavior of the bumpless pipe dream $\mathbf{D}$
in various regimes
has also been extensively studied,
and some of these results can be
brought back to Grothendieck random permutations.
Let us present just one example leading to asymptotically Gaussian fluctuations:

\begin{proposition}[Central limit theorem for the image of $n$]
	\label{prop:Gaussian}
	Let $\mathbf{w}\in S_n$ be the Grothendieck random permutation with $p=\frac{1}{2}$.
	Let $\mathbf{w}_n$ be the image of $n$ under $\mathbf{w}$.
	We have
	\begin{equation}
		\label{eq:CLT_Grothendieck}
		\frac{\mathbf{w}_n-n/2}{\sqrt{n/4}}
		\stackrel{d}{\longrightarrow} \mathcal{N}(0,1),\qquad
		n\to\infty.
	\end{equation}
	Here $\mathcal{N}(0,1)$ is the standard normal distribution.
\end{proposition}
\begin{proof}
	Let
	$\mathbf{l}_n$ be the position of the single NW bump in the
	rightmost column of $\mathbf{D}$.  In
	\Cref{fig:bumpless_pipe_dream}, we have
	$\mathbf{l}_n=5$.
	By \Cref{cor:Grothendieck_bumpless_matching}, $\mathbf{w}_n$ has the
	same distribution as $\mathbf{l}_n$.
	In the language of domino tilings
	(see the map in, e.g., \cite{ferrari2006domino}), $\mathbf{l}_n$ is the position of the
	first particle in the particle process associated to the domino tiling.
	By
	\cite[Theorem~1.5]{johansson2006eigenvalues},
	$\mathbf{l}_n$ satisfies a central limit theorem as in
	\eqref{eq:CLT_Grothendieck}. Note that our standard deviation
	$\sqrt{n/4}$ matches the normalization in \cite{johansson2006eigenvalues}
	as the latter work deals with GUE random matrices
	whose diagonal elements have variance $2$.
	This completes the proof.
\end{proof}

\begin{remark}
	\label{rmk:Gaussian_from_random_walks}
	\Cref{prop:Gaussian} can also be proven directly by looking at the trajectory of the
	pipe of color $n$ in the random pipe inside the staircase shape
	(\Cref{sub:pipe_dreams}), as this trajectory is a random walk.
\end{remark}

\subsection{Exact formulas for elementary layered permutations}
\label{sub:exact_formulas_layered}

We are now in a position to discuss principal specializations
of Grothendieck polynomials on layered permutations.
We begin with a few exact formulas, for which we need some notation.

Let $\mathsf{C}_n\coloneqq\frac{1}{n+1}\binom{2n}{n}$ denote the
$n$th \emph{Catalan number} which counts the number of
\emph{Dyck paths} of size $n$. Let $\ls_{n}$ denote the $n$-th
\emph{little Schr\"{o}der number}
\cite[\href{http://oeis.org/A001003}{A001003}]{oeis} and
$\bs_n$ the $n$-th \emph{large Schr\"oder number}
\cite[\href{https://oeis.org/A006318}{A006318}]{oeis}. It is
known that $\bs_n=2\ssp\ls_n$ for $n\geq 1$.

The \emph{Narayana polynomial} is defined as
\[
\mathcal{L}_{n}(x) \coloneqq \sum_{d \in \mathrm{Dyck}(n)} x^{v(d)},
\]
where the sum is over Dyck paths of size $n$ and $v(d)$ is the number of valleys of the path $d$. The coefficients of $\mathcal{L}_n(x)$ are the {\em Narayana numbers} $N(n,k)=\frac{1}{k}\binom{n}{k}\binom{n}{k-1}$ \cite[\href{http://oeis.org/A001263}{A001263}]{oeis}. Note that $\mathsf{C}_n=\mathcal{L}_n(1)$ and $s_n=\mathcal{L}_n(2)$. The polynomials $\mathcal{L}_n(x)$ have the following generating function:
\begin{equation} \label{eq:Narayana gf}
\sum_{n\geq 0} y^n \mathcal{L}_n(x) = \frac{1-y(1+x)-\sqrt{(1-y(1+x))^2-4xy^2}}{2y}.
\end{equation}

Consider the following
\emph{elementary layered permutation}
(see also \eqref{eq:cross_product_of_permutations} for notation):
\begin{equation}
	\label{eq:elementary_layered_permutation_definition_intro}
	 w_0(k;n)\coloneqq  \mathrm{id}_k \times w_0(n) = (1,\ldots,k, k+n,k+n-1,\ldots,k+1)
	 \in S_{k+n}.
\end{equation}
Recall the notation $\Upsilon_w(\beta)$
\eqref{eq:Upsilon_w_definition_intro}.
A
formula due to Fomin and Kirilov \cite{FK} relates the
Schubert principal specialization $\Upsilon_{w_0(a;b)}(0)$
to Proctor's formula \cite{Proctor90} counting bounded
plane partitions of staircase shape:

\begin{theorem}
	\label{thm:Proctor}
	For nonnegative integers $k$ and $n$, we have
    \begin{equation} \label{eq:proctor}
			\Upsilon_{w_0(k;n)}(0) = \mathfrak{S}_{w_0(k;n)}(1,\ldots,1) \,=\, \det[\mathsf{C}_{n-2+i+j}]_{i,j=1}^k \,=\, \prod_{1\leq i<j\leq n} \frac{2k+i+j-1}{i+j-1}.
\end{equation}
\end{theorem}

In \cite{MPP4GrothExcited}, an analogous formula for
the $\beta=1$
Grothendieck polynomials is proven:

\begin{theorem}[{\cite[Thm. 5.9]{MPP4GrothExcited}}]
	\label{thm:Groth_beta_1}
	For nonnegative integers $k$ and $n$, we have
	\begin{equation*}
			\Upsilon_{w_0(k;n)}(1) = \mathfrak{G}^{\beta=1}_{w_0(k;n)}(1,\ldots,1) = 2^{-\binom{k}{2}}\det[\ls_{n-2+i+j}]_{i,j=1}^k
		= 2^{-\binom{k+1}{2}}\det[\bs_{n-2+i+j}]_{i,j=1}^k.
	\end{equation*}
\end{theorem}

A general determinantal formula for all $\beta$ may be derived
from \cite[Thm.~5.9]{HKYY2019reverse}:
\begin{theorem}  \label{thm:det_formula_Grothendieck_beta}
	For nonnegative integers $k$ and $n$, we have
\begin{equation} \label{eq:Groth for w_0}
\Upsilon_{w_0(k;n)}(\beta) \,=\, (1+\beta)^{-\binom{k}{2}}\det[\mathcal{L}_{n+i+j-2}(1+\beta)]_{i,j=1}^k.
\end{equation}
\end{theorem}

The determinants
in
\Cref{thm:Groth_beta_1,thm:det_formula_Grothendieck_beta}
can be efficiently computed via the Dodgson condensation (the Desnanot--Jacobi identity).
We use it to get the numerical data for $\beta=1$
presented in
\Cref{app:data_Groth}.

\subsection{Asymptotically maximal specializations for layered permutations}
\label{sub:asymptotically_maximal_specializations}

Here we prove that principal
specializations of $\beta=1$
Grothendieck polynomials on layered permutations
(see \eqref{eq:layered_permutation_definition_intro} for the definition)
attain the asymptotically largest
value. That is, we prove \Cref{thm:gmax_intro} from the Introduction.

\begin{theorem}\label{thm:gmax}
Consider the compositions $b^{(n)} = (\ldots, b_2, b_1)$ of $n$
such that
$$
b_1 + \ldots + b_{i-1} + (2 + \sqrt{2})b_i \le n \textnormal{ for } i = 1,2,\ldots,
\textnormal{ while }b_1 + \ldots + b_{i-1} \le n - 4,
$$
and $b^{(n)}$ has $o(n)$ parts.
Then for the layered permutations $w(b^{(n)})$ we have
$$
\lim_{n \to \infty} \frac{1}{n^2} \log_2
\Upsilon_{w(b^{(n)})} (1)
= \frac{1}{2},
$$
which is asymptotically maximal.
\end{theorem}

\begin{remark}
	In particular, one can take compositions to be geometric
	$b_i \sim (1 - \alpha) \alpha^{i - 1} n$, for any $\alpha \in [1/\sqrt{2}, 1)$.
\end{remark}

In the rest of this \Cref{sub:asymptotically_maximal_specializations},
we prove \Cref{thm:gmax}. First, observe that the Grothendieck
specializations enjoy the following basic factorization property:
\begin{proposition}\label{prop:gprod}
    Let $u \in S_k, w \in S_n$. We have
    $$
		\Upsilon_{u \times w}(\beta) = \Upsilon_{u}(\beta) \cdot \Upsilon_{\mathrm{id}_k \times w}(\beta),
    $$
		where
		$u\times w$ is the block permutation
		defined by
		\eqref{eq:cross_product_of_permutations}.
\end{proposition}

\begin{proof}
	This follows from properties of pipe dreams
	corresponding to permutations
	of the form $u \times w$.
	See \cite[Corollary~2.4.6]{ManivelBook} or \cite[(4.6)]{MacdonaldSchubertBook}.
\end{proof}

Let us denote (recall the notation \eqref{eq:elementary_layered_permutation_definition_intro})
\begin{equation}
	\label{eq:F_G_notation_for_layered}
	F(k, n) \coloneqq  \Upsilon_{w_0(k;n)}(1), \qquad G(k,n) \coloneqq  2^{\binom{k}{2}} \cdot F(k,n)=
	\det[\ls_{n-2+i+j}]_{i,j=1}^k,
\end{equation}
where the last equality follows from \Cref{thm:Groth_beta_1}.

An important step in obtaining the asymptotics of $\Upsilon_w(1)$ for layered permutations is to understand the behavior of the function $F(k,n)$. A similar analysis \cite{MPP4GrothExcited} was possible in the Schubert case ($\beta = 0$) thanks to an explicit product formula (\Cref{thm:Proctor}). For other $\beta$'s, we have determinantal (but not product) formulas (\Cref{thm:Groth_beta_1} and \Cref{thm:det_formula_Grothendieck_beta}). However, in the case $\beta=1$ we are able to asymptotically analyze these determinants via the correspondence with 2-enumerated ASMs.

\begin{lemma}\label{lemma:fbound}
Let
$k \in (n/\sqrt{2}, n]$. We have
\begin{equation}
	\label{eq:fbound}
\log_2 F(k, n - k) = \frac{n^2}{2} - \frac{k^2}{2} - O(n), \qquad  n \to \infty.
\end{equation}
\end{lemma}
In \eqref{eq:fbound} and throughout the proofs below in this subsection,
we assume that the constants in the $O(\cdot)$ notation are nonnegative.
\begin{proof}[Proof of \Cref{lemma:fbound}]
	By \eqref{eq:F_G_notation_for_layered}, it suffices to show that
	for $k>n/\sqrt 2$, we have
	\begin{equation}
		\label{eq:determinant_NILP}
		\det[\ls_{n-k-2+i+j}]_{i,j=1}^k = 2^{\frac{n^2}{2}-O(n)}.
	\end{equation}
	We employ the interpretation of the litte Schr\"oder number $\ls_m$
	\cite[\href{http://oeis.org/A001003}{A001003}]{oeis}
	as the
	number of Motzkin paths from $(0,0)$ to $(2m,0)$
	(with steps $(1,1)$, $(1,-1)$, and $(2,0)$)
	which do not have horizontal steps at height $0$.
	Let us call these paths the \emph{$0$-Sch\"oder paths},
	where $0$ represents the lower boundary of the path.
	Then,
	by the Lindstrom--Gessel--Viennot lemma
	\cite{KMG59-Coincidence},
	\cite{lindstrom1973vector},
	\cite{gessel1985binomial}
	the determinant
	\eqref{eq:determinant_NILP}
	counts $k$-tuples of nonintersecting $0$-Sch\"oder paths
	starting from $(-2i,0)$
	and ending at
	$(2(n-k+i-1),0)$, respectively, where $i=1,\ldots,k $.
	Denote this space of configurations by $\mathcal{M}_{k,n-k}$.
	See \Cref{fig:motzkin_paths_aztec} for an illustration.

	\begin{figure}[htpb]
		\includegraphics[width=1\textwidth]{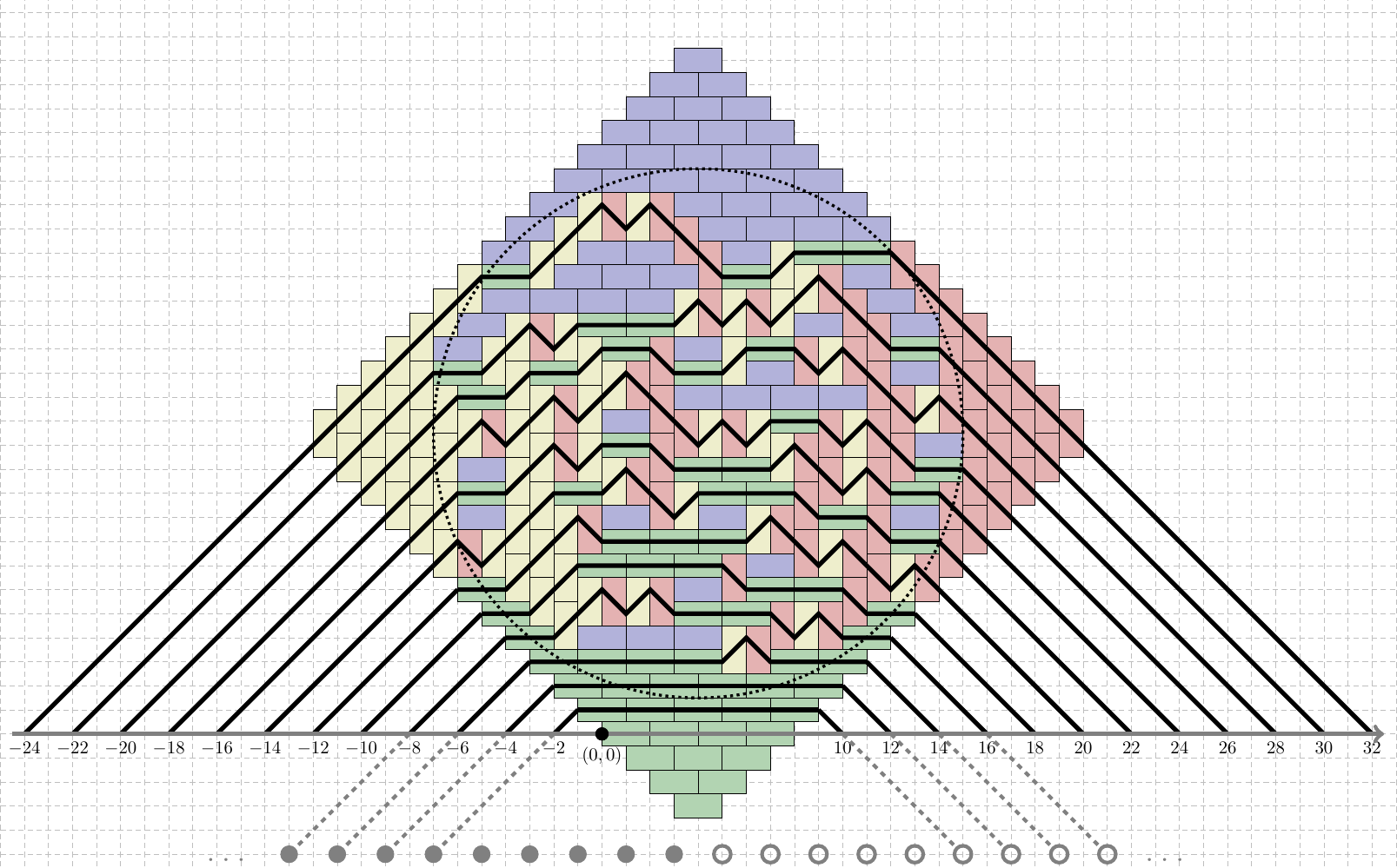}
		\centering
		\caption{Nonintersecting $0$-Sch\"oder paths
		from $\mathcal{M}_{12,5}$,
		the corresponding Aztec diamond $D_{n-1}$ of order $n-1=16$,
		and the arctic circle. The horizontal coordinate line is at the
		vertical height $0$. The gray points
		at the bottom are the starting (filled) and ending (hollow) points
		of the $(k+1-n)$-Schr\"oder paths representing the
		full domino tiling of the Aztec diamond.
		The dashed gray lines are the continuations of some of the paths from $\mathcal{M}_{k,n-k}$
		into paths from the ensemble $\mathcal{S}_n$.}
		\label{fig:motzkin_paths_aztec}
	\end{figure}

	A well-known
	correspondence (e.g., see \cite{eu2005simple})
	between domino tilings of the Aztec diamond
	and nonintersecting paths allows to interpret
	\begin{equation}
		\label{eq:determinant_NILP_count}
		\frac{\det [\ls_{n-k-2+i+j}]_{i,j=1}^k}{2^{\binom n2}}
	\end{equation}
	as a probability of a certain event in the uniformly random
	domino tiling of the Aztec diamond $D_{n-1}$ of order $n-1$.
	Let us explain this identification in detail.

	The total number of domino tilings of $D_{n-1}$ is
	$2^{\binom n2}$.
	Domino tilings of $D_{n-1}$ bijectively correspond to
	families of $n$ nonintersecting paths,
	for which we may take the $(k+1-n)$-Schr\"oder paths
	starting
	from $(n-k-2i,k+1-n)$, $i=1,\ldots,n$,
	and ending at $(n-k+2(i-1),k+1-n)$,
	respectively, where $i=1,\ldots,n$.
	Denote this ensemble of nonintersecting paths by $\mathcal{S}_n$,
	so $|\mathcal{S}_n|=2^{\binom n2}$.

	The $k$ paths from $\mathcal{M}_{k,n-k}$ may be continued diagonally
	to coincide with the outermost $k$ paths from the ensemble $\mathcal{S}_n$.
	This continuation procedure implies that the ratio
	$|\mathcal{M}_{k,n-k}|/|\mathcal{S}_n|$
	given by
	\eqref{eq:determinant_NILP_count}
	is equal to the probability that the $n-k$ paths
	in $\mathcal{S}_n\setminus \mathcal{M}_{k,n-k}$
	are in their \emph{lowest possible configuration}.
	See \Cref{fig:motzkin_paths_aztec} for an illustration.

	The model of uniformly random domino tilings of the Aztec diamond
	develops an arctic circle \cite{jockusch1998random}. In
	particular,
	the configuration outside of the circle
	inscribed in the Aztec diamond (illustrated in \Cref{fig:motzkin_paths_aztec})
	is frozen (nonrandom).
	Qualitatively,
	by \cite[Proposition~13]{cohn-elki-prop-96},
	this means that with probability $1-e^{-O(n)}$,
	the $n-k$ paths in $\mathcal{S}_n\setminus \mathcal{M}_{k,n-k}$
	are indeed in their lowest possible configuration.
	Note that here we rely on the assumption $k>n/\sqrt 2$,
	which guarantees that the top of the $n-k$ paths
	in $\mathcal{S}_n\setminus \mathcal{M}_{k,n-k}$
	does not reach the arctic circle.
	This establishes \eqref{eq:determinant_NILP},
	and completes the proof of \Cref{lemma:fbound}.
\end{proof}

\begin{remark}
	\label{rmk:colomo_pronko}
	The recent work \cite{colomo2024scaling}
	provides asymptotics of
	$\log_2 F(k, n - k)$ for all $k$, not necessarily in the range $k>n/\sqrt{2}$.
	For $k<n/\sqrt{2}$, the leading term
	is strictly smaller than~$\frac{n^2}2$,
	which agrees with the variational principle
	\cite[Theorem~1.3]{CohnKenyonPropp2000}.
	We do not need the regime $k<n/\sqrt{2}$ for \Cref{thm:gmax}.
\end{remark}

\begin{proof}[Proof of \Cref{thm:gmax}]
Let us denote $k_0 \coloneqq n$ and
$$
k_i \coloneqq n - b_1 - \ldots - b_i,
\quad i \ge 1,
$$
so that $b_i = k_{i-1} - k_i$.
We have by \Cref{prop:gprod} and \eqref{eq:elementary_layered_permutation_definition_intro} that
\begin{align} \label{eq:layered}
\Upsilon_{w(\ldots, b_2, b_1)}(1) &= \Upsilon_{w(\ldots, b_3, b_2)}(1) \cdot F(n-b_1,b_1)  \notag \\
&= \prod_{i \ge 1} F(n - b_1 - \ldots - b_i, b_i) = \prod_{i \ge 1} F(k_i, k_{i-1} - k_i).
\end{align}
From the inequalities on the $b_i$'s, we get
$$
\frac{k_i}{k_{i-1}} = \frac{n - b_1 - \ldots - b_i}{n - b_1 - \ldots - b_{i-1}}
\in (1/\sqrt{2}, 1).$$
Hence by \Cref{lemma:fbound}, we have
$$
\log_2 F(k_i, k_{i-1} - k_i) = \frac{k_{i-1}^2}{2} - \frac{k_i^2}{2} - O({k}_{i-1}).
$$
Therefore,
\begin{align*}
\frac{1}{n^2} \log_2 \Upsilon_{w(\ldots, b_2, b_1)}(1,\ldots, 1) &= \sum_{i \ge 1} \frac{1}{n^2} \log_2 F(k_i, k_{i-1} - k_i) \\
&= \sum_{i \ge 1} \frac{1}{n^2} \left(\frac{k_{i-1}^2}{2} - \frac{k_i^2}{2} - O(k_{i-1}) \right) \\
%&= \sum_{i \ge 1} \frac{n_i^2}{2 n^2} - \\
&= \frac{1}{2} - \frac{1}{n^2} O\left(\sum_{i \ge 1} k_{i-1} \right).
%= \sum_{i \ge 1} \log_2 F(b_{i+1} + b_{i+2} + \cdots, b_i).
\end{align*}
%We can assume that in our setting we have
Since $k_{i-1} = O(n)$ and there are $o(n)$ elements in our composition, we have
$$
\sum_{i \ge 1} k_{i} = o(n^2), \text{ and so }  \frac{1}{n^2} O\left(\sum_{i \ge 1} k_{i-1}  \right) = o(1),
$$
%$k_{i-1} = b_{i} + b_{i+1} + \cdots$  and $\sum_{i \ge 1} i b_i = o(n^2)$
%we obtain that
%$$\frac{1}{n^2} O\left(\sum_{i \ge 1} k_{i-1}  \right) = \frac{1}{n^2} O\left(\sum_{i \ge 1} i b_i \right)
%= o(1),
%= \frac{\log n}{n} O\left(\sum_{i \ge 1} i \alpha_i \right) = O\left(\frac{\log n}{n} \right)
%$$
%(e.g. by taking the number of terms in a composition as $O(\log n)$)
which completes the proof of \Cref{thm:gmax}.
\end{proof}

\begin{remark}[Layered vs all permutations]
	\label{rmk:layered global max}
	Recall the notation
	\eqref{eq:v_n_u_n_notation_intro},
	\eqref{eq:u_n_prime_definition_intro}.
	\Cref{thm:gmax} establishes that on layered permutations,
	the Grothendieck specializations $\Upsilon_w(1)$
	reach their asymptotic maximum: $u_n'(1)\sim 2^{\binom n2}$.
	However, it is not clear whether
	the absolute maximum, $u_n(1)$,
	is attained on layered permutations.
	Dennin
	\cite[Fig. 9]{dennin2022} verified that for $n\le 9$, one has
	$u_n(1)=u_n'(1)$.

	We present numerical data of
	optimal (i.e., achieving the exact, non-asymptotic maximum in $u_n'(1)$)
	layered
	permutations for $n\leq 500$ in
	\Cref{app:data_Groth}.
	We also comment on why these optimal layered permutations
	have parts whose behavior does not fall under the assumptions of \Cref{thm:gmax}.
\end{remark}

\subsection{Bounds on the constant for Grothendieck specializations with general $\beta$}
\label{sub:bounds_Groth}

In \cite{MPP4GrothExcited}, the authors conjectured that for $\beta$ fixed there exists a limit
\begin{equation}
	\label{eq:c_beta_defn}
	c(\beta)\coloneqq\lim_{n\to \infty} \frac{1}{n^2}\log_2 \upsilon_n(\beta).
\end{equation}
If the limit exists, it must satisfy
for $\beta\geq 0$:
\begin{equation} \label{eq:bounds from mpp4}
	\frac14 \log_2(2+\beta) \leq c(\beta) \leq \frac12 \log_2(2+\beta),
\end{equation}
see \cite[Section~7.12]{MPP4GrothExcited}.
The limit
\eqref{eq:c_beta_defn}
exists when $\beta=1$, see \eqref{eq:Grothendieck_specialization_v_n_u_n_intro},
and we have $c(1)=\frac{1}{2}$.
Here we
improve the bounds \eqref{eq:bounds from mpp4} in certain regimes:

\begin{proposition} \label{prop: new bounds from last square}
	If the limit $c(\beta)$ \eqref{eq:c_beta_defn} exists, then it must satisfy the following bounds.
	For $0 < \beta \leq 1$, we have
	\begin{equation*}
            \frac{1}{4} \max\{\log_2(2+\beta), 2\log_2(1+\beta)\} \leq c(\beta) \leq \frac{1}{2} \min\{\log_2(2+\beta), \log_2(1+1/\beta)\}.
					\end{equation*}
For $\beta \geq 1$, we have
\begin{equation*}
            \frac{1}{4} \max\{\log_2(2+\beta), 2\log_2(1+1/\beta)\} \leq c(\beta) \leq \frac{1}{2} \log_2(1+\beta).
					\end{equation*}
				See \Cref{fig:bounds_last_square} for an illustration of the bounds.
\end{proposition}
\begin{proof}
For $0<\beta \leq 1$ and
$D\in \PD(n)$, we can bound $ 1 \leq \beta^{-\ell(w(D))} \leq \beta^{-\binom{n}{2}}$.
Thus by \eqref{eq:Groth pd partition function}, we have
\begin{align*}
\sum_{D \in \PD(n)} \beta^{\#\cross(D)} \leq v_n(\beta) \leq \beta^{-\binom{n}{2}} \sum_{D \in \PD(n)} \beta^{\#\cross(D)},
\end{align*}
which leads to
$(1+\beta)^{\binom{n}{2}} \leq v_n(\beta) \leq \bigl(\frac{1+\beta}{\beta}\bigr)^{\binom{n}{2}}$.
Taking the logarithms and comparing with \eqref{eq:bounds from mpp4} gives the result.

For $\beta\geq 1$ and $D\in \PD(n)$, let us estimate
$ \beta^{-\binom{n}{2}} \leq \beta^{-\ell(w(D))} \leq 1$. Thus by \eqref{eq:Groth pd partition function}, we have
\begin{align*}
	\beta^{-\binom{n}{2}} \sum_{D \in \PD(n)} \beta^{\#\cross(D)} \leq v_n(\beta) \leq \sum_{D \in \PD(n)} \beta^{\#\cross(D)},
\end{align*}
yielding
$\bigl(\frac{1+\beta}{\beta}\bigr)^{\binom{n}{2}}\leq v_n(\beta) \leq (1+\beta)^{\binom{n}{2}}$,
which leads to the
desired bounds.
\end{proof}

\begin{figure}
\centering
    \includegraphics[width=0.3\textwidth]{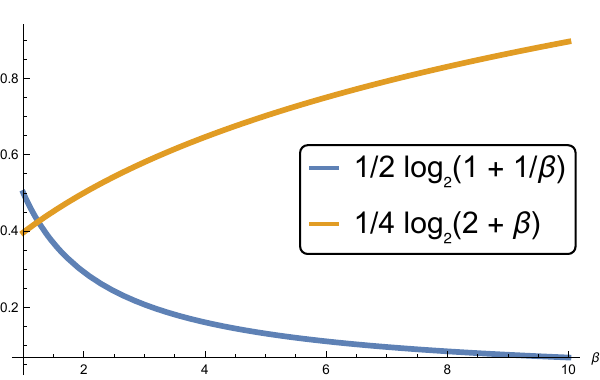} \quad   \includegraphics[width=0.3\textwidth]{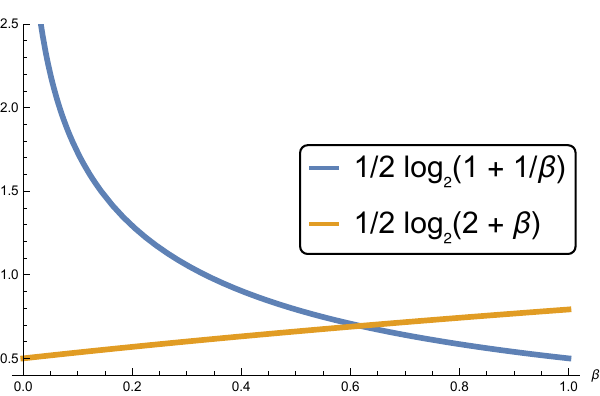} \quad
\includegraphics[width=0.3\textwidth]{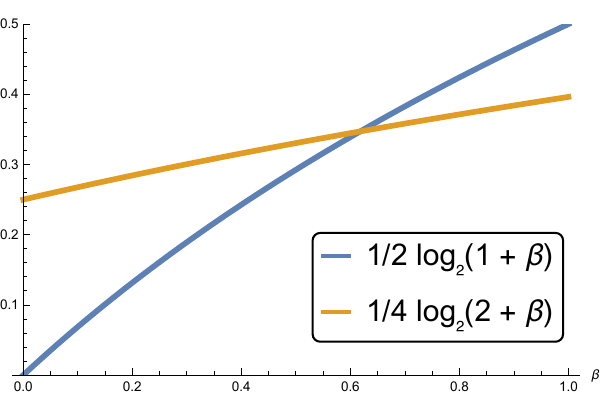}
    \caption{Comparing the bounds in \Cref{prop: new bounds from last square} for
		$\beta>1$ and $0<\beta<1$.}
		\label{fig:bounds_last_square}
\end{figure}

\section{Additional remarks and open problems}
\label{sec:add_remarks_and_open_problems}

\subsection{Optimal permutations achieving the exact maxima $u_n(1)$ and $u_n(\beta)$}

This project was motivated by Stanley's original question
from~\cite{stanley2017some} on determining the family of
permutations $w$ for which the Schubert specialization $\mathfrak{S}_w(1^n)$ is
asymptotically maximal. This question as well as the
asymptotic value remain widely open (see \cite[Table
1]{merzon2016determinantal} for data up to $n\leq 10$).
Dennin \cite[Section~4.2, Fig. 8]{dennin2022} extended this
question to Grothendieck specializations $\mathfrak{G}_w^{\beta=1}(1^n)$,
but one can ask the same question for general $\beta$.
Presumably, answering this
question would give rise to a family of optimal permutations
$w^\beta$, which for $\beta=0$ would answer the original
Schubert specialization question.

Among the candidates for
such a family $w^\beta$ are the layered permutations
\cite{merzon2016determinantal},
\cite{MoralesPakPanova2019}
(see
\Cref{app:data_Groth} for numerical evidence for $\beta=1$).
When $\beta=1$,
Dennin verified that
layered permutations indeed achieve the maximum in $u_n(1)$ for $n\leq 9$.
Our results of
\Cref{sub:asymptotically_maximal_specializations} suggest that layered
permutations are indeed good candidates
as they at least attain the asymptotic maximum in $u_n(1)$.

To obtain the asymptotics of Grothendieck specializations
on layered permutations for general $\beta$, one
has to study the asymptotics of the following Hankel
determinants:
\begin{equation}
	\label{eq:layered_Groth_det_general_beta}
	F(k,n,\beta)\coloneqq(1+\beta)^{-\binom{k}{2}}\det[\mathcal{L}_{n+i+j-2}(1+\beta)]_{i,j=1}^k,
\end{equation}
where $\mathcal{L}_n(x)$ is the Narayana polynomial, see \eqref{eq:Narayana gf}
and \Cref{thm:det_formula_Grothendieck_beta}.
For $\beta=0$, these determinants admit a product formula
and were analyzed in \cite{MPP4GrothExcited}.
We could access the special case $\beta=1$ due to the connection to
2-enumerated ASMs and domino tilings.
It would be interesting to find a deformation of the
domino tiling model related to the determinants \eqref{eq:layered_Groth_det_general_beta}.

\subsection{Typical vs optimal permutations}

One of our main results, \Cref{thm:Grothendieck_permutations_asymptotics_intro},
concerns ``typical'' Grothendieck
permutations.
By comparing the
permutation matrices in \Cref{fig:intro_simulations}, center,
and \Cref{fig:value_ps_Groth_layered}, left, we see that
typical permutations are far from layered.
However,
it is likely
that probabilities of typical Grothendieck random permutations
also achieve the asymptotic maximum.
For our permutations, this is not yet a rigorous statement,
as one has to prove the asymptotic uniformity of the probability weights
of typical Grothendieck permutations.

\subsection{Measures on permutations generated by elementary transpositions}
\label{sub:measures_on_permutations_oriented_swap}

Other Markov chains (and, more generally, stochastic processes on the symmetric
group) have been studied where
adjacent
permutations are connected
by an elementary transposition $s_i=(i,i+1)$, $i=1,\ldots,n-1 $.
These include
the oriented swap
process
\cite{angel2009oriented}, \cite{bufetov2022absorbing},
\cite{Zhang2023} which
is a natural dynamics on permutations
driven by independent nearest neighbor swaps,
which swap only pairs of indices that has not been swapped before.
The process starts from the identity and terminates at the full
reversal permutation.
Another model,
random sorting networks
\cite{angel2007randomSorting}, \cite{gorin2019random},
\cite{dauvergne2022archimedean}, is
obtained by putting a uniform probability measure
on all reduced words representing the full reversal permutation.
Random sorting networks cannot be
described by a simple Markov chain on permutations.

By design, both random sorting networks and the oriented swap
process terminate at the full reversal permutation,
and one is interested in the trajectories of individual elements.
In contrast, out Grothendieck random permutations are
obtained by running a Markov chain (reminiscent of the oriented
swap process) for a set amount of ``time'',
reading off the resulting permutation, and investigating its asymptotic properties.
Despite the differences in the setup,
in this work we observed certain asymptotic similarities
between the Grothendieck random permutations
and the models mentioned above.

\subsection{Local structure of Grothendieck permutations}

What is the local structure of the Grothendieck random permutations?
In \Cref{fig:intro_simulations}, left and center, one
can observe a certain local layered structure.
Understanding this local structure goes beyond the permuton convergence.
Can the local structure be
understood in terms of descent statistics, major index, or
pattern statistics other than the overall pattern counts?
For instance, layered permutations are the
ones maximizing the \emph{packing density} of the
pattern 132 \cite[Section~8.3.1]{Kit}.
Potential tools
for this analysis may be found in
\cite{borga2020local}. Alternatively,
one could translate the local behavior of random domino
tilings of the Aztec diamond \cite{Johansson2005arctic}
to the local behavior of Grothendieck permutations.

\subsection{The RS shape of a Grothendieck permutation}

Applying the Robinson--Schensted (RS) correspondence to large random
permutations, we numerically observe that the first column of the resulting
Young diagram grows linearly with $n$, while all other rows and column are of
order $\sqrt{n}$.\footnote{For layered permutations, all columns grow proportionally to $n$, while
the rows are of order $\log n$.}
This behavior (except for the first column)
is reminiscent of the limit shape phenomenon for Plancherel random
partitions exhibiting the
Vershik--Kerov--Logan--Shepp limit shape
\cite{VershikKerov_LimShape1077},
\cite{logan_shepp1977variational}.
It would be interesting to make these observations
rigorous, and to find the limit shape for the Grothendieck random permutations.
We are grateful to Philippe Biane for this suggestion.

\subsection{Random and maximal permutations for domain wall six-vertex model}

The Bumpless Pipe Dreams provide a surprising connection
of Schubert calculus and $K$-theory to
the six-vertex model, one of the central objects in
statistical mechanics.
The analysis of typical and maximal permutations for Schubert
and Grothendieck specializations prompts the analogous questions
about the six-vertex model.

Namely, consider the six-vertex
model with
domain wall boundary conditions and
general Boltzmann weights $a,b,c$ (for the 2-enumerated
ASMs, we have $a=b=1$ and $c=\sqrt{2}$).
Represent a random configuration of the six-vertex model
as a collection of pipes, and associate to it a random
permutation $\mathbf{w}$ by resolving the intersections
as in
\Cref{def:bumpless_pipe_dream_permutation}.
What is the behavior of $\mathbb{E}[\inv(\mathbf{w})]$ and
how does it vary depending on the weights of the six-vertex
model? Do the permutations $\mathbf{w}$ have a permuton limit?

\subsection{$q$-analogues and generalized principal specialization}

Zhang \cite{zhang2023qanalogue} extended the analysis in \cite{MoralesPakPanova2019} to the \emph{generalized principal Schubert specialization} $\mathfrak{S}_w(1,q,q^2,\ldots)$, when $q$ is a  root of unity. The analogue of layered permutations is the \emph{multi-layered} ones \cite[Section~4]{zhang2023qanalogue}. It would be interesting to extend our results to the Grothendieck specializations
	$\mathfrak{G}^{\beta}_w(1,q,q^2,\ldots)$ at a root of unity.
	However, having $q$ complex makes connections to probabilistic models less clear.

\appendix

\section{Data for maximal specializations on layered permutations}
\label{app:data_Groth}

\subsection{Dodgson condensation and data}

Principal specializations $\Upsilon_{w(b)}(\beta)$ of
Grothendieck polynomials with general $\beta$ on layered permutations
(see \Cref{sub:intro_Groth_spec} for notation)
can be maximized using dynamic programming, similarly to how this is done
for Schubert specializations in \cite{MoralesPakPanova2019}.
To implement it, one needs an efficient way to compute the
determinants
$D_{\beta}(n,k)\coloneqq \det[\mathcal{L}_{n+i+j-2}(1+\beta)]_{i,j=1}^k$
from
\Cref{thm:det_formula_Grothendieck_beta}.
A recursion
follows from
the Dodgson condensation
(Desnanot--Jacobi identity)
\cite{Dodgson1866}, \cite{Muir1906}:

\begin{proposition}
	For $k,n\in \mathbb{Z}_{\ge0}$,
	we have
\[
D_\beta(n,k) \,=\, \frac{1}{D_\beta(n+2,k-2)}\left(D_\beta(n+2,k-1)D_\beta(n,k-1) - D_\beta(n+1,k-1)^2\right),
\]
with initial conditions $D_\beta(n,0)=1$, $D_\beta(n,1)=\mathcal{L}_n(1+\beta)$, and $D_\beta(0,k)=(1+\beta)^{\binom{k}{2}}$.
\end{proposition}
\begin{proof}
	Let $M(n,k)=[\mathcal{L}_{n+i+j-2}(1+\beta)]_{i,j=1}^k$, so
	$D_{\beta}(n,k)=\det M(n,k)$. Clearly,
	\begin{equation*}
		D_{\beta}(n,0)=1,\qquad
		D_{\beta}(n,1)=\mathcal{L}_n(1+\beta).
	\end{equation*}
	For $n=0$, the
	permutation $w_0(0;k)=12\cdots k$ is the identity, and so
	$\mathfrak{G}^\beta_{w_0(0;k)}=1$. By
	\Cref{thm:det_formula_Grothendieck_beta}, we have
\[
1 = \Upsilon_{w_0(n,k)}(\beta) = (1+\beta)^{-\binom{k}{2}} \det[\mathcal{L}_{i+j-2}(1+\beta)]_{i,j=1}^k.
\]
Thus, $D_\beta(0,k)=\det M(0,k)=(1+\beta)^{\binom{k}{2}}$.

For $1\leq i,j,p,q\leq k$, let $M_{i,j}^{p,q}$ denote the matrix obtained from $M$
by deleting rows $i$ and $j$ and columns $p$ and $q$.
Similarly, let $M_i^p$ be obtained from $M$ by deleting row $i$ and column $p$.
By the Dodgson condensation, we have
\[
\det(M(n,k))\det(M(n,k)^{1,k}_{1,k}) \,=\, \det(M(n,k)^1_1)\det(M(n,k)^k_k) - \det(M(n,k)^k_1)\det(M(n,k)^1_k).
\]
Observe that
\begin{align*}
	M(n,k)^{1,k}_{1,k} &= M(n+2,k-2),\quad M(n,k)^1_1 = M(n+2,k-1),\\ \quad M(n,k)^k_k&=M(n,k-1),\quad M(n,k)^k_1=M(n,k)^1_k=M(n+1,k-1).
\end{align*}
This implies the result.
\end{proof}

The implementation of the recursion for $D_\beta(n,k)$ combined with dynamic programming
yields the data in \Cref{tab:layered_permutations}.
This data agrees with the asymptotic behavior $f(n)\to\frac{1}{2}$
obtained in \Cref{thm:gmax}.

\begin{table}[htbp]
	\scalebox{.96}{\small
\begin{minipage}{0.42\textwidth}
$$
\begin{array}{rrr}
\hline
n & (\ldots,b_2,b_1) & \qquad f(n)\\ \hline
2&(1, 1)&0.00000\\
3&(1, 2)&0.17611\\
4&(1, 3)&0.21621\\
5&(1, 1, 3)&0.24599\\
6&(1, 2, 3)&0.28068\\
7&(1, 2, 4)&0.31068\\
8&(1, 2, 5)&0.32354\\
9&(1, 3, 5)&0.33953\\
10&(1, 1, 3, 5)&0.34821\\
11&(1, 1, 3, 6)&0.35956\\
12&(1, 2, 3, 6)&0.36955\\
13&(1, 2, 4, 6)&0.37800\\
14&(1, 2, 4, 7)&0.38614\\
15&(1, 2, 4, 8)&0.39085\\
16&(1, 2, 5, 8)&0.39618\\
17&(1, 3, 5, 8)&0.40138\\
18&(1, 3, 5, 9)&0.40550\\
19&(1, 1, 3, 5, 9)&0.40887\\
20&(1, 1, 3, 6, 9)&0.41252\\
21&(1, 1, 3, 6, 10)&0.41605\\
22&(1, 2, 3, 6, 10)&0.41946\\
23&(1, 2, 4, 6, 10)&0.42223\\
24&(1, 2, 4, 6, 11)&0.42517\\
25&(1, 2, 4, 7, 11)&0.42797\\
26&(1, 2, 4, 7, 12)&0.43021\\
27&(1, 2, 4, 8, 12)&0.43206\\
28&(1, 2, 5, 8, 12)&0.43392\\
29&(1, 2, 5, 8, 13)&0.43590\\
30&(1, 3, 5, 8, 13)&0.43780\\
40&(1, 2, 4, 6, 10, 17)&0.45099\\
50&(1, 3, 5, 8, 13, 20)&0.45956\\
60&(1, 1, 3, 6, 10, 15, 24)&0.46537\\
70&(1, 2, 4, 7, 11, 18, 27)&0.46983\\
80&(1, 2, 5, 8, 13, 20, 31)&0.47312\\
90&(1, 1, 3, 5, 9, 14, 23, 34)&0.47573\\
100&(1, 2, 3, 6, 10, 16, 25, 37)&0.47792\\
110&(1, 2, 4, 7, 11, 17, 27, 41)&0.47975\\
120&(1, 2, 4, 8, 12, 19, 30, 44)&0.48125\\
 \hline
\end{array}
$$
\end{minipage}
\qquad
\begin{minipage}{0.55\textwidth}
$$
\begin{array}{rrr}
\hline
n & (\ldots,b_2,b_1) & \quad f(n)\\ \hline
130&(1, 3, 5, 8, 13, 21, 32, 47)&0.48255\\
140&(1, 1, 3, 5, 9, 14, 22, 34, 51)&0.48367\\
150&(1, 1, 3, 6, 10, 15, 24, 36, 54)&0.48465\\
160&(1, 2, 4, 6, 10, 16, 25, 38, 58)&0.48552\\
170&(1, 2, 4, 7, 11, 17, 27, 41, 60)&0.48631\\
180&(1, 2, 4, 7, 12, 19, 28, 43, 64)&0.48701\\
190&(1, 2, 5, 8, 12, 20, 30, 45, 67)&0.48763\\
200&(1, 3, 5, 8, 13, 21, 32, 47, 70)&0.48820\\
210&(1, 3, 5, 9, 14, 22, 33, 49, 74)&0.48871\\
220&(1, 1, 3, 5, 9, 15, 23, 34, 52, 77)&0.48917\\
230&(1, 1, 3, 6, 10, 15, 24, 36, 54, 80)&0.48961\\
240&(1, 2, 3, 6, 10, 16, 25, 38, 56, 83)&0.49001\\
250&(1, 2, 4, 6, 10, 17, 26, 39, 59, 86)&0.49038\\
260&(1, 2, 4, 7, 11, 17, 27, 41, 61, 89)&0.49072\\
270&(1, 2, 4, 7, 12, 18, 28, 42, 63, 93)&0.49104\\
280&(1, 2, 4, 8, 12, 19, 29, 44, 65, 96)&0.49134\\
290&(1, 2, 5, 8, 12, 20, 30, 45, 68, 99)&0.49161\\
300&(1, 3, 5, 8, 13, 20, 31, 47, 70, 102)&0.49187\\
310&(1, 3, 5, 8, 14, 21, 32, 48, 72, 106)&0.49211\\
320&(1, 3, 5, 9, 14, 22, 33, 50, 74, 109)&0.49234\\
330&	(1, 1, 3, 5, 9, 14, 23, 34, 52, 76, 112)&	0.49256\\
340	&(1, 1, 3, 6, 9, 15, 23, 35, 53, 79, 115)&	0.49276\\
350&	(1, 1, 3, 6, 10, 15, 24, 36, 55, 81, 118)&	0.49295\\
360&	(1, 2, 3, 6, 10, 16, 25, 37, 56, 83, 121)&	0.49314\\
370&	(1, 2, 4, 6, 10, 16, 25, 38, 58, 85, 125)&	0.49331\\
380&	(1, 2, 4, 6, 11, 17, 26, 39, 59, 87, 128)&	0.49347\\
390&	(1, 2, 4, 7, 11, 17, 27, 41, 60, 89, 131)&	0.49363\\
400&	(1, 2, 4, 7, 11, 18, 27, 42, 62, 92, 134)&	0.49378\\
410&	(1, 2, 4, 7, 12, 18, 28, 43, 64, 94, 137)&	0.49392\\
420&	(1, 2, 4, 7, 12, 19, 29, 44, 65, 96, 141)&	0.49406\\
430&	(1, 2, 5, 8, 12, 19, 30, 45, 67, 98, 143)&	0.49419\\
440&	(1, 2, 5, 8, 13, 20, 30, 46, 68, 100, 147)&	0.49431\\
450&	(1, 3, 5, 8, 13, 20, 31, 47, 70, 102, 150)&	0.49443\\
460&	(1, 3, 5, 8, 13, 21, 32, 48, 71, 105, 153)&	0.49454\\
470&	(1, 3, 5, 8, 14, 21, 33, 49, 73, 107, 156)&	0.49465\\
480&	(1, 3, 5, 9, 14, 22, 33, 50, 74, 109, 160)&	0.49476\\
490&	(1, 1, 3, 5, 9, 14, 22, 34, 51, 76, 111, 163)&	0.49486\\
500&	(1, 1, 3, 5, 9, 15, 23, 35, 52, 77, 113, 166)&	0.49495\\ \hline
\end{array}
$$
\end{minipage}
}
\caption{Tuples $b$ of layered permutations $w(b)$ maximizing $u'_n(1)$ for $n=2,\ldots, 30$, and every $10$ after up to $500$.
The third column is $f(n)\coloneqq  \frac{1}{n^2} \log_2 u'_n(1)$.
The full table up to $n=750$ is available on the \texttt{arXiv} as an ancillary \texttt{CSV} file.}
\label{tab:layered_permutations}
\end{table}

\subsection{Optimal layered permutations and \texorpdfstring{\Cref{thm:gmax}}{asymptotic maximum}}
\label{app:layered_permutations_discussion}

Observe that the optimal layered permutations given in
\Cref{tab:layered_permutations} do not fall into the
description of sequences given in \Cref{thm:gmax}.
For example, for $n=700$, the numerically optimal layered permutation
has parts $b^{(700)}_{\mathrm{opt}}=(1, 3, 5, 9, 14, 21, 33, 49, 73, 107, 157, 228)$,
whereas a sequence $b^{(700)}$ from \Cref{thm:gmax}
would satisfy the bound $b_1=n-k \le 700/(2 + \sqrt{2}) \approx
205$ on its largest part.
The parts of $b^{(700)}_{\mathrm{opt}}$ are suspiciously close to
the Fibonacci sequence (see \Cref{tab:layered_permutations_Fibonacci}).
Let us explain this discrepancy.

Optimal
layered permutations can be constructed inductively via the
optimization
$$
\max_{b_1 + b_2 + \cdots = n} \Upsilon_{w(\ldots, b_2, b_1)}(1) =
\max_{k} \max_{b_2 + b_3+\cdots = k}\bigl\{ \Upsilon_{w(\ldots,b_3, b_2)}(1) \cdot F(k, n-k)
\bigr\},
$$
where $b_1 = n-k$.
Assume that $\max_{b_2 +b_3+ \cdots = k} \Upsilon_{w(\ldots,
b_3,b_2)}(1) = 2^{\binom k2}  \cdot g(k)$, where $g(k) =
2^{-o(k^2)}$. To extract the first part $b_1$, we thus need to
find
\begin{equation*}
	\max_k
	\bigl\{ 2^{\binom k2}\cdot g(k) \cdot F(k, n-k) \bigr\}
	=
	\max_k \bigl\{g(k)\cdot G(k,n-k)\bigr\},
\end{equation*}
see \eqref{eq:F_G_notation_for_layered} for the notation.
On the other hand, in the proof of \Cref{thm:gmax} (in particular, see
\Cref{lemma:fbound}),
we estimated the asymptotic behavior of
$F(k,n-k)$, and
found that the coefficient
by $n^2$ in $\log_2 F(k,n-k)$ does not depend on $k$ as
long as $k>n/\sqrt{2}$.
These two problems lead to different choices of $k$.
See also
the
illustration in \Cref{fig:F_G_optimization}.

\begin{table}[htb]
\[
\begin{array}{ccccccccccccccccc} \hline
n     & 1&1& 2& 3& 5& 8& 13& 21& 34& 55& 89& 144& 233& 377& 610& \\
k^*&&&&1& 2& 3& 5& 8& 13& 21& 34& 56& 90& 147& 237\\ \hline
\end{array}
\]
\caption{If $n$ runs over the Fibonacci numbers, the
optimal $k^*$
maximizing $G(k,n-k)$
is close to the Fibonacci number
of index $2$ less. Note that $k^*$ is not the same as $n-b_1$ in \Cref{tab:layered_permutations}
due to the presence of the lower order term $g(k)$.}
\label{tab:layered_permutations_Fibonacci}
\end{table}

\begin{figure}
    \centering
    \includegraphics[width=0.4\linewidth]{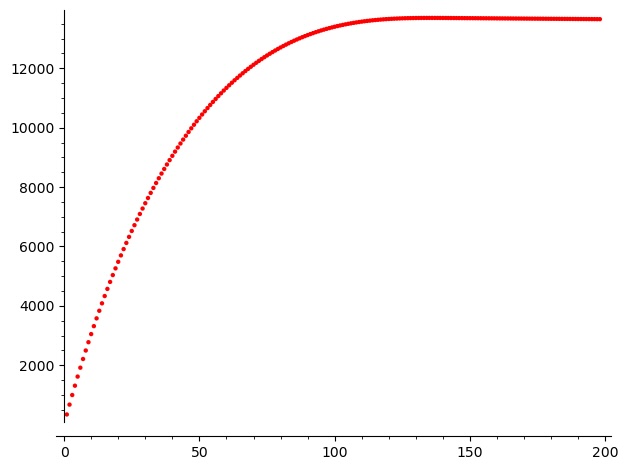}
    \qquad
    \includegraphics[width=0.4\linewidth]{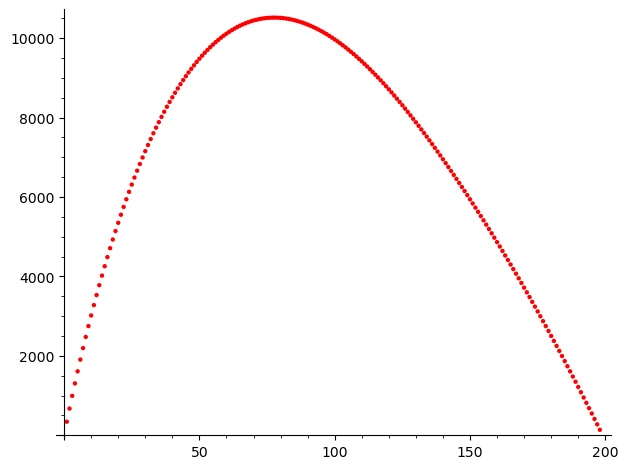}
		\caption{Plot of $\log_2 G(k,200-k)$ (\textbf{left})
			and $\log_2
			F(k,200-k)$ (\textbf{right}) for $k=1,\ldots,200$. The maxima on the figures
		occurs at $k=135$ and $k=78$, respectively.}
		\label{fig:F_G_optimization}
\end{figure}

The discrepancy between optimal layered permutations and the ones appearing in
\Cref{thm:gmax} should be of order $o(n)$ in the large $n$ limit. Indeed,
assume that $k/n<1/\sqrt{2}$. Then, as in the proof of
\Cref{lemma:fbound} (see also \Cref{rmk:colomo_pronko}),
one can show that $\log_2 F(k,n-k)\sim \alpha n^2-k^2/2-O(n)$ for some
$\alpha<\frac{1}{2}$, so
$2^{\binom k2} g(k) \ssp F(k, n-k)\sim 2^{\alpha n^2-o(n^2)}$,
which is not maximal.
We believe that the discrepancy of the part sizes
in optimal layered permutations is of order
$o(n)$ as $n\to\infty$.

\section{Limit shape and fluctuations in TASEP from contour integrals}
\label{sec:app_TASEP_asymptotics}

In this section we compute the constants
$\mathsf{c},\mathsf{v}_1,\mathsf{v}_2$
from \Cref{def:TASEP_constants}, and
show how to establish
\Cref{thm:TASEP_limit} by an asymptotic analysis of the correlation kernel.
We omit tedious but standard
estimates,
as they are
very similar to those in \cite[Section~3]{johansson2000shape}
or \cite{Petrov2017push}.

\subsection{Fredholm determinant}
\label{sub:TASEP_setup_app}

Recall that $\bar\xi_m(t)=\xi_m(t)-\xi_m(0)$ is the displacement of the $m$-th particle
at time $t$.
By \Cref{prop:TASEP_Schur_measure,prop:det_Schur},
we have
\begin{equation}
	\label{eq:lambda_m_t_Schur}
	\Prob_{\mathrm{TASEP}}(\bar\xi_{m}(t)\ge u)=
	\Prob(\lambda_m\ge u),
\end{equation}
where $\lambda_m$ is a point of the determinantal point process
$X(\lambda)=\{ \lambda_j+m-j \}_{j=1}^{m}\subset \mathbb{Z}_{\ge0}$ with the correlation kernel
\eqref{eq:kernel_Schur}:
\begin{equation}
	\label{eq:kernel_Schur_app}
	K(u_1,u_2)=\frac{1}{(2\pi\mathbf{i})^2}\oiint \frac{dz \ssp dw}{z-w}\frac{w^{u_2}}{z^{u_1+1}}
	\left( \frac{z-p}{w-p} \right)^{m}
	\left( \frac{1-w}{1-z} \right)^{t},
	\qquad
	u_1,u_2\in \mathbb{Z}_{\ge0},
\end{equation}
where the contours are positively oriented simple closed curves satisfying
$p<|w|<|z|<1$.
\begin{lemma}
	\label{lemma:Fredholm}
	The probability
	in the right-hand side of
	\eqref{eq:lambda_m_t_Schur} is equal to the Fredholm determinant
	of the kernel $K$ on $\{0,1,\ldots,u-1\}$:
	\begin{equation}
		\label{eq:Fredholm_Schur}
		\Prob(\lambda_m\ge u)=
		\det\left[ \mathrm{Id}-K \right]_{\le u-1}
		\coloneqq
		1+\sum_{\ell=1}^{\infty}\frac{(-1)^\ell}{\ell!}
		\sum_{i_1,\ldots,i_\ell=0}^{u-1}\det\left[ K(i_a,i_b) \right]_{a,b=1}^{\ell},
	\end{equation}
	where $\mathrm{Id}$ is the identity operator.
\end{lemma}
\begin{proof}
	The event $\{\lambda_m\ge u\}$ means that in the determinantal
	point process $X(\lambda)$, no particles
	are located in $\{0,1,\ldots,u-1 \}$.
	It is a well-known property of determinantal point processes that this
	probability is given by a Fredholm determinant, see e.g.
	\cite[Theorem~2]{Soshnikov2000}.
	This statement can be traced back to the inclusion--exclusion principle
	(e.g., see \cite[Appendix~A.3]{Borodin2000b}).
	Note that the sum over $\ell$ in \eqref{eq:Fredholm_Schur} is finite, as
	for $u>\ell$, the $\ell\times \ell$ determinant must have identical rows.
	Further details on Fredholm determinants may be found in
	\cite{Simon2005} or \cite{Bornemann_Fredholm2010}.
\end{proof}

Let us assume that
$m=\lfloor L\ssp \mathsf{m} \rfloor$,
$t=\lfloor L\ssp \mathsf{t} \rfloor$,
$u=\lfloor L\ssp \mathsf{u} \rfloor$,
where $L\to \infty$. By \Cref{lemma:a_priori_zero_of_limit_shape},
it suffices to consider the case $\mathsf{t}\ge\mathsf{m}/p$,
which we assume throughout the present
\Cref{sec:app_TASEP_asymptotics}.

\subsection{Steepest descent and law of large numbers}
\label{sub:steepest_descent_app}

The asymptotic analysis of $K(u_1,u_2)$ is done by
steepest descent (see, e.g., \cite[Section~3]{Okounkov2002} for an accessible introduction).
The integrand in \eqref{eq:kernel_Schur_app} has the form
\begin{equation}
	\label{eq:K_integrand}
	\frac{1}{z-w}\frac{w^{u_2}}{z^{u_1+1}}
	\left( \frac{z-p}{w-p} \right)^{m}
	\left( \frac{1-w}{1-z} \right)^{t}
	=\frac{\exp\left\{ L \bigr(
	S(z;u_1/L)
	-S(w;u_2/L)
	\right)
	+O(1)
	\bigr\}}{z(z-w)},
\end{equation}
where
\begin{equation}
	\label{eq:S_function}
	S(z;\mathsf{u})\coloneqq
	-\mathsf{u}\log z+
	\mathsf{m}\log(z-p)
	-\mathsf{t}\log(1-z).
\end{equation}
The term $O(1)$
arises from dropping integer parts in $m$ and $t$.
It is
negligible in the limit because the determinantal process
will be scaled from discrete to continuous space.

The point $\lambda_m$ is at the left edge of the determinantal process.
Therefore, to catch its asymptotic location $\mathsf{u}=\mathsf{c}(\mathsf{m},\mathsf{t})$,
we need to find a \emph{double critical point} of the function
$z\mapsto S(z;\mathsf{u})$.

\begin{remark}
	\label{rmk:double_critical_point}
	Let us explain the need for the double critical point
	in more detail.
	By \eqref{eq:determinantal_process_definition}, we have
	$K(u,u)=\Prob(u\in X(\lambda))$, which is the density
	of the random point configuration $X(\lambda)$.
	The equation $S'(z;\mathsf{u})=0$ for single critical points
	is quadratic, and its discriminant is
	a function of $p$ and our parameters $(\mathsf{m},\mathsf{t},\mathsf{u})$.
	Fox fixed $\mathsf{m}$ (viewed as a parameter of the determinantal
	process and not the scaled index on the last part $\lambda_m$),
	there are three possibilities depending on $(\mathsf{t},\mathsf{u})$:
	\begin{enumerate}[$\bullet$]
		\item If the discriminant is negative,
			then
			$K(u,u)$ converges to a value strictly between $0$ and $1$. Therefore,
			at the scaled time $\mathsf{t}$ around the scaled location
			$\mathsf{u}$ there is a random configuration of points from $X(\lambda)$
			of density strictly between zero and one.
		\item If the discriminant is positive, then
			$K(u,u)$ converges to $0$ or $1$,
			and around $\mathsf{u}$ there is either a densely packed, or an empty
			region of points from $X(\lambda)$.
		\item The case of zero discriminant (double critical points)
			is at the transition between the two cases, and
			thus
			corresponds to the edge of the random configuration $X(\lambda)$.
	\end{enumerate}
\end{remark}

Solving the double critical point equations $S'(z;\mathsf{u})=S''(z;\mathsf{u})=0$ in
$z,\mathsf{u}$, we find
\begin{equation}
	\label{eq:correct_edge_zc_u}
	z_{cr}=
	\frac{\sqrt{p\ssp \mathsf{t}}-\sqrt{\mathsf{m}}}{\sqrt{\mathsf{t}/p}-\sqrt{\mathsf{m}}}
	,\qquad
	\mathsf{u}_{cr}=
	\frac{(\sqrt{p\ssp \mathsf{t}}-\sqrt{\mathsf{m}})^2}{1-p}.
\end{equation}
Here out of two solutions to quadratic equations, we picked
the one with the smaller $\mathsf{u}_{cr}$ which corresponds to the left edge of $X(\lambda)$.
Note that the solution $\mathsf{u}_{cr}$ in \eqref{eq:correct_edge_zc_u} is the same as
the constant $\mathsf{c}(\mathsf{m},\mathsf{t})$
from \Cref{def:TASEP_constants}, which is not a coincidence.

Let us deform the integration contours in
the kernel \eqref{eq:kernel_Schur_app} to pass near the double critical point
$z_{cr}$. This point lies inside the interval $(0,p)$, while
the original contours satisfy $p<|w|<|z|<1$.
Due to high multiplicity poles in the integrand,
the $w$ contour should not be deformed through $p$, and
the $z$ contour should not be deformed through $0$ and $1$.

We deform the contours as follows. First,
deform $w$ such that the new $w$ contour goes around $p$ and passes to the right of $z_{cr}$.
Then, deform $z$ through $w$, such that
the new $z$ contour goes around $0$ and passes to the left of $z_{cr}$.
Dragging $z$ through $w$ picks up a residue at $z=w$, which needs to be integrated over the new $w$ contour.
However, taking the residue at $z=w$ eliminated the pole at $w=p$, so the
new single integral is zero.
This implies that the contour deformation does not produce extra terms in the kernel.
See \Cref{fig:contour_deformation}
for an illustration.

\begin{figure}[htpb]
	\centering
	\includegraphics[width=0.75\textwidth]{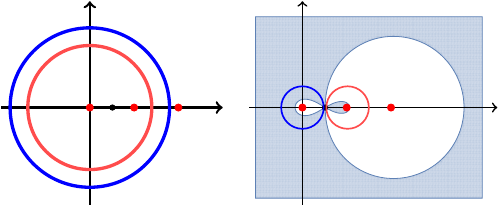}
	\caption{\textbf{Left}: Original contours in $K$ \eqref{eq:kernel_Schur_app}.
	\textbf{Right}: Deformed contours passing through the critical point. All contours are positively oriented.
	\\
	The
	highlighted region in the right plot is $\Re \bigl(S(z;\mathsf{u}_{cr})- S(z_{cr};\mathsf{u}_{cr})\bigr)<0$.
	The new $z$ and $w$ contour must lie inside (resp., outside) of this region
	for the contribution outside a neighborhood of $z_{cr}$ to vanish in the limit.
	In this example, $p=0.5$, $\mathsf{m}=1$, $\mathsf{t}=5$, so $z_{cr}\approx 0.27$ and $\mathsf{u}_{cr}\approx 0.68$.
	Large (red) dots show the points $0$, $p$, and $1$; the smaller dot is $z_{cr}$.}
	\label{fig:contour_deformation}
\end{figure}

To arrive at the law of large numbers $\lambda_m/L \to \mathsf{c}(\mathsf{m},\mathsf{t})$
(recall that this is the same as $\mathsf{u}_{cr}$),
it remains to show that changing $\mathsf{c}$
by $\varepsilon$ or $(-\varepsilon)$ makes the Fredholm determinant \eqref{eq:Fredholm_Schur} go to $0$ or $1$,
respectively. To show convergence to~$1$, it suffices to
estimate $|K(u_1,u_2)|\le e^{-\delta L}$ for
$u_1,u_2\le L\ssp (\mathsf{c}-\varepsilon)$, which follows from the steepest descent analysis.
To show convergence to $0$,
it suffices to look at the density function $K(u,u)$ for $u$ around $L(\mathsf{c}+\varepsilon)$.
In the limit, this density
is strictly positive, and so the event
$\lambda_m\ge L\ssp (\mathsf{c}+\varepsilon)$ has probability tending to zero.
We omit further details.

\subsection{Asymptotic fluctuations}
\label{sub:asymptotic_fluctuations_app}

Consider the probabilities of fluctuations as in \Cref{thm:TASEP_limit}:
\begin{equation}
	\label{eq:TASEP_fluctuations_app_initial_formula}
	\Prob_{\mathrm{TASEP}}
	\left( \bar\xi_{\lfloor L\ssp \mathsf{m}- L^{1/3}\ssp \alpha\ssp \mathsf{v}_1\rfloor}
		(\lfloor L\ssp \mathsf{t} \rfloor )\ge
	L\ssp \mathsf{c} -
	L^{1/3}\ssp \beta\ssp \mathsf{v}_2
\right)= \det\left[
		\mathrm{Id}
		-
		K^{L\ssp \mathsf{t},
		L\ssp \mathsf{m} - L^{1/3}\ssp \alpha\ssp \mathsf{v}_1}
	\right]_{\le L\ssp \mathsf{c}-
	L^{1/3}\ssp \beta\ssp \mathsf{v}_2 }+o(1).
\end{equation}
Assume that $\mathsf{m}$ and $\mathsf{t}$ are fixed
(recall that $\mathsf{c},\mathsf{v}_1,\mathsf{v}_2$ depend on them).
In the right-hand side of
\eqref{eq:TASEP_fluctuations_app_initial_formula},
we dropped integer parts as
they do not affect the limit to continuous space.

One can show that in the
series
\eqref{eq:Fredholm_Schur}
for the
Fredholm determinant \eqref{eq:TASEP_fluctuations_app_initial_formula},
terms where
$i_a< L\ssp \mathsf{c}-s L^{1/3}$ for at least one $a=1,\ldots,\ell $
(and $s$ sufficiently large but finite), are negligible in the limit.
Thus,
it suffices to consider
$
\tilde K(\tilde u_1,\tilde u_2)\coloneqq
K^{L\ssp \mathsf{t},
L\ssp \mathsf{m} - L^{1/3}\ssp \alpha\ssp \mathsf{v}_1}
(L\ssp \mathsf{c}+\tilde u_1 L^{1/3},L\ssp \mathsf{c}+\tilde u_2 L^{1/3})$,
where $\tilde u_1,\tilde u_2\in \mathbb{R}$.
This computation will also produce
the normalization
constants $\mathsf{v}_1$ and $\mathsf{v}_2$
for the Tracy--Widom GUE fluctuations.

In $\tilde K$,
deform the contours to pass near the double critical point
$z_{cr}=z_{cr}(\mathsf{m},\mathsf{t})$ given by \eqref{eq:correct_edge_zc_u}.
The contribution outside a small neighborhood of $z_{cr}$ vanishes in the limit.
In the neighborhood,
make a change of variables
\begin{equation*}
	z=z_{cr}+L^{-1/3}\frac{\tilde z}{\varrho},
	\qquad
	w=z_{cr}+L^{-1/3}\frac{\tilde w}{\varrho},
\end{equation*}
where $\varrho>0$ is to be determined.
The new contour $\tilde z$ goes from
$e^{-\frac{2\pi\mathbf{i}}{3}}\infty$ to $-1$ to
$e^{\frac{2\pi\mathbf{i}}{3}}\infty$,
and $\tilde w$ goes from
$e^{-\frac{\pi\mathbf{i}}{3}}\infty$ to $0$ to
$e^{\frac{\pi\mathbf{i}}{3}}\infty$.
The function
\eqref{eq:S_function}
is Taylor expanded in $L$ as
\begin{equation*}
	\begin{split}
		&L\left( S(z;\mathsf{c}+L^{-2/3}\tilde u_1 )-S(z_{cr};\mathsf{c}+L^{-2/3}\tilde u_1 ) \right)
		\\&\hspace{80pt}=
		-\frac{\tilde z^3}{3\varrho^3}
		\frac{p^{3/2}}{\sqrt{\mathsf{m}\ssp \mathsf{t}}}
		\frac{(\sqrt{\mathsf{t}/p}-\sqrt{\mathsf{m}})^4}{(1-p)^3 z_{cr}}
		+
		\frac{\tilde z}{\varrho\ssp z_{cr}}
		\left( -\tilde u_1
			+
			\frac{\alpha\ssp \mathsf{v}_1(\sqrt{p\ssp \mathsf{t}/\mathsf{m}}-1)}{1-p}
		\right).
	\end{split}
\end{equation*}
Let us pick $\varrho>0$ such that the coefficient
by $\tilde z^3$ is $(-1/3)$,
that is,
\begin{equation*}
	\varrho=
	\frac{\sqrt{p}}{(1-p)\ssp z_{cr}}
	\frac{(\sqrt{p\mathsf{t}}-\sqrt{\mathsf{m}})^{2/3}
	(\sqrt{\mathsf{t}/p}-\sqrt{\mathsf{m}})^{2/3}}{(\mathsf{m}\mathsf{t})^{1/6}}.
\end{equation*}
We also have
in the integrand:
\begin{equation*}
	\frac{dz \ssp dw}{z(z-w)}\sim -\frac{L^{-1/3}}{\varrho z_{cr} }
	\frac{d\tilde z \ssp d\tilde w}{\tilde w- \tilde z}.
\end{equation*}
Adding the difference
$S(z_{cr};\mathsf{c}+L^{-2/3}\tilde u_2 )-S(z_{cr};\mathsf{c}+L^{-2/3}\tilde u_1 )$
inside the exponent in the integrand \eqref{eq:K_integrand}
does not change the correlation kernel, since it is a
gauge transformation of the form
$\tilde K(\tilde u_1,\tilde u_2)\mapsto
\tilde K(\tilde u_1,\tilde u_2)\frac{f(\tilde u_2)}{f(\tilde u_1)}$,
where $f(\cdot)$ is a nonvanishing function.
Thus, we have
\begin{equation}
	\label{eq:K_Airy_asymptotics}
	\tilde K(\tilde u_1,\tilde u_2)\frac{f(\tilde u_2)}{f(\tilde u_1)}
	\sim
	-\frac{L^{-1/3}}{\varrho z_{cr} }
	\int_{e^{-\frac{2\pi\mathbf{i}}{3}}\infty}^{e^{\frac{2\pi\mathbf{i}}{3}}\infty}
	d\tilde z
	\int_{e^{\frac{\pi\mathbf{i}}{3}}\infty}^{e^{-\frac{\pi\mathbf{i}}{3}}\infty}
	d\tilde w\,
	\frac{\exp\left[ \frac{\tilde w^3}{3}
			-
			\frac{\tilde z^3}{3}
			+
			\tilde z U_1
			-
			\tilde w U_2
	\right]}{\tilde w-\tilde z},
\end{equation}
where
\begin{equation}
	\label{eq:big_U_as_function_of_small_u}
	U_{1,2}\coloneqq
	\frac{1}{\varrho\ssp z_{cr}}
	\left(
	-\tilde u_{1,2}
		+
		\frac{\alpha\ssp \mathsf{v}_1(\sqrt{p\ssp \mathsf{t}/\mathsf{m}}-1)}{1-p}
	\right)
	.
\end{equation}
The integral in \eqref{eq:K_Airy_asymptotics} together with the minus sign from the prefactor
is equal to the celebrated Airy$_2$ kernel
$\mathsf{A}\bigl(U_1,U_2\bigr)$
(e.g., see \cite[(40)]{Okounkov2002}).
The Tracy--Widom GUE cumulative distribution function
is the Fredholm determinant of the Airy$_2$ kernel, that is,
$F_2(r)=\det[\mathrm{Id}-\mathsf{A}]_{\ge r}$, $r\in \mathbb{R}$.

The additional prefactor $(L^{1/3}\varrho\ssp z_{cr})^{-1}$
in \eqref{eq:K_Airy_asymptotics}
combined with the sums in
the pre-limit Fredholm determinant \eqref{eq:Fredholm_Schur}
turns them into Riemann integrals of the Airy$_2$ kernel.
The normalization by
$\varrho z_{cr}$
is absorbed by
a change of variables:
\begin{equation*}
	\frac{1}{L^{1/3}\varrho \ssp z_{cr}}
	\sum_{i=\lfloor L^{1/3}\tilde u \rfloor \le -L^{1/3}\beta\ssp \mathsf{v}_2}
	g(\tilde u)
	\sim
	\frac{1}{\varrho \ssp z_{cr}}
	\int_{-\infty}^{-\beta \ssp \mathsf{v}_2}
	g(\tilde u)\ssp d\tilde u=
	\int_{
	\frac{1}{\varrho\ssp z_{cr}}
	\left(
		\beta \ssp \mathsf{v}_2
		+
		\frac{\alpha\ssp \mathsf{v}_1(\sqrt{p\ssp \mathsf{t}/\mathsf{m}}-1)}{1-p}
	\right)
	}^{+\infty}
	\mathsf{g}(U)\ssp dU.
\end{equation*}
Here we incorporated the bound
$L\ssp \mathsf{c}- L^{1/3}\ssp \beta\ssp \mathsf{v}_2$ from \eqref{eq:TASEP_fluctuations_app_initial_formula}
into the bound for $i$, and shifted
all indices by $L\ssp \mathsf{c}$.
The function $g(\tilde u)$ represents
terms like $K(\cdot, i_a)K(i_a,\cdot)$ in the
Fredholm determinant \eqref{eq:Fredholm_Schur}
arising when expanding the individual $\ell\times \ell$
determinants as sums over permutations.
Here $\tilde u$ and $i_a$ are related as
$i_a=\lfloor L^{1/3}\tilde u \rfloor$.
We also set $\mathsf{g}(U)=\mathsf{g}(\tilde u)$.

We see that
the following
convergence of Fredholm determinants holds:
\begin{equation}
	\label{eq:Airy_limit}
	\lim_{L\to \infty}
	\det\left[
		\mathrm{Id}
		-
		K^{L\ssp \mathsf{t},
		L\ssp \mathsf{m} - L^{1/3}\ssp \alpha\ssp \mathsf{v}}
	\right]_{\le L\ssp \mathsf{c} }
	=
	\det\left[
		\mathrm{Id}
		-
		\mathsf{A}
	\right]_{\ge r},
\end{equation}
where
\begin{equation*}
	r=
	\frac{1}{\varrho\ssp z_{cr}}
	\left(
		\beta \ssp \mathsf{v}_2
		+
		\frac{\alpha\ssp \mathsf{v}_1(\sqrt{p\ssp \mathsf{t}/\mathsf{m}}-1)}{1-p}
	\right).
\end{equation*}
We see that setting
\begin{equation}
	\label{eq:v_definition}
	\begin{split}
		\mathsf{v}_1=
		\mathsf{v}_1
		(\mathsf{m},\mathsf{t})
		&\coloneqq
		\frac{\varrho\ssp z_{cr}\ssp(1-p)}{\sqrt{p\ssp \mathsf{t}/\mathsf{m}}-1}
		=
		\frac{\sqrt{p}\ssp\mathsf{m}^{1/3}}{\mathsf{t}^{1/6}}
		\frac{(\sqrt{\mathsf{t}/p}-\sqrt{\mathsf{m}})^{2/3}}
		{(\sqrt{p\mathsf{t}}-\sqrt{\mathsf{m}})^{1/3}}
		;\\
		\mathsf{v}_2=
		\mathsf{v}_2
		(\mathsf{m},\mathsf{t})
		&\coloneqq
		\varrho\ssp z_{cr}
		=
		\frac{\sqrt{p}}{(1-p)}
		\frac{(\sqrt{p\mathsf{t}}-\sqrt{\mathsf{m}})^{2/3}
		(\sqrt{\mathsf{t}/p}-\sqrt{\mathsf{m}})^{2/3}}{(\mathsf{m}\mathsf{t})^{1/6}}
	\end{split}
\end{equation}
turns the right-hand side of \eqref{eq:Airy_limit} into the Tracy--Widom GUE distribution $F_2(\alpha+\beta)$.
This completes the proof of \Cref{thm:TASEP_limit}.

\bibliographystyle{alpha}
\bibliography{bib_pipes}

\medskip

\textsc{A. H. Morales, Université du Québec à Montréal, Montreal, QC, Canada and University of Massachusetts Amherst, Amherst, MA, USA}

E-mail: \texttt{morales\_borrero.alejandro@uqam.ca}

\medskip

\textsc{G. Panova, University of Southern California, Los Angeles, CA, USA}

E-mail: \texttt{gpanova@usc.edu}

\medskip

\textsc{L. Petrov, University of Virginia, Charlottesville, VA, USA}

E-mail: \texttt{lenia.petrov@gmail.com}

\medskip

\textsc{D. Yeliussizov, Kazakh--British Technical University, Almaty, Kazakhstan}

E-mail: \texttt{yeldamir@gmail.com}

\end{document}